\documentclass{article}
\usepackage{fontsize}
\usepackage[hang,small,bf]{caption}
\RequirePackage{amsthm,amsmath,graphicx,enumitem,comment}

\usepackage[utf8]{inputenc}
\usepackage[ margin=1in]{geometry}

\DeclareMathSizes{12}{30}{16}{12}

\usepackage{appendix}

\usepackage{amsfonts}
\usepackage{amsmath}
\usepackage{framed}
\usepackage[ruled,vlined]{algorithm2e}
\usepackage{algorithmic}
\usepackage{upgreek}
\usepackage{cases}

\usepackage{verbatim}
\usepackage{listings}
\usepackage{fancyvrb}

\allowdisplaybreaks

\usepackage[english]{babel}
\usepackage{amsthm}
\usepackage{framed}
\usepackage{tikz}
\usepackage{bbold}

\usepackage{array}

\usepackage{amssymb}
\usepackage{mathtools}
\usepackage{commath}
\usepackage{graphicx} 
\usepackage{multirow}

\newtheorem{theorem}{Theorem}
\newtheorem{lemma}{Lemma}
\newtheorem{proposition}{Proposition}
\newtheorem{corollary}{Corollary}

\newtheorem{remark}{Remark}

\newcommand{\R}{\mathbb{R}}
\newcommand{\N}{\mathbb{N}}
\newcommand{\E}{\mathbb{E}}

\newcommand{\CL}[1]{C_{#1}}
\newcommand{\Cbar}{\bar C}
\newcommand{\Cnu}{c_{\nu}}

\newcommand{\Bernoulli}[1]{}

\usepackage{amsfonts}
\usepackage{amsmath}
\usepackage{framed}
\usepackage[ruled,vlined]{algorithm2e}
\usepackage{algorithmic}
\usepackage{upgreek}
\usepackage{cases}
\usepackage{verbatim}
\usepackage{listings}
\usepackage{fancyvrb}

\allowdisplaybreaks

\usepackage{amsthm}
\usepackage{framed}
\usepackage{bbold}
\usepackage{array}

\usepackage{amssymb}
\usepackage{mathtools}
\usepackage{commath}
\usepackage{graphicx} 
\usepackage{multirow}
\usepackage{anyfontsize}

\newcommand{\revision}[1]{#1}

\title{Sparse Signal Detection in Heteroscedastic  Gaussian Sequence Models: Sharp Minimax Rates}

\usepackage{authblk}
\author[1]{Julien Chhor}
\author[1]{Rajarshi Mukherjee}
\author[1]{Subhabrata Sen}
\affil[1]{Harvard University}

\date{}

\begin{document}

\maketitle
\vspace{-12mm}

\begin{center}
    Contact: jchhor@hsph.harvard.edu, ram521@mail.harvard.edu, subhabratasen@fas.harvard.edu
\end{center}

\begin{abstract}
    Given a heterogeneous Gaussian sequence model with unknown mean $\theta \in \R^d$ and known covariance matrix $\Sigma = \operatorname{diag}(\sigma_1^2,\dots, \sigma_d^2)$, we study the signal detection problem against sparse alternatives, for known sparsity~$s$. 
    Namely, we characterize how large $\epsilon^*>0$ should be, in order to distinguish with high probability the null hypothesis $\theta=0$ from the alternative composed of $s$-sparse vectors in $\R^d$, separated from $0$ in $L^t$ norm ($t \in [1,\infty]$) by at least~$\epsilon^*$.
    We find non-asymptotic minimax upper and lower bounds over the minimax separation radius $\epsilon^*$ and prove that they are always matching. 
    We also derive the corresponding minimax tests achieving these bounds. 
    Our results reveal new phase transitions regarding the behavior of $\epsilon^*$ with respect to the level of sparsity, to the $L^t$ metric, and to the heteroscedasticity profile of $\Sigma$. In the case of the Euclidean (i.e. $L^2$) separation, we bridge the remaining gaps in the literature.
\end{abstract}
\section{Introduction}
Global testing against structured alternatives is a canonical problem in modern statistics. Under the minimax hypothesis testing framework formulated by \cite{burnashev1979minimax,ingster1982minimax,ingster1987minimax,ingster2003nonparametric}, the central object of interest is the \emph{minimax separation radius}--- the smallest separation between the null and alternative hypotheses so that consistent detection is possible.  Motivated by questions arising from genomics, communications, social sciences etc., diverse global testing problems have been rigorously investigated, and their associated separation radii have been characterized -- see e.g. \cite{burnashev1979minimax,ingster1982minimax,ingster1987minimax,ingster2003nonparametric,donoho2004higher,hall2010innovated,arias2011global,ingster2010detection,tony2011optimal} and references therein. Further papers have investigated the effect of nuisance parameters (e.g. effect of dependence parameters on high-dimensional mean testing problems)  on the separation radius \cite{hall2010innovated,kotekal2021minimax,mukherjee2018global}. Responding to this line of research, in this paper, we  provide a complete picture of the precise dependence on the minimax rate of testing for the mean vector in  the Gaussian Sequence Model \cite{ingster2003nonparametric} as a function of a heteroscedastic variance profile.

The Gaussian sequence model furnishes arguably the most canonical setup to explore the fundamental thresholds for global testing. Starting with the original works of \cite{burnashev1979minimax,ingster1982minimax,ingster1987minimax,ingster2003nonparametric}, the global testing problem has been carefully studied in this context under diverse structured alternatives, and under many different notions of separation. In the modern age of big-data, sparse alternatives are particularly important. The detection thresholds for Gaussian sequence model under sparse alternatives has been derived in \cite{ingster2003nonparametric,donoho2004higher,hall2010innovated,arias2011global,ingster2010detection}. However, most of the existing literature focuses on the homoscedastic sequence model, where the error variances are all equal. In this work, we go beyond the homoscedastic case, and derive the separation radius for heteroscedastic Gaussian sequence models under sparse alternatives.

\vspace{2mm}
Hereafter, we write $[k] = \{1,\dots,k\}$ for any positive integer $k$. 
For any positive integer $d$, any vector $u \in \R^d$, and any $t >0$, we let $\|u\|_t = \left(\sum_{j=1}^d |u_j|^t\right)^{1/t}$, $\|u\|_\infty = \max_{i=1}^d |u_i|$ and $\|u\|_0 = \operatorname{Card} \left\{j\in [d] \,\big|\, u_j \neq 0\right\}$. 

\subsection{Model}
Fix $d$ real numbers $\sigma_1 \geq \dots \geq \sigma_d >0$, which we assume to be known throughout the paper. 
For some unknown $\theta \in \R^d$, suppose we observe $X = (X_1,\dots, X_d)$ where $\forall j \in [d]: X_j = \theta_j + \sigma_j \xi_j$. 
Here, $\xi_j \sim \mathcal{N}(0, \sigma_j^2)$ are independent and distributed as normal random variables with mean $0 \in \R^d$ and variance $\sigma_j^2$.
We shall often denote this observation scheme as $X \sim \mathcal{N}(\theta, \Sigma)$ where $\theta \in \R^d$ and $\Sigma = \operatorname{diag}(\sigma_1^2, \dots, \sigma_d^2)$. 
For some $t \in [1,\infty]$, some known sparsity $s \in [d]$, and some $\epsilon>0$, we define $\Theta(\epsilon,s,t) = \big\{\theta \in \R^d ~\big|~ \|\theta\|_t \geq \epsilon \text{ and }
    \|\theta\|_0 \leq s\big\}$ and consider the following testing problem
\begin{align}
    H_0: \theta = 0 \quad \text{ against } \quad H_1: \theta \in \Theta(\epsilon,s,t).\label{eq:testing_pb}
\end{align}

Here, the parameter $\epsilon>0$ induces a separation between the two hypotheses, and our goal is to characterize how large $\epsilon$ should be for the testing problem~\eqref{eq:testing_pb} to be feasible in a sense defined below. 
The $L^2$ separation (i.e. the $t=2$ case) is by far the most studied case in the literature~\cite{baraud2002non,ingster2010detection,laurent2012non,collier2017minimax,kotekal2021minimax,liu2021minimax}. 
In this paper, we go beyond the $L^2$ case and study general separations in $L^t$ norm for any $t  \in [1, \infty]$. 

To fix ideas, we set up a few definitions and define some notation. A \textit{test} $\psi$ is defined as a measurable function of the data $X$ taking its values in $\{0,1\}$:
$$ \psi: \R^d \longrightarrow \{0,1\}.$$

In the minimax paradigm, one measures the quality of any test by its \textit{risk}.
We assume that the matrix $\Sigma$ is fixed and known, and for any $\theta \in \R^d$, we let $\mathbb P_\theta$ denote the probability measure associated with the law of $X \sim \mathcal N(\theta, \Sigma)$. 
Then the \textit{risk} of test $\psi$ is defined as the sum of its type I and type II error probabilities:
\begin{align}
    R(\psi, \epsilon, s,t,\Sigma):= \mathbb P_{0} \left(\psi = 1\right) + \sup \left\{\mathbb P_{\theta} \left(\psi = 0\right)~ \Big| ~ \theta \in \Theta(\epsilon,s,t)\right\}.\label{eq:risk}
\end{align}

Note that in equation~\eqref{eq:risk}, the notation $\mathbb P_\theta$ denotes the probability distribution of $\mathcal{N}(\theta, \Sigma)$ and implicitly depends on $\Sigma$.
The \textit{minimax risk} represents the infimal risk among all possible tests, and can be understood as the risk of the best test if it exists:
\begin{align}
    R^*(\epsilon,s,t,\Sigma) := \inf_{\psi} R(\psi,\epsilon,s,t,\Sigma).
\end{align}

In the above definition, the infimum is taken over all tests $\psi$.
Note that if $R^*(\epsilon,s,t,\Sigma)=1$, then random guessing is optimal (the test $\overline \psi$ that accepts $H_0$ with probability $1/2$ independently of the data $X$ achieves a risk equal to $1$).
Therefore, we say that the testing problem~\eqref{eq:testing_pb} is feasible if, for some tolerance $\eta \in (0,1)$ fixed in advance, we have $R^*(\epsilon,s,t,\Sigma)\leq \eta<1$.
The sparsity $s$, the metric-inducing norm $\|\cdot\|_t$ and the covariance matrix $\Sigma$ being fixed, the difficulty of this testing problem is entirely characterized by the separation parameter $\epsilon>0$. 
Noting that $R^*(\epsilon,s,t,\Sigma)$ decreases with respect to $\epsilon$, our goal is therefore to determine the smallest value of $\epsilon$ ensuring feasibility of Problem~\eqref{eq:testing_pb}. 
This value is referred to as the \textit{minimax separation radius}, defined as
\begin{align}
    \epsilon^*(s,t,\Sigma) = \inf\left\{\epsilon>0 \;\big|\; R^*(\epsilon,s,t,\Sigma) \leq \eta\right\}.
\end{align}

Note that we drop the dependency of $\epsilon^*(s,t,\Sigma)$ on $\eta$ as this parameter is assumed to be a fixed constant throughout the paper. 
Our goal is to characterize the minimax separation radius $\epsilon^*(s,t,\Sigma)$ up to multiplicative constants depending only on $\eta$ and $t$.
\vspace{2mm}

\subsection{Motivation and background} 

\revision{
One specific example which motivates our mathematical framework pertains to large-scale genetic data analysis where one aims to test for the effects of a collection of genetic variants on a particular disease phenotype \cite{kraft2009genetic,mccarthy2008genome,manolio2009finding,visscher2012five}.
To do so, a common procedure pertains to developing a  test statistic for each individual genetic variant by computing a z-score -- that is centered if the genetic variant has no influence \cite{laird2011fundamentals} and has a mean that reflects the strength of the effect of the genetic variant on the phenotype otherwise. Moreover one also assumes the operates under approximate normality of the Z-scores \cite{visscher2012five,barnett2017generalized,sun2020genetic,zhang2022general} and thereby they can be thought to jointly represent an approximate Gaussian sequence model.
In this regard, since different genetic variants occur under different minor allele frequencies \cite{hartl1997principles}, the standard errors of these Z-scores are naturally heteroscedastic. Moreover, under careful LD-pruning one often also assumes almost independence of the resulting Z-scores \cite{hartl1997principles}.
Finally, only a few out of many genetic variants are often thought to be related to the phenotype -- and thereby implying sparsity on the mean vector of the Z-scores \cite{visscher2012five}. 
In terms of the statistical question of interest, given the possibility of a large number of heteroscedastic genetic variants,  one might not have enough data to accurately identify the variants responsible for the disease, if any \cite{lee2014rare,li2008methods}. Therefore, an investigator often is only interested in detecting whether genetic variants can explain some variability of the disease phenotype at all \cite{lee2014rare,barnett2017generalized}. This naturally motivates us to explore the setting~\eqref{eq:testing_pb} where our goal is to determine the minimum signal strength that allows for consistent detection. 
\vspace{2mm}
}

In terms of theoretical motivations, the aim of this paper is two-fold:
\vspace{-2mm}
\begin{itemize}
    \item To go beyond isotropic noise and study the interplay between sparsity and heteroscedasticity for testing~\eqref{eq:testing_pb}.
    \vspace{0mm}
    \item To go beyond the Euclidean separation ($t=2$), by providing a complete overview of the results for the $L^t$ separation, $t \in [1,\infty]$.
\end{itemize} 

We explain here why these two goals are natural and challenging.

\vspace{1mm}

{\textit{Heteroscedasticity}:} 
Here we discuss a another class of motivating examples in the setting of statistical inverse problems \cite{laurent2012non}, as described below,   where it is natural to consider heteroscedastic variance profiles. 
    Specifically, consider a Hilbert space $H$ with inner product $(\cdot,\cdot)$, and suppose that $T$ denotes a known linear operator on $H$. Assume moreover that, for an unknown element $f \in H$, our observations $Y$ from the following model
    \begin{align}
        Y(g) = (Tf,g) + \sigma \varepsilon(g), \quad \forall g \in H,  \nonumber 
    \end{align}
    where $\varepsilon(g)$ is a centered Gaussian with variance $(g,g)$. 
    If the operator $T$ is compact, there exist  two orthonormal bases of $H$, denoted as $(\psi_j)_j$ and $(\phi_j)_j$, for which the following singular value decomposition holds $T \phi_j = \lambda_j \psi_j$ and $T^* \psi_j = \lambda_j \phi_j$, for any $j\in\N$, where $T^*$ is the adjoint of $T$. 
    In particular, the observations $(Y(\psi_j)_{j\in\N})$ are independent and distributed as $Y(\psi_j) \sim \mathcal{N}(\lambda_j \theta_j, \sigma^2)$, where $\theta_j = (f, \phi_j)$ for any $j\in\N$. This is clearly equivalent to the heteroscedastic Gaussian sequence model introduced above, with $\sigma_j^2 = \sigma^2/\lambda_j^2$. We note that under an appropriate choice of basis $\phi_j$, the representation of the function $f$ is often \emph{sparse}---we refer the interested reader to \cite{johnstone2002function,laurent2012non} for an in-depth discussion of this phenomenon. 
    \vspace{1mm}

    \revision{Heteroscedasticity also arises naturally in some applications in genetics, where the simplification to the Gaussian sequence model comes from the computation of summary statistics. 
    Each one may have a different standard deviation, which is then treated as a known variance parameter in later analysis (see~\cite{stephens2017false} for instance).}
    \vspace{1mm}

Exploring heteroscedasticity turns out to be significantly more complex than testing with isotropic noise, even for $t=2$, as the transition phenomena depend on intricate interactions between the sparsity level and the heteroscedasticity profile of $\Sigma$. 

\vspace{2mm}


{\textit{$L^t$ separations}:} The geometry induced by the $\|\cdot\|_t$ metric changes which signals are easier to detect. 
Since we aim at detecting signals $\theta$ for which $\|\theta\|_t$ is large, the $L^t$ norms for large $t$ are more adapted to detect imbalanced signals, such as, for instance, those with one dominating coordinate. 
The extreme case of the $L^\infty$ norm can be used if we aim at detecting if the signal has at least \textit{one} non-zero coordinate at all.
In contrast, the $L^1$ norm enables us to better detect signals with numerous tiny coordinates, but large total mass. 
\vspace{2mm}

We motivate in more detail the case of the $L^1$, $L^2$ and $L^\infty$ norms, which stand out as important particular cases of our results and could have implications in related settings. 
Another goal of the present paper is also to thoroughly study the interpolation between those three cases of interest.
\vspace{-1mm}
\begin{enumerate}
    \item \revision{The $L^1$ metric has an important connection with the total variation distance in the context of discrete (i.e. multinomial) distributions.    Namely, if $p$ and $q$ are two probability vectors over the space $\{1,\dots, d\}$, then the total variation between the probability measures induced by $p$ and $q$ is equal to $\|p-q\|_1/2$.} 
    Considering that multinomial data are fundamentally heteroscedastic, our results could find applications of independent interest in the important field of distribution learning and testing, where sparse alternatives have recently received much attention~\cite{bhattacharya2021sparse,donoho2022higher,kipnis2022higher,kipnis2021unification,kipnis2021two}. 
    We provide a more detailed discussion about our results' implications for multinomial testing in Section~\ref{subsec:mult_testing}.
    It turns out that our results in $L^1$ separation are not direct analogs of the case $t=2$, and instead exhibit much more intricate phase transitions. 
    \vspace{0mm}
    \item The $L^2$ or Euclidean separation has a natural geometric interpretation, and its smoothness properties make it an ideal choice for studying signal detection. 
    As such, it constitutes a canonical case to compare results with existing literature. 
    In Euclidean separation, signal detection under general covariance matrix is well understood~\cite{ingster2003nonparametric,laurent2012non}. 
    Signal detection under sparsity and isotropic noise is also well understood~\cite{collier2017minimax}. 
    However, the interaction of the two constraints raises important challenges and is just starting to be explored (see Subsection~\ref{subsec:prior_results}).
    \vspace{0mm}
    \item It turns out that the case of the $L^\infty$ has occurred naturally and extensively in the literature of hypothesis testing under sparse alternatives \cite{cai2021statistical,cai2013two,cai2014two,cai2014high,cao2018two,chang2017simulation,li2021novel,xia2018two,xue2020distribution,yu2022power,zhang2017simultaneous}.   Indeed, for homoscedastic variance profiles it is easy to see that testing in $L^\infty$-separation is naturally equivalent to testing with a sparsity of $s=1$. Our results supplement this philosophy for heteroscedastic variance profiles and complete the picture of hypothesis testing against sparse alternatives in Gaussian sequence models under any $L_r$-norm.  
\end{enumerate}

Going beyond the $L^2$ case presents numerous technical challenges. 
For instance, when $t \in [1,2)$, the functional $\|\cdot\|_t$ is not twice continuously differentiable, which raises complications for constructing chi-square type test statistics, which are based on degree $2$ polynomials in the coordinates $X_j$'s. 
Similar difficulties for estimating non-smooth functionals have been highlighted in~\cite{cai2011testing,han2020estimation,lepski1999estimation}.
We further discuss the technical challenges raised by the $L^t$ separation in the Sections below.

\subsection{Prior results and contributions}\label{subsec:prior_results}
\textit{Prior results:} As remarked above, global null testing for the Gaussian sequence model has been extensively studied, for various notions of separation and alternatives. 
In the Gaussian mean model, minimax testing under sparse alternatives was pioneered by~\cite{ingster1997some,donoho2004higher} in the asymptotic where $s = d^{\beta}$, for $\beta \in (0,1)$. 
Nonasymptotic rates were first given in~\cite{baraud2002non}, matching up to an extra logarithmic gap in the upper bound, which was later closed by~\cite{collier2017minimax}.  
However, all of the above papers only restrict to isotropic noise. 
It turns out that non-isotropic Gaussian noise still represents a challenge in sparse signal detection. 
Important attempts to go beyond isotropic noise include~\cite{kotekal2021minimax}, dealing with correlated noise structures, and more closely, \cite{laurent2012non} considering the same heteroscedastic sequence model (Problem~\ref{eq:testing_pb}) studied in this paper, but only with Euclidean separation ($t=2$). 
Leveraging techniques developed in~\cite{baraud2002non}, they obtained non-asymptotic 
upper and lower bounds on the minimax separation radius, which did not match for specific profiles of the covariance matrix $\Sigma$. 

\vspace{2mm}

\textit{Contributions:} In this paper, we provide a complete understanding of sparse signal detection with diagonal noise covariance matrix and general $L^t$ separation, $t\in[1,\infty]$. 
We derive upper and lower bounds for the minimax separation radius that are always matching, and explicitly construct the corresponding minimax tests. 
All of our results are non-asymptotic. 
We uncover new interplays between sparsity, the $L^t$ metric, and to the heteroscedasticity profile of $\Sigma$, and thoroughly study the corresponding phase transitions.
To the best of our knowledge, the matching upper and lower bounds for diagonal covariance had not been established in the literature, even for the $L^2$ separation.


\revision{To obtain some insight into our results, we will show that the minimax separation radius $\epsilon^*(s,t,\Sigma)$ can be derived by solving the following optimization problem:
\begin{align*}
    \epsilon^*(s,t,\Sigma)^t \asymp \max_{\gamma, \pi} ~~ \sum\limits_{j=1}^d \gamma_j^t \pi_j ~~ \text{ s.t. } ~~ \begin{cases}\sum_{j=1}^d \pi_j = s/2\\
    \pi_j \in [0,1] ~~ \forall j \in [d]\\
    \sum_{j=1}^d \pi_j^2 \sinh^2\left(\frac{\gamma_j^2}{2\sigma_j^2}\right) \leq c,\end{cases}
\end{align*}
for some small enough constant $c>0$ depending only on $\eta$. 
Although this is not how $\epsilon^*(s,t,\Sigma)$ is expressed in Theorems~\ref{th:rate_ellt_>2}--\ref{th:rate_Linfty}, this expression can be understood from the lower bound perspective. 
More precisely, for $s$ larger than some $c(\eta)$, the optimal prior is obtained by defining a random vector $\theta \in \R^d$ whose coordinates $\theta_j$ are mutually independent and satisfy 
\begin{align*}
    \forall j \in [d]: \theta_j = b_j \omega_j \gamma_j,
\end{align*}
where $b_j \sim \operatorname{Ber}(\pi_j)$, $\omega_j \sim \operatorname{Rad}(1/2)$ are mutually independent, and where $\pi_j$ and $\gamma_j$ solve the above optimization problem. 
The term $\sum_{j=1}^d \gamma_j^t \pi_j$ represents $\mathbb E\|\theta\|_t^t$, and the constraints $\sum_{j=1}^d \pi_j = s/2$ and $
    \pi_j \in [0,1]$ ensure that $\|\theta\|_0 \leq s$ with high probability. 
    Moreover, the condition $\sum_{j=1}^d \pi_j^2 \sinh^2\big(\gamma_j^2/2\sigma_j^2\big) \leq c$ ensures \textit{indistinguishability condition} 
    $$\operatorname{TV}^2(\mathbb P_{prior}, \mathbb P_0) \leq \chi^2(\mathbb P_{prior}, \mathbb P_0) \leq c,$$ 
where $\mathbb P_{prior}$ denotes the probability distribution induced by the above prior. 
Understanding the phase transitions in the behavior of $(\pi_j)_j$ and $(\gamma_j)_j$ gives valuable insights into the phase transitions occurring in the problem under study, including for the upper bounds. 
We discuss this optimization problem in more detail in Subsections~\ref{subsec:LB_t>2} and~\ref{subsec:LB_ellt}, and in Appendix~\ref{appendix:derivation_weights}, we give a heuristic intuition on how some test statistics from the upper bounds can be derived from this optimization problem as well.
}

\vspace{2mm}

\textit{Organization:} The rest of the paper is structured as follows. In Section~\ref{sec:results_ellt_>2}, we present the case where $t \geq 2$ before moving to the case $t \in [1,2]$ in Section~\ref{sec:results_ellt_leq2} and $t=\infty$ in Section~\ref{sec:Results_Linfty}. We give some examples in Section~\ref{sec:examples}, and conclude with a discussion of our results and some directions for future inquiry in Section \ref{sec:discussion}.

\vspace{2mm} 

\textit{Notation:} We denote by $\N^*$ the set of positive integers. 
For any $k \in \N^*$, we denote by $I_k$ the identity matrix of size $k$. 
Let $\eta>0$ and $t\in [1,\infty]$. 
For any two real-valued functions $f$ and $g$, we write $f \lesssim g$ (resp. $f \gtrsim g$) when there exists a constant $c(\eta,t) >0$ (resp. $C(\eta,t) >0$) depending only on $\eta$ and $t$, such that $c(\eta,t) \cdot g \leq f$ (resp. $f \leq C(\eta,t) \cdot g$). We write $f \asymp g$ if $g \lesssim f \text{ and } f \lesssim g$. 
We respectively denote by $x \vee y$ and $x \wedge y$ the maximum and minimum of the two real values~$x$ and~$y$, and we set $x_+ = x \lor 0$.
Note that the constants denoted by $C$ or $c$, depending on $\eta$ and $t$, are allowed to take different values on each appearance.
We also denote by $\operatorname{TV}(P,Q)$ the total variation between any two probability measures $P,Q$ defined over the same measurable space $(\mathcal{X},\mathcal{U})$. 
For any $d \in \mathbb N^*$, and for any property $P(j)$ over index $j \in [d]$, we set $\max\big\{j \in [d] ~\big|~ P(j)\big\} = 0$ if for any $j \in [d]$, $P(j)$ is false.

\section{Minimax rates in $L^t$ separation for $t \in [2,\infty)$}\label{sec:results_ellt_>2}

In this Section, we study Problem~\eqref{eq:testing_pb} with $t \in [2,\infty)$, and highlight that all of these problems can be approached in a similar way. 
The $L^2$ separation represents an important special case of this Section's results, as it has been extensively studied in the literature. 
Here we bridge the remaining gaps in Euclidean separation, which were left as an open question in~\cite{laurent2012non}.
To present the main result, we first need to introduce a few definitions.

Let $\beta \in \R$ be the unique solution to the equation
\begin{align}
    \frac{\sum_{j=1}^d \sigma_j^{t} \exp\left(-\beta/\sigma_j^2\right)}{\sqrt{\sum_{j=1}^d \sigma_j^{2t} \exp\left(-\beta/\sigma_j^2\right)}} =
     \frac{s}{2}, \quad \text{ and let } \lambda = \sqrt{\beta_+}.\label{eq:beta_t>2}
\end{align}

We first observe that the values $\beta$ and $\lambda$ in~\eqref{eq:beta_t>2} are well-defined: Lemma~\ref{lem:monotonicity_t>2} guarantees that equation~\eqref{eq:beta_t>2} admits a unique solution, by ensuring that the left-hand side is continuous, strictly decreasing, and tends to $+\infty$ as $\beta \to -\infty$ and to $0$ as $\beta \to +\infty$. 
The parameter $\lambda$ connects the three components of the problem under study, namely: the sparsity $s$, the heteroscedasticity profile of $\Sigma$, and the metric-inducing norm $\|\cdot\|_t$. 
The information captured by $\lambda$ is essential, as this value fundamentally appears in the expression of the minimax separation radius $\epsilon^*(s,t,\Sigma)$, in our lower bounds, and in the construction of our minimax optimal tests.
Unfortunately, the value of $\lambda$ cannot be expressed explicitly as a function of the $\sigma_j$'s in general, although it is possible to solve equation~\eqref{eq:beta_t>2} for some specific profiles of $\Sigma$ (see Section~\ref{sec:examples}).
Theorem \ref{th:rate_ellt_>2} below states the expression of the minimax separation radius.

\begin{theorem}\label{th:rate_ellt_>2}
Assume that $t \geq 2$.
Let $\lambda$ be defined as in equation~\eqref{eq:beta_t>2} and let $\nu = \bigg[\sum_{j=1}^d\sigma_j^{2t} e^{-\lambda^2/\sigma_j^2}\bigg]^{1/2t}$. Then the following hold.

\begin{enumerate}[label=\textbf{\roman*}.]
    \item\label{th:rate_ellt_>2_lower} [Lower Bound] 
    There exists a small constant $c$ depending only on $\eta$, such that 
    \vspace{-2mm}
    $$\epsilon^*(s,t,\Sigma)^t \geq  c \big(\lambda^t s + \nu^t \big).$$
    \Bernoulli{\vspace{5mm}}
    \vspace{-10mm}
    
    \item\label{th:rate_ellt_>2_upper} [Upper Bound] There exists a large enough constant $C'$ depending only on $\eta$ such that the test $\psi$ defined in~\eqref{eq:def_tests_>2} satisfies
    \vspace{-3mm}
\begin{align*}
\begin{cases}
    \mathbb P_\theta\left(\psi = 1\right) \leq \eta/2 & \text{ if } ~~ \theta=0,\\
    \mathbb P_\theta\left(\psi = 0\right)\leq \eta/2 & \text{ if } ~ \|\theta\|_0 \leq s \text{ and }\big\|\theta\big\|_t^t \geq C'\big(\lambda^t s + \nu^t\big).
\end{cases}  
\end{align*}
Therefore, $\epsilon^*(s,t,\Sigma)^t \leq C'\!\left( \lambda^t s \hspace{-.5mm} + \hspace{-.5mm} \nu^t\right).$
\end{enumerate}
\end{theorem}
Theorem \ref{th:rate_ellt_>2} immediately establishes that $\epsilon^*(s,t,\Sigma)^t \asymp \lambda^t s \!+\! \nu^t$. 
Another expression of $\epsilon^*(s,t,\Sigma)$ is given in the following corollary, which will be useful for further connections and discussions in later sections. 

\begin{corollary}\label{corr:ellt>2_lower} Under the assumptions of Theorem \ref{th:rate_ellt_>2} we have that
\begin{align}
    \epsilon^*(s,t,\Sigma)^t &\asymp  \lambda^t s + \nu^t, \quad\quad\quad\quad~ \text{ where } \nu^t = \Bigg[\sum_{j=1}^d\sigma_j^{2t} \exp\left(-\lambda^2/\sigma_j^2\right)\Bigg]^{1/2},\label{eq:implicit_separation_ellt_>2}
    \\
    & \asymp \lambda^t s + \sqrt{\sum_{j\leq j_*} \sigma_j^{2t}}, \quad \text{ where } j_* = \max\Big\{j \in [d]~\big|~ \sigma_j \geq \lambda\Big\}. \label{eq:explicit_separation_ellt_>2}
\end{align}

\end{corollary}
Theorem~\ref{th:rate_ellt_>2} is proved in Section \ref{sec:proof_theorem_ellt>2}. Proofs of the lower and upper bound are provided in subsections \ref{subsec:proof_theorem_ellt>2_lower} and \ref{subsec:proof_theorem_ellt>2_upper} respectively. Further, 
the simplification in~\eqref{eq:explicit_separation_ellt_>2} of Corollary \ref{corr:ellt>2_lower}  is proved in Lemma~\ref{lem:relations_between_contributions_ell>2}. 
 
Equation~\eqref{eq:explicit_separation_ellt_>2} will be interesting to compare with the opposite case $t\in [1,2]$ covered in Theorem~\ref{th:rate_ellt}, Section~\ref{sec:results_ellt_leq2}. 
\revision{It turns out that the separation radius $\epsilon^*(s,t,\Sigma)$ admits a simpler expression when $s\leq c(\eta)$ for some constant depending only on $\eta$. 
This simpler expression is given in Corollary~\ref{cor:s_leq_constant}.}
Finally, a few remarks are in order regarding the implications of Theorem~\ref{th:rate_ellt_>2}. We organize them along three subsections regarding the special case of $L^2$-separation, proof ideas behind the lower bounds, and motivations behind the upper bounds in the theorem.

\subsection{$L^2$-Separation}\label{subsec:L2_separation}

The $L^2$ separation has attracted the strongest attention in signal detection under Gaussian sequence models.
In the next corollary, we therefore collect our result in Euclidean separation, and specifically discuss its connections with and differences from existing literature.

\begin{corollary}\label{cor:rate_ell2} 
Let $\lambda_2 = \lambda^2$ where $\lambda$ is defined as in equation~\eqref{eq:beta_t>2} for $t=2$.
Then 
\begin{align}
    \epsilon^*(s,2,\Sigma)^2 &\asymp  \lambda_2 s + \nu_2 \quad\quad\quad\quad\quad\quad \text{ where } \nu_2 = \bigg(\sum_{j=1}^d\sigma_j^4 \exp\left(-\lambda_2/\sigma_j^2\right)\bigg)^{1/2}.\label{eq:implicit_separation_ell2}
    \\
    &\asymp \lambda_2 s + \bigg(\sum_{j\leq j_*} \sigma_j^4\bigg)^{1/2}, ~~ \text{ where } j_* = \max\Big\{j \in [d]~\big|~ \sigma_j^2 \geq \lambda_2\Big\}. \label{eq:explicit_separation_ell2}
\end{align}
\end{corollary}
A few remarks are in order regarding Corollary~\ref{cor:rate_ell2}.
\begin{enumerate}
    \item [(a)] Corollary~\ref{cor:rate_ell2} is best understood when compared with classical results in the isotropic case for $t = 2$. 
    In particular, assuming that $\sigma_1 = \dots = \sigma_d =: \sigma$ and using $I_d$ to denote the identity matrix of size $d$, ~\cite{collier2017minimax} derives
    \begin{align*}
     \epsilon^*(s,2,\sigma^2 I_d) \asymp \begin{cases}\sigma d^{1/4} & \text{ if } s\geq \sqrt{d},\\
        \sigma \sqrt{s\log\left(1+d/s^2\right)}& \text{ otherwise.}\end{cases}
    \end{align*}
    Further, the elbow at $s = \sqrt{d}$ in the expression above can be replaced by $s = c\sqrt{d}$ for any absolute constant~$c$ without affecting the rate, up to multiplicative constants depending only on $c$. Indeed, 
    this result can be recovered from Corollary~\ref{cor:rate_ell2} in the present paper. 
    To see this note that in the homoscedastic model $\beta$ solves $\sqrt{d}\exp\left(-\beta/2\sigma^2\right) = s/2$, i.e. $\beta = 2\sigma^2 \log\left(2\sqrt{d}/s\right)$. Therefore
 two cases emerge.
    \begin{itemize}
        \item [--] If $s> 2\sqrt{d/e}$, then $\sigma^2> \lambda_2$ and $j_* = d$.
    Therefore, equation~\eqref{eq:explicit_separation_ell2} yields that $\epsilon^*(s,2,\sigma^2 I_d) = \sigma d^{1/4}$. 
    \item [--] Otherwise, if $s < 2\sqrt{d/e}$, then $\sigma^2\leq \lambda_2$ and $j_*=0$ so that the minimax separation radius scales as $\epsilon^*(s,2,\sigma^2 I_d) \asymp \sqrt{\lambda_2 s} \asymp \sigma \sqrt{s \log(1+d/s^2)}$. 
    \end{itemize}
    \item [(b)] Noticeably, the isotropic case involves an extreme phase transition: we either have $j_*=0$  or $j_* = d$.
    In this case, all of the coordinates exclusively belong to the \textit{dense set of coordinates} $\{j \leq j_*\}$ or to the \textit{sparse set of coordinates} $\{j>j_*\}$.
    In our heteroscedastic model, however, the phase transition is more subtle. 
    When $s =  d$, all of the coordinates contribute to the dense regime $\left(\sum_{j=1}^d \sigma_j^4\right)^{1/4}$.
    When we let $s$ decrease from $d$ to $1$, the cut-off $j_*$ progressively shifts from $d$ to $0$. 
    \item [(c)] This progressive shift between the dense and sparse regimes is reflected in our expression of $\epsilon^*(s,2,\Sigma)$, which involves two contributions. 
    We recall that in the fully dense case $s=d$, the minimax separation radius $\epsilon^*(d,2, \Sigma)^2 $ is known to be $\left(\sum_{j=1}^d \sigma_j^4\right)^{1/2}$ -- see e.g.  \cite[Propositions 1 and 2]{laurent2012non}. 
    In comparison, the dense contribution $\big(\sum_{j\leq j_*} \sigma_j^4\big)^{1/4}$in Corollary~\ref{cor:rate_ell2} therefore represents the separation obtained by only testing the first $j_*$ coordinates with a sparsity $s' = j_*$. 
    The second contribution in this rate corresponds to the term $\lambda_2 s$. 
    In the isotropic case, the term $\lambda_2 s$ is responsible for the rate $\sigma \sqrt{s\log(d/s^2)}$ when $s\ll \sqrt{d}$.
    \item [(d)] When $\lambda_2 = 0$, which, by Lemma~\ref{lem:monotonicity_t>2}, is equivalent to the condition $$s/2 \geq  \frac{\sum_{j=1}^d \sigma_j^2}{\sqrt{\sum_{j=1}^d \sigma_j^4}},$$ it holds that $\epsilon^*(s,2,\Sigma) \asymp \epsilon^*(d,2,\Sigma)$. 
    In other words, sparsity does not help. 
    In the homoscedastic case, this phenomenon arises when $s\geq \sqrt{d}$. 
    In the heteroscedastic case, the elbow at $s^2=d$ is replaced by an elbow at $s^2 = \operatorname{Tr}^2(\Sigma) / \operatorname{Tr}(\Sigma^2)$. 
    This quantity is commonly referred to as the stable rank of $\Sigma$, and represents a suitable notion of intrinsic dimension.
    \item[(e)] Note that for general $t \geq 2$, still by Lemma~\ref{lem:monotonicity_t>2}, we have $\lambda = 0$ whenever $s^2\geq \operatorname{Tr}^2(\Sigma^{t/2}) / \operatorname{Tr}(\Sigma^t)$ b. 
    Sparsity never helps in this case. 
    We also note that the intrinsic dimension depends on $t$.
    \end{enumerate}

\subsection{Lower bounds for $t \geq 2$}\label{subsec:LB_t>2}

In our lower bound construction, we use Le Cam's two points method by defining a prior distribution over the parameter space $\Theta = \R^d$, which we detail here. 
We distinguish between two cases. 
\vspace{1mm}

If $s< c(\eta)$ for a sufficiently large constant $c(\eta)$ depending only on $\eta$, we use the immediate relation $\epsilon^*(s,t,\Sigma) \geq \epsilon^*(1,t,\Sigma)$ and further show that $\epsilon^*(1,t,\Sigma) \geq c(\lambda+\sigma_1) \asymp \lambda + \nu$, where $c>0$ is a sufficiently small constant.  
To do so, we propose a combination of two $1$-sparse priors, each being separated from the null hypothesis in $L^t$ norm by $c\lambda$ and $c\sigma_1$ respectively. 
Moreover, we prove that no test can distinguish them from the null hypothesis with high probability. 
This is made precise in Lemma~\ref{lem:generalities_LB} from Appendix~\ref{subapp:generalities_LB}, where we give the expression of the corresponding priors.
\vspace{1mm}

Conversely, if $s\geq c(\eta)$, our prior is as follows. 
We define a random vector $\theta \in \R^d$ whose coordinates $\theta_j$ are mutually independent and satisfy 
\begin{align}
    \forall j \in [d]: \theta_j = b_j \omega_j \gamma_j,\label{eq:def_prior_ellt>2}
\end{align}
where $b_j \sim \operatorname{Ber}(\pi_j)$, $\omega_j \sim \operatorname{Rad}(1/2)$ are mutually independent. 
Here,
\begin{align}
    \pi_j &= \frac{\sigma_j^t \exp\left(-\lambda^2/\sigma_j^2\right)}{\sqrt{\sum_{j=1}^d \sigma_j^{2t} \exp\left(-\lambda^2/\sigma_j^2\right)}},\label{eq:pi_j_>2}\\
    \gamma_j &= \sigma_j \arg \sinh^\frac{1}{2} \left[c \cdot \exp\left(\lambda^2/\sigma_j^2\right)\right] \asymp \lambda + \sigma_j,\label{eq:gamma_j>2}
\end{align} 
for some small enough $c$ depending only on $\eta$. 
In other words, each coordinate $\theta_j$ takes the value $0$ with probability $1-\pi_j$, and $\pm \gamma_j$ with probability $\pi_j \in [0,1]$. 
This vector $\theta$ therefore has a random sparsity $\sum_{j=1}^d b_j$, whose expectation is $\sum_{j=1}^d \pi_j  = s/2$ by the definition of $\lambda$ from~\eqref{eq:beta_t>2}. 

\vspace{2mm}

We now describe the intuitions behind the latter prior construction. 
The parameters $\pi_j$ and $\gamma_j$ in~\eqref{eq:pi_j_>2} and~\eqref{eq:gamma_j>2} are found by solving the following optimization problem
\begin{align}
    \max_{\gamma, \pi} ~~ \sum\limits_{j=1}^d \gamma_j^t \pi_j ~~ \text{ s.t. } ~~ \begin{cases}\sum_{j=1}^d \pi_j = s/2\\
    \pi_j \in [0,1] ~~ \forall j \in [d]\\
    \sum_{j=1}^d \pi_j^2 \sinh^2\left(\frac{\gamma_j^2}{2\sigma_j^2}\right) \leq c',\end{cases}
    \label{eq:optim_pb}
\end{align}
for some sufficiently small constant $c'$ depending on $\eta$. 
The interpretation of Problem~\eqref{eq:optim_pb} is as follows: Among all priors of the form~\eqref{eq:def_prior_ellt>2}, Problem~\eqref{eq:optim_pb} finds the parameters $(\pi_j)_j$ and $(\gamma_j)_j$ maximizing $\mathbb E\|\theta\|_t^t = \sum_{j=1}^d \gamma_j^t \pi_j$ under the constraints that $\|\theta\|_0 \leq s$ with high probability, and that no test can distinguish this prior from the null hypothesis with sufficient probability.
\vspace{2mm}

More precisely, 
the condition $\sum_{j=1}^d \pi_j = s/2$ guarantees that our prior's sparsity is at most $s$ with high probability, and the condition $\sum_{j=1}^d \pi_j^2 \sinh^2\big(\gamma_j^2/2\sigma_j^2\big) \leq c'$ ensures that $$\operatorname{TV}^2(\mathbb P_{prior}, \mathbb P_0) \leq \chi^2(\mathbb P_{prior}, \mathbb P_0) \leq c'$$ (see equation~\eqref{eq:chi2}), which we refer to as the \textit{indistinguishability condition}. 
In the last equation, we denoted by $\mathbb P_{prior}$ the probability distribution induced by the prior~\eqref{eq:def_prior_ellt>2}.
\vspace{2mm}

This variational problem elucidates how the phase transitions arise in the behavior of $\epsilon^*(s,t,\Sigma)$. 
We recall that $j_*$ is the index where the transition  $\sigma_j \lessgtr \lambda$ occurs (see equation~\eqref{eq:explicit_separation_ellt_>2}). 
Now from~\eqref{eq:gamma_j>2}, we can deduce that $j_*$ is also the index where the transition $\gamma_j \lessgtr c \sigma_j$ occurs. 
Therefore, over $\{1,\dots, j_*\}$, the following linearization holds: $\sinh\big(\gamma_j^2/2\sigma_j^2\big) \asymp \gamma_j^2/\sigma_j^2$. 
Conversely, over $\{j_*\!\!+\!1,\dots, d\}$, the following relation holds: $\sinh^2\big(\gamma_j^2/2\sigma_j^2\big) \asymp \exp(\gamma_j^2/\sigma_j^2)$. 
These two parts therefore exhibit fundamentally different behaviors in the analysis of the lower bound.
For $j\leq j_*$, the perturbation $\gamma_j \asymp \sigma_j$ coincides with the optimal perturbation that one would set in absence of sparsity (it suffices to evaluate our results at $s=d$, which implies $\lambda = 0$, so that $\gamma_j \asymp \sigma_j$ for any $j \in [d]$ by~\eqref{eq:gamma_j>2}). 
For $j>j_*$, the perturbation $\gamma_j \asymp \lambda$ is more surprising, as it does not depend on $\sigma_j$.
\vspace{2mm}

Moreover, we note that $\pi_j \propto \sigma_j^2 \exp\big(\!-\!\lambda^2/\sigma_j^2\big)$, which implies $\pi_1\geq \dots \geq \pi_d$, with a very fast decay of $\pi_j$ with $\sigma_j$ when $\sigma_j \ll \lambda^2$. 
In other words, our sparse prior preferably selects and perturbs the coordinates with largest standard deviations $\sigma_j$'s. 
This makes intuitive sense: If $\sigma_j = 0$, then under $H_0$, one should observe $X_j = 0$ \textit{a.s.}, so that the optimal value of $\pi_j$ should be $0$. 
This stands in contrast with the lower bound proposed in the paper~\cite{laurent2012non}, which also considers Problem~\eqref{eq:testing_pb} for $t = 2$. 
In the latter paper, the prior is defined by perturbing some coordinates $\big\{\theta_j ~\big|~ j\in J \big\}$, for a set $J$ selected \textit{uniformly at random} among subsets of $[d]$ of cardinality $s$, leading to sub-optimality in specific regimes. 
\vspace{2mm}

Finally, for $t \geq 2$, the constraint $\pi_j \in [0,1]$ is never saturated. 
To anticipate on Subsection~\ref{subsec:LB_ellt}, we will see that this constraint can be saturated for $t<2$, giving rise to a third regime which we will call the \textit{fully dense regime}.

\vspace{2mm}


\vspace{2mm}

Finally, the reason why we need to define two different priors when $s\geq c(\eta)$ or $s<c(\eta)$ is as follows.  
In the homoscedastic case~\cite{baraud2002non,laurent2012non,collier2017minimax, kotekal2021minimax}, the optimal sparse prior can be defined by selecting the support of $\theta$  uniformly at random among all subsets of $\{1,\dots,d\}$ with cardinality $s$. 
In our case, we would like to select exactly $s$ coordinates \textit{without replacement} in $[d]$, each coordinate being selected with probability proportional to $\pi_j$. 
This leads to technical difficulties, which we circumvent using the prior~\eqref{eq:def_prior_ellt>2}, which has independent coordinates but \textit{random} sparsity, equal to $\sum_{j=1}^d b_j$. 
In expectation, our prior has sparsity $\sum_{j=1}^d \pi_j \leq s/2$, with a  standard deviation of $\big(\sum_{j=1}^d  \pi_j(1-\pi_j)\big)^{1/2} \leq \sqrt{s}$. 
By Chebyshev's inequality our prior's sparsity is therefore at most $s$ with probability $1-\eta/10$, provided that $s$ is greater than a constant depending only on $\eta$ (see equation~\eqref{eq:control_prior_sparsity}). 
In the opposite case $s < c(\eta)$, such a prior would not be at most $s$-sparse with high probability. 
This is why we adopt another strategy, which is to set a $1$-sparse prior, selecting only \textit{one} coordinate $\{1,\dots,d\}$, each coordinate $j$ being selected with probability
proportional to~$\pi_j$. 

\subsection{Upper bounds for $t\geq 2$}

In this subsection, we describe the test achieving the upper bound in Theorem~\ref{th:rate_ellt_>2}.
Let $\lambda$ and $\nu$ be defined as in equation~\eqref{eq:beta_t>2} and~\eqref{eq:implicit_separation_ellt_>2} respectively, and let 
\begin{align}
    \uptau &= C_t \lambda^t + \nu^t/s \quad \text{ where } C_t = (4t)^{t}.
    \label{eq:def_nu_>2}
    \\
    \alpha_j &= \mathbb E \left[|Z_j|^t \,\Big|\, |Z_j|^t  >  \uptau \right]\quad \forall j > j_*, \quad\quad \text{where $Z_j \sim \mathcal{N}(0,\sigma_j^2)$}.
\end{align}

Subsequently we define our test statistics as follows:
\begin{align}
    \begin{cases}
     T_{dense} = \displaystyle\sum\limits_{j \leq j_*} |X_j|^t- \mathbb E_{H_0}|X_j|^t,\\[6pt]
       T_{sparse} = \displaystyle\sum\limits_{j > j_*} \big(|X_j|^t - \alpha_j\big)\mathbb 1\left\{|X_j|^t > \uptau\right\}.
    \end{cases}
    \label{eq:def_test_stat_>2}
\end{align}

For some large enough constant $C$ depending only on $\eta$, we finally define the test functions
\begin{align}
    &\begin{cases}
        \psi_{dense} = \mathbb 1\bigg\{T_{dense} \geq C \Big(\sum_{j \leq j_*} \sigma_j^{2t}\Big)^{1/2}\bigg\},\\ 
        \psi_{sparse} = \mathbb 1\Big\{T_{sparse} > C\rho \Big\},\\
        \psi = \psi_{dense} \lor \psi_{sparse}
    \end{cases}
    \label{eq:def_tests_>2}\\
    & \text{ where } \rho = \lambda^t s +  \nu^t.\label{eq:def_rho_>2}
\end{align}


We can now compare the upper bound from Theorem~\ref{th:rate_ellt_>2}.\ref{th:rate_ellt_>2_upper} with the literature in the isotropic case~\cite{collier2017minimax} for $t=2$. 
Recalling the notation from Corollary~\ref{cor:rate_ell2}, when $\sigma_1 = \dots = \sigma_d =: \sigma$ the test statistic used in~\cite{collier2017minimax} is defined as follows
\vspace{-10mm}

\begin{numcases}{T=}
    $$\displaystyle\sum\limits_{j=1}^d \big(X_j^2 - \sigma^2\big) $$ & $ \text{ if } s > \sqrt{d}$,\label{eq:dense_test_isotropic}
    \\[-8pt]
    $$\displaystyle\sum_{j=1}^d (X_j^2 - \alpha_s \sigma^2)\mathbb 1_{|X_j| > \sigma\sqrt{2\log(1+d/s^2)}}$$ &  $\text{ otherwise},$\label{eq:sparse_test_isotropic}
\end{numcases}
where $\alpha_s = \mathbb E\left[X^2 \,\big|\, X^2> 2\log(1+d/s^2)\right]$.
We can compare this with our results from Theorem \ref{th:rate_ellt_>2}.
As noted above, when $s > 2\sqrt{d/e}$, our result implies that $j_*=d$ and our dense test statistic coincides with~\eqref{eq:dense_test_isotropic}. 
On the other hand, when $s \leq 2\sqrt{d/e}$, we have $j_* = 0$ and our sparse test coincides with~\eqref{eq:sparse_test_isotropic}. 
Contrary to~\cite{collier2017minimax}, our phase transition occurs at $s = 2 \sqrt{d/e}$ instead of $s=\sqrt{d}$, which only affects the multiplicative constants in the rates.
Noticeably, equations~\eqref{eq:dense_test_isotropic} and~\eqref{eq:sparse_test_isotropic} show that the isotropic case requires only \textit{one} test at a time. 
To understand this, it suffices to note that in the isotropic case, $j_*$ is always equal to~$0$ or to~$d$, meaning that either the dense region $\{1,\dots, j_*\}$ or the sparse region $\{j_*\!+1,\dots,d\}$ is empty. 
In contrast, under heteroscedastic noise, two fundamentally different contributions coexist in the rate, requiring the use of two separate tests to handle both regimes. 

\vspace{2mm}

We can also analyze the truncation parameters of the sparse test statistics.
If $s \leq \sqrt{d}$, the isotropic sparse test~\eqref{eq:sparse_test_isotropic} requires the truncation $X_j^2 >2 \sigma^2 \log\!\big(1\!+\!d/s^2\big) \asymp 2 \lambda_2$. 
In comparison, our sparse test $T_{sparse}$ requires the truncation $X_j^2>\uptau_2 = C\lambda_2 + \nu_2/s$, which is larger than~$2\lambda_2$. 
The reason is that setting $\uptau_2 = 2\lambda_2$ would cause  $\operatorname{Var}_{H_0}\!\left[T_{sparse}\right]$ to be too large when the dense regime dominates, that is, when $\nu_2 \gg \lambda_2 s$. 
However, when the sparse regime dominates, i.e. when $\nu_2 \lesssim \lambda_2 s$, we recover $\uptau_2 \asymp \lambda_2$. 
In the isotropic case~\cite{collier2017minimax}, truncating at $2\lambda_2$ is sufficient, since $\nu_2/s$ never dominates over $2\lambda_2$ when $s \lesssim \sqrt{d}$ (see Remark~\ref{rem:isotropic}).
\vspace{2mm}

Finally, we compare our results with \cite{laurent2012non} where the authors also consider the heteroscedastic Gaussian sequence model under $L^2$ separation. Specifically \cite{laurent2012non} proposes to combine the test $\mathbb 1\left\{\sum_{j=1}^d X_j^2 > t_{d,1-\alpha/2}(\sigma)\right\}$ with $\mathbb 1 \big\{\max_j X_j^2/\sigma_j^2 > q_{d,1-\alpha/2}\big\}$. 
For completeness, we give the expressions of $t_{d,1-\alpha}$ and $q_{d,1-\alpha}$. 
Here $\alpha$ denotes the type-I error probability;
moreover, for any $\delta>0$, the quantity $t_{d,1-\delta}$ denotes the quantile of order $1-\delta$ of $\sum_{j=1}^d \xi_j^2$, with $\xi_j \sim \mathcal{N}(0,\sigma_j^2)$, and $q_{d,1-\delta}$ denotes the quantile of order $1-\delta$ of $\max_{j} \xi_j^2$.
They obtain an upper bound of the order of
$$ \epsilon^*(s,2,\Sigma)^2\lesssim \bigg(\sum_{j=1}^d \sigma_j^4\bigg)^{1/2} \land \sum_{j: \theta_j\neq 0}\sigma_j^2 \log n.$$

In comparison, we only use the chi-square test $T_{dense}$ over the first $j_*$ coordinates, which allows us to reduce the term $\left(\sum_{j=1}^d \sigma_j^4\right)^{1/4} $ to $\left(\sum_{j\leq j_*} \sigma_j^4\right)^{1/4} $. 
Moreover, we use the test statistic $T_{sparse}$, which allows us to take into account the \textit{number} of coordinates exceeding a certain value (namely $\tau$), rather than deciding in favor of $H_1$ if one value exceeds a suitable threshold.

\section{Minimax rates in $L^t$ separation for $t \in [1,2]$}\label{sec:results_ellt_leq2}

In this Section, we examine Problem~\eqref{eq:testing_pb} with $t\in[1,2]$. We emphasize that the case $t<2$ is significantly more challenging than $t\geq2$. 
Namely, it leads to more intricate phase transitions due to the fact that $\|\cdot\|_t$ is less smooth when $t<2$ than when $t\geq 2$. 
Once again we begin with a few definitions and some intermediate lemmas which will help us present the main result of the section. 
First we let
\begin{align}
    a = \frac{4t}{4-t}, \quad b = \frac{4-2t}{4-t}. \label{eq:def_a_b}
\end{align}

For any $x\geq 0$, we define
\begin{align}
     j_*(x) := \max \left\{j \in [d] ~\big|~ \sigma_j \geq x\right\}.\label{eq:def_jstar_function}
\end{align}

Thereafter, we use the convention that $\frac{1}{0} = +\infty$; moreover, for some small enough constant $\Cnu = \Cnu(\eta)$ depending only on $\eta$, and for any $x\geq 0$, we define $\overline \nu(x)$ as the solution to the equation
\begin{align}
    \sum_{j \leq j_*(x)} \frac{\sigma_j^{a}}{\overline \nu^{a}(x)} \land \frac{\sigma_j^4}{x^{4-2t} \, \overline \nu^{2t}(x)} + \sum_{j>j_*(x)} \frac{\sigma_j^{2t} }{\overline \nu^{2t}(x)}\exp\!\bigg(\!\!-\frac{x^2}{\sigma_j^2}+1\bigg) = \Cnu.\label{eq:def_nu_ellt}
\end{align}

We note that the equation above always admits a unique solution. 
To see this, we can note that, 
$j_*(x)$ being fixed, the left-hand side is a continuous function of $\overline \nu(x)$ that decreases from $+\infty$ to $0$.
We now also set 
\begin{align}
    &f(x) = \sum_{j\leq j_*(x)} 1 \land \frac{\sigma_j^4}{x^{4-t} \, \overline \nu^t(x)} + \sum_{j>j_*(x)} \frac{\sigma_j^t}{\overline \nu^t(x)} \exp\!\bigg(\!\!-\frac{x^2}{\sigma_j^2}+1\bigg)\label{eq:def_f_ellt}\\
    & \lambda := \inf f^{-1}\!\Big(\!\!\left\{s/2\right\}\!\!\Big) \label{eq:def_lambda_ellt}\\
    & \nu = \overline \nu(\lambda)\label{eq:def_nu(lambda)_ellt}\\
    &j_* = j_*(\lambda) \label{eq:def_jstar_ellt}\\
    &i_* = i_*(\lambda) = \max \left\{j \leq j_* ~\big|~ \sigma_j^4 \geq \lambda^{4-t} \nu^t\right\}.\label{eq:def_istar_ellt}
\end{align}

Lemma \ref{lem:nu_continuous_ellt} in Appendix \ref{sec:proof_theorem_ellt} ensures that the quantities $\lambda$ and $\nu$ given in~\eqref{eq:def_lambda_ellt} and \eqref{eq:def_nu(lambda)_ellt} are well-defined.
With this we are ready to present the main result regarding the minimax separation radius $\epsilon^*(s,t,\Sigma)$ for $t\in [1,2]$.

\begin{theorem}\label{th:rate_ellt}
Let $t\in [1,2]$ and $\lambda$  and  $\nu = \nu(\lambda)$ be defined as in~\eqref{eq:def_lambda_ellt} and~\eqref{eq:def_nu(lambda)_ellt}. Then the following hold
\vspace{-2mm}
\begin{enumerate}[label=\textbf{\roman*}.]
    \item\label{th:rate_ellt_<2_lower} [Lower Bound] 
There exists a small constant $c$ depending only on $\eta$, such that 
\vspace{-1mm}
$${\epsilon^*}(s,t,\Sigma)^t \geq  c\big( \lambda^t s+\nu^t \big).$$
\Bernoulli{\vspace{5mm}}
\vspace{-10mm}
    
    \item\label{th:rate_ellt_<2_upper} [Upper Bound] There exists a large enough constant $C'$ depending only on $\eta$ such that the test~$\psi^*$ defined in~\eqref{eq:def_tests_ellt} satisfies
\begin{align*}
\begin{cases}
    \mathbb P_\theta\left(\psi^* = 1\right) \leq \eta/2 & \text{ if } ~~ \theta=0,\\
    \mathbb P_\theta\left(\psi^* = 0\right)\leq \eta/2 & \text{ if } ~ \|\theta\|_0 \leq s \text{ and } \|\theta\|_t^t \geq C' \left(\lambda^t s + \nu^t\right).
\end{cases}  
\end{align*}
Therefore, $\epsilon^*(s,t,\Sigma)^t \leq C'\left( \lambda^t s + \nu^t\right).$
\end{enumerate}

\end{theorem}

Theorem \ref{th:rate_ellt_>2} immediately establishes $\epsilon^*(s,t,\Sigma)^t \asymp \left( \lambda^t s + \nu^t\right)$ when $t\in [1,2]$. It turns out that the separation behavior admits another representation which will be useful for interpretations and our later discussions. 
For easy reference we collect them in the following corollary.

\begin{corollary} Under the assumptions of Theorem \ref{th:rate_ellt} we have
\begin{align*}
    \epsilon^*(s,t,\Sigma)^t \asymp \lambda^t s + \nu^t.
\end{align*}
Moreover, writing $\sigma_{\leq i_*} = (\sigma_1,\dots, \sigma_{i_*})$ and recalling~\eqref{eq:def_a_b}, this expression can be rewritten as
\begin{align}
    \epsilon^*(s,t,\Sigma)^t \asymp \lambda^t s + \left\|\sigma_{\leq i_*}\right\|_a^t.\label{eq:explicit_separation_ellt_without_interm}
\end{align}
\end{corollary}

The proof of Theorem~\ref{th:rate_ellt} is detailed in Appendix \ref{sec:proof_theorem_ellt}  with proofs of the lower and upper bound provided in subsections \ref{subsec:proof_theorem_ellt_lower} and \ref{subsec:proof_theorem_ellt_upper} respectively. Further, the simplification claimed in equation~\eqref{eq:explicit_separation_ellt_without_interm} is proved in Lemma~\ref{lem:relations_between_contributions}, item~\ref{item:rate_ellt}. 
\revision{It turns out that the separation radius $\epsilon^*(s,t,\Sigma)$ admits a simpler expression when $s\leq c(\eta)$ for some constant depending only on $\eta$. 
This simpler expression is given in Corollary~\ref{cor:s_leq_constant}.} 
Now we once again provide a detailed discussion regarding the implications of Theorem~\ref{th:rate_ellt}.

\subsection{Regimes of minimax separation}

This rate is a combination of two terms. 
Similarly to the case $t \geq 2$, the first term $\lambda^t s$ cannot be expressed explicitly as a function of the $\sigma_j$'s in general. 
However, it is possible to solve the equations~\eqref{eq:def_nu_ellt}, \eqref{eq:def_f_ellt} and \eqref{eq:def_lambda_ellt} for some specific profiles of $\Sigma$ (see Section~\ref{sec:examples}).
The second term $\|\sigma_{\leq i_*}\|_a^t$ is best understood when compared with the rate in the absence of sparsity. 
When $s=d$, our result directly implies that the minimax separation radius in $L^t$ norm scales as $\epsilon^*(d,t,\Sigma)^t \asymp \|\sigma\|_a^t$. 
On the \textit{dense part} $(\sigma_1,\dots, \sigma_{i_*})$, sparsity is therefore irrelevant. 
We call this the ``fully dense'' regime, as in the lower bound, this contribution is obtained by setting a non-sparse prior (see Subsection~\ref{subsec:LB_ellt}).
This term showcases an interesting interpolation between the norms.
When $t\in[1,2]$, the minimax separation radius is expressed in terms of the $L^a$ norm where $a = a(t) = 4t/(4-t)$. 
This duality has also been highlighted for testing discrete distributions without sparsity in $\ell^t$ norm, $t\in[1,2]$ in the paper~\cite{chhor2022sharp} (see Subsection~\ref{subsec:mult_testing} for further details).

\vspace{2mm}
Our next lemma helps us present more insights to the results by providing a more interpretable expression for $\nu^t$.

\begin{lemma}\label{lem:explicit_expr_nu_ellt}
There exist two large constants $C_1, C_2 >1$ depending only on $\Cnu$ (hence independent of $\nu$), which can be made arbitrarily large provided that $\Cnu$ is small enough, such that $\nu^t \in \Big[C_1 \widetilde \nu^{\,t},\, C_2 \widetilde \nu^{\,t}\Big]$, where
$$ \widetilde \nu^{\,t} := \underbrace{\left[\sum_{j \leq i_*} \sigma_j^{a}\right]^{t/a}}_{\nu_{fdense}^t} + \underbrace{\frac{1}{\lambda^{2-t}} \left(\sum_{j = i_*\!+1}^{j_*} \sigma_j^4\right)^{1/2}}_{\nu_{inter}^t} +  \underbrace{\sqrt{\sum_{j>j_*} \sigma_j^{2t} \exp\left(-\lambda^2/\sigma_j^2\right)}}_{\nu_{sparse}^t}.$$
\end{lemma}

Note that the expression of $\nu$ from Lemma~\ref{lem:explicit_expr_nu_ellt} involves three contributions. 
Surprisingly, the last two terms $\nu_{inter}^t$ and $\nu_{sparse}^t$ never dominate in the rate (see Lemma~\ref{lem:relations_between_contributions}). 
One could therefore believe that only two regimes coexist in the minimax separation radius. 
However, this is not the case. 
In fact, the minimax separation radius contains three regimes: a \textit{fully dense} regime over $(\sigma_1,\dots,\sigma_{i_*})$, an \textit{intermediate} regime over $(\sigma_{i_*\!+1}, \dots, \sigma_{j_*})$ and a \textit{sparse} regime over $(\sigma_{j_*\!+1},\dots, \sigma_d)$, 
the intermediate and sparse regimes being hidden in the term $\lambda^t s$. 
As will be discussed in Subsections~\ref{subsec:LB_ellt} and~\ref{subsec:UB_ellt}, the three regimes involve very different phenomena, the intermediate one, however, sharing similarities with the other two.
This is reflected by our upper bound, which requires three tests.
When $t=2$, we get $a=4$, and the fully dense and intermediate parts merge into one single regime.
\vspace{2mm}

Finally, we note that although Theorem~\ref{th:rate_ellt} can be made to encompass the case where $t=2$, the method presented in Section~\ref{sec:results_ellt_>2} is more intuitive in that special case.

\subsection{Lower bounds for $t\in [1,2]$}\label{subsec:LB_ellt}
\begin{proposition}\label{prop:LB_gaussian_ellt}
Let $\lambda$ be defined as in~\eqref{eq:def_lambda_ellt} and let $\nu = \nu(\lambda)$ be the solution to equation~\eqref{eq:def_nu_ellt} for $x=\lambda$.
Then there exists a small constant $c$ depending only on $\eta$, such that 
$${\epsilon^*}^t \geq  c\big(\nu^t + \lambda^t s\big).$$
\end{proposition}

To prove the lower bound, we once again use Le Cam's two points method. 
Again, we distinguish between two cases. 
\vspace{1mm}

If $s< c(\eta)$ for a sufficiently large constant $c(\eta)$ depending only on $\eta$, then we bound from below $\epsilon^*(s,t,\Sigma)$ by $\epsilon^*(1,t,\Sigma)$ and propose a combination of $1$-sparse priors separated from the null hypothesis by $c\lambda$ and $c\nu$ respectively, and that cannot be distinguished from the null hypothesis with high probability (see Lemma~\ref{lem:generalities_LB} for details).
\vspace{1mm}

Conversely, if $s\geq c(\eta)$, we define the following prior distribution over the parameter space $\Theta = \R^d$. 
We define a random vector $\theta \in \R^d$ whose coordinates $\theta_j$ satisfy $\forall j \in [d]: \theta_j = b_j \omega_j \gamma_j$, where $b_j \sim \operatorname{Ber}(\pi_j)$, $\omega_j \sim \operatorname{Rad}(\frac{1}{2})$ are mutually independent, and where the parameters $\pi_j, \gamma_j$ are given in Table~\ref{tab:prior_param_1_2} below

\begin{table}[h!]
\centering
\begin{tabular}{|c|c|c|c|}
     \hline
     & {\phantom{\Big|}}$j \leq i_*${\phantom{\Big|}} & $i_*\! < j \leq j_*$ & $j> j_*$  
     \\
     \hline
     {\phantom{\Bigg|}} $\pi_j$ {\phantom{\Bigg|}} &  $1$ & $\dfrac{\sigma_j^4}{\nu^{\,t} \lambda^{4-t} }$ & ${\phantom{\Bigg|}}\dfrac{\sigma_j^t}{\nu^t} \exp\!\left(\!\!-\dfrac{\lambda^2}{\sigma_j^2}\right)$ {\phantom{\Bigg|}}\\
     \hline
     $\gamma_j^t$ &  $c\dfrac{\sigma_j^a}{\nu^{a-t}}$ {\phantom{\Big|}} & $c\lambda^t$ & $c\lambda^t$\\
     \hline
\end{tabular}
    \caption{Values of the prior parameters for $t\in[1,2]$.}
    \label{tab:prior_param_1_2}
\end{table}

In Table~\ref{tab:prior_param_1_2}, $c$ denotes some small enough constant depending only on $\eta$. 
The parameters in the above table are found by solving the variational problem~\eqref{eq:optim_pb} for $t \leq 2$, which is significantly more involved than in the case $t\geq 2$, as the constraint $\pi_j \leq 1$ can be saturated. 
This gives rise to a new phase transition occurring at $i_*$, on top of the phase transition at $j_*$. 
As a reminder, $j_*$ is the index after which the linearization $\sinh^2\big(\gamma_j^2/2\sigma_j^2\big) \asymp \gamma_j^4/\sigma_j^4$ no longer holds and has to be replaced by $\sinh^2\big(\gamma_j^2/2\sigma_j^2\big) \asymp \exp(\gamma_j^2/\sigma_j^2)$.

\vspace{2mm}

The indices $j\leq i_*$ form the \textit{fully dense regime}, which characterizes the largest values of the $\sigma_j$'s. 
In this regime, the Bernoulli parameters $\pi_j$ are all equal to $1$ (in other words, the optimal prior has no sparsity at all) and the optimal perturbations $\gamma_j^t = c\sigma_j^a/\nu^{a-t}$ are proportional to the values that would be optimal in absence of sparsity (up to the rescaling factor $1/\nu^{a-t}$). 
In this regime, sparsity is irrelevant.
As soon as $\pi_j$ no longer saturates the constraint ``$\pi_j=1$'', a phase transition occurs in the behavior of $\gamma_j$.   
An interesting phenomenon arises concerning the decay of $\gamma_j$.
The perturbation $\gamma_j^t$ first decreases from $j=1$ to $j = i_*$ until it reaches $c\lambda^t$.
After $i_*$ it remains equal to $c \lambda^t$ independently of $\sigma_j$. 
Over $\{i_*\!+\!1,\dots, d\}$, $\gamma_j$ therefore does not exhibit any phase transition, contrary to $\pi_j$. 
This is very surprising, given that the two parts $j\leq j_*$ and $j>j_*$ exhibit fundamentally different behaviors in the analysis of the lower bound, and it was unexpected to observe that the intermediate zone $j \in \{i_*\!+1, \dots, j_*\}$ and the sparse zone $j > j_*$ share the same magnitude of the $\gamma_j$'s.

\subsection{Upper bounds for $t\in [1,2]$}\label{subsec:UB_ellt}

In this subsection, we describe the tests achieving the rate in Theorem~\ref{th:rate_ellt}.
Let $\lambda$ and $\nu$ be defined as in~\eqref{eq:def_lambda_ellt} and~\eqref{eq:def_nu(lambda)_ellt} respectively. 
We let 
\begin{align}
    &\uptau ~= C_t \lambda^t + \nu^t/s, \quad \text{ where } C_t = (4t)^{t}  \label{eq:def_tau_ellt}\\
    &\alpha_j := \mathbb E \Big(\,|Z_j|^t \; \Big| \; |Z_j|^t  >  \uptau \Big) \quad \text{where $Z_j \sim \mathcal{N}(0,\sigma_j^2)$}, ~~\forall j >j_*. \nonumber
\end{align}

and now define the test statistics as follows:
\begin{align}
    & T_{fdense} = \sum_{j \leq i_*} \frac{1}{\sigma_j^{2b}} \left(X_j^2 -\sigma_j^2 \right), \quad \text{ where $b$ is defined in~\eqref{eq:def_a_b}},\label{eq:def_Tfdense_ellt}\\
    &T_{inter} = \sum_{j = i_*\!+1}^{j_*} X_j^2 - \sigma_j^2,
    \label{eq:def_Tinter_ellt}\\
    &T_{sparse} = \sum_{j > j_*} \big(|X_j|^t - \alpha_j\big)\mathbb 1\left\{|X_j|^t > \uptau\right\}.\label{eq:def_Tsparse_ellt}
\end{align}

For a large enough constant $C$ depending only on $\eta$, we finally define the test functions as follows
\begin{alignat}{2}
    &\psi_{fdense} &&= \mathbb 1\Bigg\{T_{fdense} \geq C \bigg(\sum_{j \leq i_*} \sigma_j^{a}\bigg)^{1/2}\Bigg\},\\
    &\psi_{inter} &&= \mathbb 1\Bigg\{T_{inter} \geq \frac{C}{\lambda^2} \sum_{j = i_*\!+1}^{j_*} \sigma_j^4\Bigg\},\\
    &\psi_{sparse} &&= \mathbb 1\Big\{T_{sparse} > C\rho \Big\}, \quad \quad \quad \quad  \text{ where } \rho = \lambda^t s +  \nu^t.\label{eq:def_rho_ellt}\\ 
    &~~ \psi^* &&= \psi_{fdense} \lor \psi_{inter} \lor \psi_{sparse},\label{eq:def_tests_ellt}
\end{alignat}

We prove the following Proposition which complements the lower bound from Proposition~\ref{prop:LB_gaussian_ellt} by a matching upper bound.
\begin{proposition}\label{prop:UB_gaussian_ellt}
Recall the definition of $\rho$ in~\eqref{eq:def_rho_ellt}. There exists a large enough constant $C'$ depending only on $\eta$ such that
\begin{align*}
\begin{cases}
    \mathbb P_\theta\left(\psi^* = 1\right) \leq \eta/2 & \text{ if } ~~ \theta=0,\\
    \mathbb P_\theta\left(\psi^* = 0\right)\leq \eta/2 & \text{ if } ~ \|\theta\|_0 \leq s \text{ and } \|\theta\|_t^t \geq C' \rho.
\end{cases}  
\end{align*}
\end{proposition}
Proposition~\ref{prop:UB_gaussian_ellt} is proved in Appendix~\ref{subsec:proof_theorem_ellt_upper}.
As already evoked, the test $\psi^*$ is a combination of three tests, each one being taylored for the three different regimes. 
The first two tests statistics $T_{fdense}$ and $T_{inter}$ share the similarity of being non-truncated chi-square test statistics. 
Their reweightings are however different. 
The last two test statistics $T_{inter}$ and $T_{sparse}$ are both non-reweighted test statistics. 

\vspace{2mm}

The re-weighting in the test statistic $T_{fdense}$ shares similarities with a test statistic obtained in~\cite{chhor2022sharp} in the context of multinomial distributions. 
Letting $\mathcal{P}_d = \{p \in \R_+^d: \sum_{j=1}^d p_j = 1\}$, and given an observation $X \sim \mathcal{M}(p,n)$ drawn from a multinomial distribution with a probability parameter $p \in \R^d$ and a number of observations $n \in \N$, the authors consider the following testing problem: 
\begin{align*}
    H_0: p=p_0 \quad \quad \text{ against } \quad \quad H_1: \begin{cases}
    p\in \mathcal{P}_d\\
    \|p-p_0\|_t\geq \rho^*,
    \end{cases}
\end{align*}

Here, $t \in [1,2]$ and $p_0 \in \mathcal{P}_d$ is a given probability vector assumed to be sorted in non-increasing order $p_0(1)\geq \dots, \geq p_0(d)>0$.   
For a suitably defined index $I$, this problem involves two fundamentally different regimes. 
The ``bulk'' part $(p_0(1),\dots,p_0(I))$ involves a subgaussian regime, whereas the ``tail'' part $(p_0(I+1),\dots,p_0(d))$ involves a subpoissonian regime. 
The test statistic corresponding to the bulk part is a reweighted $\chi^2$ test statistic similar to $T_{fdense}$, where the reweighting is $1/p_0^b(j)$ instead of $1/\sigma_j^{2b}$. 
These two reweightings turn out to be analogous: Namely, each coordinate $X_j$ of $X$ has a standard deviation of the order of $\widetilde \sigma_j := \sqrt{np_0(j)}$ under $H_0$, hence proportional to $\sqrt{p_0(j)}$. 
Therefore, the reweighting $1/\widetilde \sigma_j^{2b}$ is proportional to $1/p_0^b(j)$. 
More connections with the multinomial setting are also discussed in Subsection~\ref{subsec:mult_testing}.

\section{Minimax rates in $L^\infty$ separation}\label{sec:Results_Linfty}

\revision{
In this section, we consider the case of the $L^\infty$ separation. 
The results from this Section can be obtained from Sections~\ref{sec:results_ellt_>2} and~\ref{sec:results_ellt_leq2} by noting that, for any $t\geq 1$, testing in $L^\infty$ separation is equivalent to testing in any $L^t$ separation with $s=1$. 
Indeed, when $t=\infty$, the problem reduces to detecting signals $\theta$ whose largest coordinate $|\theta_j|$ exceeds a certain threshold. 
Among such vectors, the most difficult ones to detect are clearly $1$-sparse, and their $L^t$ norm therefore coincides with their $L^\infty$ norm. 
The next lemma makes this claim precise.
\vspace{-1mm}
\begin{lemma}\label{lem:equiv_Linfty_Lt}
The following relations hold for any $s \in [d]$ and any $t \geq 1$:
\vspace{-2mm}
$$ \epsilon^*(s,\infty, \Sigma) = \epsilon^*(1,\infty, \Sigma) = \epsilon^*(1,t,\Sigma).$$
\end{lemma}
\vspace{-1mm}
From Lemma~\ref{lem:equiv_Linfty_Lt}, we observe that $\epsilon^*(1,t,\Sigma)$ is independent of $t \in [1,\infty]$, which is not immediate from Theorems~\ref{th:rate_ellt_>2} and \ref{th:rate_ellt}.
To obtain the expression of $\epsilon^*(s,\infty,\Sigma)$, we can therefore evaluate any $\epsilon^*(1,t',\Sigma)$ and choose $t'=2$ for simplicity.
The following theorem gives the behavior of the minimax separation radius in $L^\infty$ separation.
}
\begin{theorem}\label{th:rate_Linfty}
Define $\lambda$ and $\nu$ respectively as in~\eqref{eq:beta_t>2} and~\eqref{eq:implicit_separation_ellt_>2} by taking $t'=2$ and $s' = 1$ in these equations.
\vspace{-5.5mm}

\begin{enumerate}[label=\textbf{\roman*}.]
    \item\label{th:rate_Linfty_rate} There exist two constants $C,c>0$ depending only on $\eta$ such that for any $s \in [d]$, we have $c\,(\lambda + \nu)\leq \epsilon^*(s,\infty, \Sigma) \leq C(\lambda + \nu)$.
    \vspace{0mm}
    
    \item\label{th:rate_Linfty_UB} Moreover, there exists a large enough constant $C'$ depending only on $\eta$ such that the test~$\psi^*$ defined in~\eqref{eq:test_Linfty} satisfies
    \vspace{-3mm}
    \begin{align*}
    \begin{cases}
        \mathbb P_\theta\left(\psi^* = 1\right) \leq \eta/2 & \text{ if } ~~ \theta=0,\\
        \mathbb P_\theta\left(\psi^* = 0\right)\leq \eta/2 & \text{ if } ~ \|\theta\|_\infty \geq C' (\lambda + \nu).
    \end{cases}  
    \end{align*}
\end{enumerate}
\end{theorem}
Theorem~\ref{th:rate_Linfty}.\ref{th:rate_Linfty_rate} is a corollary of Lemma~\ref{lem:equiv_Linfty_Lt} and Theorem~\ref{th:rate_ellt_>2}.  
The proof of Theorem~\ref{th:rate_Linfty}.\ref{th:rate_Linfty_UB} is given in Appendix~\ref{subapp:UB_Linfty} and relies on the following test. 
For some large enough constant $C>0$ depending only on $\eta$, we define the test rejecting if at least one coordinate $|X_j|$ exceeds the threshold $C\rho$, where $\rho = \lambda+\nu$:
\begin{equation}\label{eq:test_Linfty}
    \psi^* = \mathbb 1\Big\{\max_{j \, \in \, [d]} |X_j| > C\rho\Big\}.
\end{equation}

As in the cases $t\in[1,2]$ and $t \geq 2$, the term $\lambda + \nu$ cannot be expressed as an explicit function of the $\sigma_j$'s in general. 
However, in the isotropic case, we recover a well-known result. 
Indeed, if $\sigma_1 = \dots = \sigma_d =: \sigma$, then $\lambda+ \nu \asymp \sigma\sqrt{\log(d)}$. 
This result makes intuitive sense. 
In the isotropic case, the largest value $X_j$ where $X_j \overset{iid}{\sim} \mathcal{N}(0,1)$ is of the order of $\sigma \sqrt{\log(d)}$ with high probability. 
The test~\eqref{eq:test_Linfty} achieving this rate is equivalent to checking that no value observed exceeds this threshold.
Note also that our bound does not depend on $s$. 
This is also clear: To detect in $L^\infty$ separation, the worst possible perturbations are $1$-sparse.\\ 

\revision{
It turns out that the minimax separation radius in $L^\infty$ norm admits an explicit expression with respect to the $\sigma_j$'s, which is given in the Theorem below. 
We thank the anonymous Referee for bringing this to our attention.
We first define the test statistic 
\begin{align}
    \widetilde \psi =\mathbb{1}\Big\{\max _{i}\left|X_{i}\right|>C \rho\Big\} \quad \text{ where } \quad
    \rho =\max _{i} \sigma_{i} \sqrt{\log (1+i)}.\label{eq:def_test_explicit_infty}
\end{align}
\begin{theorem}\label{th:rate_Linfty_explicit}
\begin{enumerate}
    \item (Lower bound) There exists a small constant $c>0$ depending only on $\eta$ such that
    $$\epsilon^{*}(1, \infty, \Sigma) \geq c \cdot \max_{i} \sigma_{i} \sqrt{\log (1+j)}.$$
    \item (Upper bound) There exist large enough constants $C, C^{\prime}$ depending only on $\eta$ such that the test $\widetilde \psi$ satisfies
    $$\begin{cases}\mathbb{P}_{\theta}(\widetilde\psi=1) \leq \eta/2 \quad \text { if } \theta=0, \\ \mathbb{P}_{\theta}(\widetilde \psi=0) \leq \eta/2 \quad \text { if }\|\theta\|_{\infty} \geq C^{\prime} \rho \text{ and } \|\theta\|_0 \leq s.\end{cases}$$
\end{enumerate}
\end{theorem}
The proof of Theorem~\ref{th:rate_Linfty_explicit} can be found in Appendix~\ref{app:proof_th_infty_explicit}. This theorem immediately establishes that $\epsilon^*(s,\infty,\Sigma) \asymp \max_i \sigma_i\sqrt{\log(1+i)}$.
This expression, which also appears in the threshold of test~$\widetilde\psi$, precisely represents the rate of the maximum of $d$
independent centered Gaussian random variables with variances $\sigma_1^2\geq\dots\geq\sigma_d^2$ respectively (see Lemmas 2.3 and 2.4 from~\cite{van2017spectral}, and~\cite{talagrand2014upper}). 
In the isotropic case, we recover $\epsilon^*(s,\infty,\sigma^2 I_d) \asymp \sigma \sqrt{\log(d)}$. 
Note that the two equivalent tests~\eqref{eq:test_Linfty} and~\eqref{eq:def_test_explicit_infty} are both adaptive to unknown sparsity, as their expression is independent of $s$.
\begin{corollary}\label{cor:s_leq_constant}
Let $C(\eta)\geq 1$ be any constant depending only on $\eta$, and assume that $s\leq C(\eta)$. 
Then, for any $t\in[1,\infty]$, it holds that
$$\epsilon^*(s,t,\Sigma) \asymp \max_{i} \sigma_i \sqrt{\log(1+i)}.$$
\end{corollary}
\begin{proof}
From Theorems \ref{th:rate_ellt_>2}--\ref{th:rate_ellt} if $t\in [1,\infty)$ and Lemma~\ref{lem:equiv_Linfty_Lt} if $t = \infty$, we directly obtain that $\epsilon^*(s,t,\Sigma) \asymp \epsilon^*(1,t,\Sigma)$ when $s \leq C(\eta)$. 
The result follows from $\epsilon^*(1,t,\Sigma) = \epsilon^*(1,\infty,\Sigma)$ (Lemma~\ref{lem:equiv_Linfty_Lt}) and Theorem~\ref{th:rate_Linfty_explicit}.
\end{proof}
}

\section{Examples}\label{sec:examples}
To provide more insight into our results we now discuss some examples for specific heterogeneity profiles $\Sigma$. 
The proofs of the results presented in this Section can be found in Appendix~\ref{app:proof_examples}.

\subsection{Isotropic case}\label{subsec:isotropic}
Assume that $\sigma_1 = \dots = \sigma_d = \sigma$. 
Up to multiplicative constants, the minimax separation radius $\epsilon^*(s,t,\sigma^2 I_d)$ can be expressed as in Table~\ref{tab:isotropic}. 
Note that these results encompass the case $t=2$ which has been thoroughly examined in~\cite{collier2017minimax}.
\begin{table}[h!]
\begin{center}
    \begin{tabular}{|c|cc|c|}
\hline
 $\epsilon^*(s,t,\sigma^2 I_d)$ & \multicolumn{1}{c|}{\hspace{5mm} $t \leq 2$ \hspace{5mm} } & \hspace{5mm} $t \geq 2$ \hspace{5mm} & $t= \infty$                 \rule[-2ex]{0pt}{6ex}  \\ \hline
If $s \geq C\sqrt{d}$ & \multicolumn{1}{c|}{$\sigma d^{\frac{1}{4}} s^{\frac{1}{t} - \frac{1}{2}}$ }  & $\sigma d^{\frac{1}{2t}}$ & \multirow{2}{*}{$\sigma\sqrt{\log(d)}$} \rule[-3ex]{0pt}{7ex} \\ \cline{1-3}
If $s<C\sqrt{d}$ & \multicolumn{2}{c|}{$\sigma s^{\frac{1}{t}} \log^{\frac{1}{2}}\bigg(\dfrac{2\sqrt{d}}{s}\bigg)$}    &                  \rule[-3ex]{0pt}{7ex}  \\ \hline
\end{tabular}
\end{center}
    \caption{Minimax separation radii in the isotropic case, for $t \in [1,\infty]$.}
    \label{tab:isotropic}
\end{table}

In the isotropic case for $t \geq 2$, sparsity does not help for testing when $s \geq C \sqrt{d}$. 
However, this is no longer the case for $t<2$: sparsity always improves the rates as soon as $s \ll d$. 
The case $t \geq 2$ has been investigated in~\cite{gutzeit2019topics} in the isotropic case and without sparsity.


\subsection{Polynomially increasing variances}
For ease of notation, we assume in this subsection only that the $\sigma_j$'s are sorted in non-decreasing order: $0< \sigma_1 \leq \dots \leq \sigma_d$. 
We assume that for some fixed constant $\alpha$, we have $\forall j \in [d]: \sigma_j = j^\alpha$. 
The case where $t=2$ has been examined in~\cite{laurent2012non}, but with a logarithmic gap between the upper and the lower bound. 
Here, we assume that $t \geq 2$. 
The following relations hold up to constants depending on $\eta, \alpha$ and $t$
\begin{align*}
    \epsilon^*(s,t,\Sigma)^t \asymp \begin{cases}d^{\alpha t} s \log \left(C\dfrac{d}{s^2}\right) & \text{ if } s \leq \sqrt{d} \\
    d^{\alpha t + \frac{1}{2}} & \text{ otherwise.}
    \end{cases}
\end{align*}

\subsection{Variances decreasing at least exponentially}
Let $\phi: [0,\infty) \longrightarrow \R$ be a non-decreasing convex function such that $\phi(0) = 0$, and let $$\sigma_j = \exp\left(-\phi(j)\right) \quad \text{ for } \quad j \in \{0,\dots,d\}.$$

We note that here we let $\sigma_0 = 1$ for the sake of convenience of exposition. 
Let $j_0 = \min\{j \,|\, \sigma_j < 1/2\}$ if this minimum is taken over a non-empty set and $j_0 = d+1$ otherwise. 
Then the minimax separation radius satisfies
\begin{align}
    \forall t \geq 2: ~~ \epsilon^*(s,t,\Sigma) \asymp \epsilon^*(s,t,\sigma_{j_0}^2 I_{j_0}) \asymp \begin{cases} j_0^{1/2t} & \text{ if } s \geq C \sqrt{j_0},\\
    s \log^{t/2}\left(\dfrac{2\sqrt{j_0}}{s}\right) & \text{ otherwise.}
\end{cases}.\label{eq:rate_faster_exponential}
\end{align}
Here, the dimension $j_0$ can understood as the size of the set $\{0,\dots j_0\}$ on which the values $\sigma_j$ can be considered as constant: $\sigma_j \in \big[\frac{1}{2}, 1\big], \forall j \in \{0,\dots j_0\}$.
\vspace{2mm}

\revision{Note that this result encompasses the case of exponentially decreasing variances. 
More precisely, let $\alpha \in (0,1]$ and assume that $\forall j \in \{0,\dots d\}: \sigma_j = \alpha^j$. 
Let $j_0 = \min\big\{j \in [d] ~\big|~ \alpha^{j} < \alpha/2\big\}$ if this minimum is taken over a non-empty set, and set $j_0 = d+1$ otherwise.
The case where $t=2$ has been examined in~\cite{laurent2012non}. 
For $t \geq 2$,  the minimax separation radius is given by~\eqref{eq:rate_faster_exponential}.
Observe that this result also encompasses the isotropic case from Subsection~\ref{subsec:isotropic}.
For non-pathological decays of the values $\sigma_j$'s, that is, for $\alpha \leq 1-\delta$ where $\delta>0$ is some fixed constant, the index $j_0$ will typically be a constant depending on $\delta$.
Namely: $j_0 \leq \log_{\alpha^{-1}} (2) \leq \log_{(1-\delta)^{-1}} (2)$. 
Therefore, as soon as $s$ is greater than a constant (depending on $\delta$), the minimax separation radius will further simplify as $\epsilon^*(s,t,\Sigma) \asymp_\delta \alpha = \sigma_1$, regardless of the sparsity.
This makes intuitive sense: 
When the $\sigma_j$'s decay exponentially fast, the intrinsic dimension of the data, given by $d_{intrinsic} = \operatorname{Tr}(\Sigma^{t/2}) / \sqrt{\operatorname{Tr}(\Sigma^t)}$, is of the order of a constant. 
Therefore, sparsity should not be relevant if $s \geq \sqrt{d}_{intrinsic} = Cste$ (see the discussion in Subsection~\ref{subsec:L2_separation}).
\vspace{2mm}
}

\revision{However, the rate can be more subtle if $\alpha$ approaches $1$, and equation~\eqref{eq:rate_faster_exponential} reveals that whenever $\sigma_j = \alpha^j$ for $j=0\dots,d$, the testing problem is essentially as difficult in the isotropic case with covariance matrix $\alpha^2 I_{j_0}$. 
Here, the dimension $j_0$ can understood as the size of the set $\{0,\dots j_0\}$ on which the values $\sigma_j$ can be considered as constant: $\sigma_j \in \left[\sigma_1/2,\, \sigma_1\right], \forall j \in \{0,\dots, j_0\}$.
}

\section{Discussion}
\label{sec:discussion} 

In this Section, we discuss further connections and implications of our results in comparison to the literature along with possible future directions.

\subsection{Minimax estimation of $\|\theta\|_t$ for $t\geq 1$}
A natural question is to compare this paper's results about testing with the corresponding task of estimating $\|\theta\|_t$ for $t\geq 1$. 
Estimation of non-smooth functionals have been considered in~\cite{cai2011testing,jiao2015minimax,wu2019chebyshev,wu2016minimax,carpentier2019adaptive,fukuchi2017minimax,butucea2021locally}.
Some techniques used in the present paper can be linked with techniques developed in~\cite{collier2017minimax}, and more closely, in~\cite{collier2020estimation}.
The paper~\cite{collier2020estimation} considered the problem of estimating $\|\theta\|_t$ over $\big\{\theta \in \mathbb R^d ~\big|~ \|\theta\|_0 \leq s \big\}$ given an observation $X \sim \mathcal{N}(\theta, \sigma^2 I_d)$. 
The difficulty of this estimation problem is characterized by the minimax estimation risk defined as
$$ \mathcal R_{s,d}(\sigma,t) := \inf_{\widehat T} \sup_{\substack{\theta\in \R^d\\ \|\theta\|_0 \leq s}} \mathbb E \left[\big|\widehat T - \|\theta\|_t\big|^2\right],$$
where  the infimum is taken over all estimators $\hat T: \R^d \longrightarrow \R$.
In Table~\ref{tab:comparison_estim_test} below, we collect the results  of~\cite{collier2020estimation} for $t\geq 1$, and compare them with the results of the present paper. 
To preserve homogeneity, we give here the expression of the \textit{square root} of the minimax estimation risk of $\|\theta\|_t$, namely $\mathcal{R}_{s,d}^{1/2}(\sigma,t)$, as well as our expression for the minimax separation radius in $L^t$ norm $\epsilon^*(s,t,\sigma^2I_d)$. 
We denote by $2\N^*$ the set of positive even integers.

\begin{table}[!ht]
\centering
\begin{tabular}{|c|c|cc|}
\hline
If $t \notin 2\mathbb{N}^*$     
& $s \leq \sqrt d$                           
& \multicolumn{1}{l|}{$s > \sqrt d$ (Lower bound)}        
& $s > \sqrt d$ (Upper bound)     
\rule[-2ex]{0pt}{6ex}
\\ 
\hline
$\mathcal{R}_{s,d}^{1/2}(\sigma,t)$   
& \multirow{3}{*}{$\sigma s^{\frac{1}{t}} \sqrt{\log\!\Big(1 \! + \! \dfrac{d}{s^2}\Big)}$}
& \multicolumn{1}{l|}{$\sigma s^{\frac{1}{t}} \log^{\frac{1}{2}-t}\log\Big(1\!+\! \dfrac{s^2}{d}\Big)$} 
& $\sigma s^{\frac{1}{t}} \log^{-\frac{1}{2}}\Big(1\!+\!\dfrac{s^2}{d}\Big)$ 
\rule[-2ex]{0pt}{6ex}
\\ \cline{1-1} \cline{3-4}
$\epsilon^*(s,t,\sigma^2I_d)$, $t <2$ 
&  
& \multicolumn{2}{c|}{ $\sigma d^{\frac{1}{4}} s^{\frac{1}{t}-\frac{1}{2}}$}       
\rule[-2ex]{-3.5pt}{6ex}
\\ \cline{1-1} \cline{3-4} 
$\epsilon^*(s,t,\sigma^2I_d)$, $t \geq 2$ 
&                                 
& \multicolumn{2}{c|}{$\sigma d^{\frac{1}{2}t}$} 
\rule[-2ex]{-3.5pt}{6ex}
\\ \hline
\end{tabular}

\vspace{3mm}

\begin{tabular}{|c|c|c|}
\hline
If $t \in 2\mathbb{N}^*$ 
& $s\leq \sqrt{d}$
& $s > \sqrt{d}$
\rule[-2ex]{0pt}{6ex}
\\ 
\hline
$\mathcal{R}_{s,d}^{\frac{1}{2}}(\sigma,t)$   
& \multirow{2}{*}{$\sigma s^{\frac{1}{t}} \sqrt{\log\Big(1\!+\!\dfrac{d}{s^2}\Big)}$  } 
& \multirow{2}{*}{$\sigma d^{\frac{1}{2t}}$ } 
\rule[-2ex]{0pt}{6ex}
\\ 
\cline{1-1}
$\epsilon^*(s,t,\sigma^2I_d)$ 
&                                  
&                                           
\rule[-2ex]{0pt}{6ex}
\\ \hline
\end{tabular}
    \caption{Comparison of the minimax estimation and testing problems across $t\in[1,\infty)$.}
    \label{tab:comparison_estim_test}
\end{table}

\vspace{2mm}

Our analysis of the isotropic case can be found in Subsection~\ref{subsec:isotropic} (in Subsection~\ref{subsec:isotropic}, the lower bounds for $s \leq C\sqrt{d}$ involve the term $\log\left(2\sqrt{d}/s\right)$, but we use the fact that $\log(1+d/s^2) \asymp \log\left(2\sqrt{d}/s\right)$ when $s \leq \sqrt{d}$).
Interestingly, for $s \leq \sqrt{d}$ and for any $t\geq 1$, $\mathcal{R}_{s,d}^{1/2}(\sigma,t)$    and $\epsilon^*(s,t,\sigma^2I_d)$ are always of the same order. 
This is reflected in the similarity between our test statistic $T_{sparse}$ and the estimator used in~\cite{collier2020estimation} in the sparse zone $s \leq \sqrt{d}$. 
We recall that in this regime, for some constant $C_t$:
$$ T_{sparse} = \sum_{j=1}^d \left(|X_j|^t - \alpha_j\right) \mathbb 1_{|X_j|^t > \uptau}, \quad \text{ where } \uptau = C_t\lambda^t + \nu^t/s \asymp \sigma^t \log^{t/2} \left(1\!+\!\frac{d}{s^2}\right).$$
In comparison, the estimator of $\|\theta\|_t$ used in~\cite{collier2020estimation} is as follows:
$$  \widehat T = (\widehat N_t)_+^{1/t} \quad \text{ where } \quad \widehat N_t = \sum_{j=1}^d \left(|X_j|^t - \alpha_j\right) \mathbb 1\left\{|X_j|^2 > 2 \sigma^2 \log\left(1\!+\!d/s^2\right)\right\}.$$

In the estimation problem, the constant $2$ in the indicator function is important, in order to balance the bias and variance of the estimator. 
In the testing problem, the constant $2$ can be replaced by any sufficiently large constant, only at the price of a larger constant in the upper bound. 
\vspace{2mm}

However, when $s > \sqrt{d}$, the square root estimation rate $\mathcal{R}_{s,d}^{1/2}(\sigma,t)$ is always at least as large as the rate of testing $\epsilon^*(s,t,\sigma^2 I_d)$.
The only case where the two quantities coincide for $s > \sqrt d$ is when $t$ is an even integer. 
In this case, the functional $\theta \mapsto \|\theta\|_t^t$ is sufficiently smooth to ensure that there exists unbiased estimators with much faster rates than for other $L^t$ norms, which are not smooth.


\subsection{Multinomial testing}\label{subsec:mult_testing}

The case of the $L^1$ separation is an interesting special case of our results and could be of independent interest. 
Indeed, in the context of discrete distributions, the $L^1$ distance is proportional to the total variation distance, and is therefore commonly used for multinomial testing~\cite{valiant2017automatic,berrett2020locally,lam2022local,gerber2022likelihood,balakrishnan2019hypothesis,canonne2020survey,CanonneTopicsDT2022}. 
In this Subsection, we set $\mathcal{P} = \big\{p = (p_1,\dots,p_d) \in \mathbb R_+^d ~\big|~ \sum_{j=1}^d p_j = 1\big\}$, and denote by $\mathcal{M}(n,p)$ the multinomial distribution with parameters $n \in \N^*$ and $p \in \mathcal{P}$. 
We also denote by $\operatorname{Unif}(d) = \big(\frac{1}{d}, \dots, \frac{1}{d}\big)$ the uniform distribution over $\{1,\dots,d\}$. 
We also fix a histogram $N \sim \mathcal{M}(n,p)$ for some $p \in \mathcal{P}$.
Multinomial testing against sparse alternatives has been considered in~\cite{bhattacharya2021sparse}. 
Namely, the authors considered the following global testing problem:
\begin{align}
    p = \operatorname{Unif}(d) \quad \text{ against } \quad H_1: \begin{cases}p \in \mathcal{P} \\ \|p-\operatorname{Unif}(d)\|_1\geq \epsilon \text{ and } \|p-\operatorname{Unif}(d)\|_0\leq s. \end{cases}\label{eq:global_mul_testing_pb}
\end{align}

In the asymptotic $s =d^{1-\alpha}$ for $\alpha \in (0,1)$, they proved that the minimax separation radius $\epsilon^* = \epsilon^*(s,n,d)$ for Problem~\eqref{eq:global_mul_testing_pb} scales as
\begin{align}
    \epsilon^* \asymp \frac{s}{d} \land  \begin{cases}\dfrac{\sqrt{s}}{\sqrt{n}d^{1/4}} & \text{ if } \alpha \leq \frac{1}{2}\\[10pt]
    s\sqrt{\dfrac{\log d}{nd}}& \text{ if } \alpha >\frac{1}{2}.\label{eq:global_mult_rate}
    \end{cases}
\end{align}

The term $s/d$ represents the impossibility regime: Any distribution $p \in \mathcal{P}$ such that $\|p-\operatorname{Unif}(d)\|_0 \leq s$ necessarily satisfies $\|p-\operatorname{Unif}(d)\|_1 \leq 2s/d$.
The second term interestingly bears similarity with our results in $L^1$ separation. 
Indeed, for $t=1$ and $\Sigma = \sigma^2 I_d$, Theorem~\ref{th:rate_ellt} yields (see Subsection~\ref{subsec:isotropic}):
\begin{align}
    \epsilon^*(s,1,\sigma^2 I_d) \asymp \begin{cases}\sigma d^{1/4} \sqrt{s} & \text{ if } s\geq C\sqrt{d}, \\
    \sigma s \sqrt{\log\left(2\sqrt{d}/s\right)}& \text{ otherwise.}\end{cases}\label{eq:rate_gaussian_ell1} 
\end{align}

\revision{Here, the choice of $\Sigma = \sigma^2 I_d$ is natural for comparing the two results. 
Indeed, under $H_0$, the rescaled histogram $N/n$ can be written as $(N_1/n, \dots, N_d/n)$ where each coordinate $N_j/n$ has a variance of $1/nd$. 
This justifies considering $\sigma_j^2 = 1/nd =: \sigma^2$ for any $j \in [d]$, hence $\Sigma = \sigma^2 I_d$. 
As we observe, the second term in the rate~\eqref{eq:global_mult_rate} is exactly analogous to~\eqref{eq:rate_gaussian_ell1} when $\sigma^2 = 1/nd$.}
This comparison therefore proves that the testing problem~\eqref{eq:global_mul_testing_pb} is either impossible or analogous to a Gaussian testing problem in $L^1$ separation, and that the correlation between the coordinates of $X$ do not affect the minimax rates. 
Further interplays between correlation and sparsity in signal detection have been thoroughly discussed in~\cite{kotekal2021minimax}, in the case of an isotropic covariance matrix and with Euclidean separation. 
The results of the present paper could therefore find natural applications to the local analog of Problem~\eqref{eq:global_mul_testing_pb}, which is left for future work.\\

In the absence of sparsity, the paper~\cite{valiant2017automatic} considered the following local testing problem in multinomials:
\begin{align}
    H_0: p = p_0 \quad \text{ against } \quad H_1: p \in \left\{ q\in \mathcal{P} ~\big|~ \|q-p_0\|_1\geq \epsilon\right\}, \label{eq:Valiant}
\end{align}
where $p_0$ is a fixed and known distribution in the class $\mathcal{P}$. 
The result is as follows: 
Assume without loss of generality that $p_0(1) \geq \dots \geq p_0(d)$, and, for any $\delta>0$, define $p_{0,-\delta}^{-\max} = (0,p_2,\dots,p_j,0,\dots,0)$ where $j = \max \{j ~|~ \sum_{i\geq j} p_0(i) >\delta\}$. Then the minimax separation radius for Problem~\eqref{eq:Valiant} is defined as the solution to the equation:
\begin{equation}
    C\epsilon = \sqrt{\frac{\|p_{0,-\epsilon}^{-\max}\|_{2/3}}{n}} + \frac{1}{n},\label{eq:Valiant_rate}
\end{equation}
for some absolute constant $C>0$. See~\cite{balakrishnan2019hypothesis} and~\cite{chhor2022sharp} Appendix D for the equivalence between~\eqref{eq:Valiant_rate} and the formulation of the results in~\cite{valiant2017automatic}. 
See also~\cite{blais2019distribution} for a further discussion about the relation between the $\ell^1$ and the $\ell^{2/3}$ norms. 
The $2/3$-norm exhibits some similarities with the Gaussian testing problem~\eqref{eq:testing_pb} for $t=1$ and $s=d$. 
Indeed, in light of our Theorem~\ref{th:rate_ellt}, we get $\epsilon^*(d,1,\Sigma) = \|\sigma\|_{4/3}$. 
Fixing $J = \max \{j ~|~ \sum_{i\geq j} p_0(i) >\epsilon\}$, the term $\sqrt{\|p_{0,-\epsilon}^{-\max}\|_{2/3}/n}$ is exactly analogous to $\|\widetilde \sigma\|_{4/3}$ where $\forall i \in \{2,\dots,J\} : \widetilde \sigma_j = \sqrt{p_0(j)/n}$ which is proportional to the standard deviation of $N_j/n$. 
We can take this analogy further by comparing with the results in~\cite{chhor2022sharp}, which considered the problem
\begin{align}
    H_0: p = p_0 \quad \text{ against } \quad H_1: p \in \left\{ q\in \mathcal{P} ~\big|~ \|q-p_0\|_t\geq \epsilon\right\}, \label{eq:mult_ellt}
\end{align}
for $t \in [1,2]$. The authors proved that, for a suitably defined index $I \in \{1,\dots, d\}$, the minimax separation radius for Problem~\eqref{eq:mult_ellt} scales as
\begin{align*}
    \epsilon^* = \sqrt{\frac{\|p_{0, \leq I}^{-\max}\|_r}{n}} + \frac{\|p_{0,>I}\|_1^{(2-t)/t}}{n^{2(t-1)/t}} + \frac{1}{n},
\end{align*}
where $r = \frac{2t}{4-t}$ and $p_{0,\leq I}^{-\max} = (p_0(2),\dots, p_0(I))$ and $p_{0,>I} = (0,\dots,0, p_0(I+1),\dots, p_0(d))$. 
The term $\sqrt{\|p_{0, \leq I}^{-\max}\|_r/n\,}$ can therefore be written as $\|\widetilde \sigma'\|_a$ where $\forall j \in \{2,\dots, I\}: \widetilde \sigma'_j = \sqrt{p_0(j)/n}$ which is once again proportional to the standard deviation of $N_j/n$.
The paper~\cite{waggoner2015lp} also highlighted this duality in the global case rather than in the local one. 
The paper~\cite{chhor2021goodness} considered an analogous version of Problem~\eqref{eq:mult_ellt}, for Hölder-continuous densities. 

\subsection{Adaptation to unknown sparsity}

\revision{
    In the paper~\cite{kotekal2021minimax}, the authors obtain adaptation to unknown sparsity $s$ by successively applying the tests corresponding to each possible level of sparsity $s\in[d]$ and rejecting $H_0$ if one of these tests rejects $H_0$.
    We conjecture that some techniques from~\cite{kotekal2021minimax}, in particular Lemmas 29 and 30, could be useful in our setting to analyze the chi-squared test statistics $T_{fdense}, T_{inter}$ and the hard-thresholding test $T_{sparse}$ when $t\leq 2$. 
    However, there are several differences compared with our setting. 
    First, the above tests are respectively defined over the regions $\{1,\dots,i_*\}$, $\{i_*\!+1, \dots, j_*\}$ and $\{j_*\!+1,\dots,d\}$, and the cut-offs $i_*$ and $j_*$ depend intricately on the sparsity level. 
    Second, Lemma 30 from~\cite{kotekal2021minimax} is only valid for centered and standard normal variables, whereas our truncation test $T_{sparse}$ involves heterogeneous centered normal random variables. 
    Third, for $t\in(2,\infty)$, the test statistic $T_{dense} = \sum_{j\leq j_*} |X_j|^t - \mathbb E_{H_0} |X_j|^t$ is not a chi-squared type test statistic and does not concentrate sub-exponentially, making it impossible for us to apply Lemma 29 from~\cite{kotekal2021minimax}. 
    Obtaining adaptation for $t \in [1,\infty)$ would therefore require a careful analysis of the concentration of our test statistics under $H_0$ and is left for future work.
    However, in the case of the $L^\infty$ separation, our procedure is automatically adaptive since in this case, our test is independent of $s$.}

\subsection{Connection with linear regression}

Suppose that we observe $(X,y)$ satisfying 
\begin{align}
    y = X^\top \theta + \xi \quad \text{ where } \quad \begin{cases}
    y = (y_1,\dots,y_n)^\top \in \R^n\\
    \xi = (\xi_1,\dots,\xi_n)^\top \in \R^n\\
    \text{$X=\left[X_1,\dots,X_n\right] \in \R^{d \times n}$ has rank $d$},
\end{cases} \label{eq:linear_regression}    
\end{align}
where for any $i \in [n]$, $\xi_i \overset{iid}{\sim} \mathcal{N}(0,\sigma) $ for $ \sigma>0$.
    Assume that the columns of the design $X:=\left[X_1,\dots,X_n\right]$ are orthogonal (not necessarily orthonormal), and for simplicity, assume that $n=d$. 
    Then $X$ can be decomposed as $X = U\Sigma^{1/2}$, where $U^\top U  = I_n$ and $\Sigma$ is diagonal. 
    Setting $y= (y_1,\dots,y_n)^\top$, $\xi = (\xi_1,\dots,\xi_n)^\top$ and $z = \left(XX^{\top}\right)^{-1}\!\! X y$, we exactly obtain that $z = \theta + w$, where $w = \left(XX^{\top}\right)^{-1}\!\! X \xi \sim \mathcal{N}(0,\sigma^2\Sigma^{-1})$ in the present case. 
    Moreover, the regressor $\theta$ is often assumed to be sparse in many applications in genetics, communications, and compressed sensing. 
    We refer the interested reader to \cite{arias2011global,ingster2010detection,carpentier2019minimax,carpentier2022estimation} for a discussion of the relevant motivations.
    Therefore, observing $z \sim \mathcal{N}(\theta,\sigma^2\Sigma^{-1})$ for sparse $\theta$ and diagonal $\Sigma$ allows us to exactly recover the setting of Problem~\eqref{eq:testing_pb}.
    Our results consequently open up interesting developments in this setting of practical interest, which is left for future inquiries. 

\subsection{Toward general covariance matrix $\Sigma$?}
The natural and important case of a general covariance matrix $\Sigma$, not necessarily diagonal, stands out as a highly non-trivial extension of our results, and goes far beyond the scope of this paper. 
Solving this problem would allow for important developments in sparse linear regression. 
Indeed, for the linear regression problem in the low dimensional regime $d\leq n$, assume that we observe $(X,y)$ distributed according to~\eqref{eq:linear_regression}, with $\xi \sim N(0,I_n)$. 
Moreover, we assume that $X$ has full rank but not necessarily orthogonal columns or rows. 
Then, writing $z = \left(XX^{\top}\right)^{-1}\!\! Xy\,$ and $w = \left(XX^{\top}\right)^{-1}\!\! X\xi$, we can rewrite our model as $z = \theta + w \sim \mathcal{N}(\theta, \Sigma)$, where $\Sigma = X^\top\!\big(X X^\top\big)^{-2}X$, linking the linear regression model with the Gaussian sequence model. 
Despite its importance for many practical applications, the optimal detection rate for sparse signals~$\theta$ in this setting remains unknown. 
This is due to the major technical challenges arising when combining sparsity with non-isotropic noise. 
One of them is the heterogeneity of the eigenvalues of $\Sigma$, addressed in this paper.
\vspace{2mm}

Another important challenge is that the sparsity basis (i.e. the canonical basis of $\R^d$) might not be aligned with the basis of eigenvectors of $\Sigma$. 
An important step toward that direction is the paper~\cite{kotekal2021minimax}, 
which thoroughly solves the case $\Sigma = (1-\gamma) I_d + \gamma \mathbb 1 \mathbb 1^\top$ where $\mathbb 1 = (1, \dots, 1)^\top \in \R^d$, for any $\gamma \in [0,1]$, and in Euclidean separation. 
This paper develops important techniques to tackle the non-alignment of the two bases, and elegantly decomposes the problem under study into separate isotropic detection problems. 
In this regard, however, it does directly address heteroscedasticity. 
We believe that the present paper captures fundamentally different phenomena from those uncovered in~\cite{kotekal2021minimax}, and hope that these two papers will constitute a basis to further explore sparse signal detection with arbitrary covariance matrix.

\section{Conclusion and future work}
In this paper, we solved the problem of sparse signal detection in the heteroscedastic Gaussian sequence model with a diagonal covariance matrix $\Sigma$, for any $L^t$ separation, $t\geq 1$. 
The present paper is a step toward addressing the much more ambitious case of general $\Sigma$, and it will be interesting to see how the present results can be combined with correlations in the noise. 
The $L^1$ separation distance also opens up interesting future directions for sparse testing in multinomials, which is left for future work.
Another avenue of research for future work is to translate the present results concerning the heteroscedastic Gaussian sequence model to the setting of testing sparse linear regression with non-isotropic design, generalizing~\cite{ingster2010detection}. \\

\revision{
\textit{Acknowledgements} We would like to thank the anonymous Associate Editor and Referee for many suggestions that helped improve the present manuscript, and the anonymous Referee for pointing out the explicit expression of $\epsilon^*(s,\infty,\Sigma)$ from Theorem~\ref{th:rate_Linfty_explicit}.}

\newpage
\appendix

\section{Generalities}

\subsection{Generalities for lower bounds}\label{subapp:generalities_LB}

\begin{lemma}\label{lem:generalities_LB}(Case of constant sparsity)
Assume that $s \leq C$ where $C>0$ is a constant depending only on $\eta$. 
Then it holds that $\epsilon^*(s,t,\Sigma) \geq c(\lambda + \nu)$ for some sufficiently small constant $c>0$, depending only on $\eta$, where $\lambda$ and $\nu$ are defined as in \eqref{eq:beta_t>2}, \eqref{eq:implicit_separation_ellt_>2} if $t \in [2,\infty)$ and as in \eqref{eq:def_lambda_ellt}, \eqref{eq:def_nu(lambda)_ellt} if $t \in [1,2]$.
\end{lemma}
Note that for $t=2$, two definitions are possible for the quantities $\lambda $ and $\nu$, but the conclusion of the lemma still holds in both cases.

\begin{proof}[Proof of Lemma~\ref{lem:generalities_LB}]
We first set some notation and distinguish between two cases.
\begin{enumerate}
    \item For $t \in [2,\infty)$, let $\beta$, $\nu$, $j_*$ be defined as in \eqref{eq:beta_t>2}, \eqref{eq:implicit_separation_ellt_>2}, \eqref{eq:explicit_separation_ellt_>2} and let 
    $$\pi_j = \dfrac{2}{s}\dfrac{\sigma_j^t e^{-\beta/\sigma_j^2}}{\sqrt{\sum_{j=1}^d \sigma_j^{2t} \exp\left(-\beta/\sigma_j^2\right)}},\quad \text{ for any $j \in [d]$.}$$
    \item For $t \in [1,2]$, define $\lambda$, $\nu$, $j_*$ as in~\eqref{eq:def_lambda_ellt}, \eqref{eq:def_nu(lambda)_ellt}, 
\eqref{eq:def_jstar_ellt}, and for all $j \in [d]$, define 
$$
    \pi_j = \dfrac{2}{s}\mathbb 1_{j\leq i_*} + \dfrac{2\sigma_j^4}{s\nu^{\,t} \lambda^{4-t} } \mathbb 1_{i_* < j \leq j_*} + \dfrac{2\sigma_j^t }{s\nu^t} e^{-\lambda^2/\sigma_j^2 + 1} \mathbb 1_{\sigma_j < \lambda}.
$$
\end{enumerate} 

For $t=2$, choose either one of the definitions above. 
In both cases, note that $\sum_{j=1}^d \pi_j=1$. 
This is a consequence of equation~\eqref{eq:beta_t>2} if $t \geq 2$ and of~\eqref{eq:def_lambda_ellt} for $t \in [1,2]$.

Now, let $p \sim \operatorname{Mult}((\pi_1,\dots,\pi_d),1)$  and define the random vector $\theta$ such that $\forall j \in [d]: \theta_j = c\lambda\mathbb 1_{j=p}$ where $c>0$ is a small enough constant depending only on $\eta$. 
We denote by $\Pi$ the prior distribution over $\theta$, defined such that $\mathbb P_\Pi(\forall j \in [d]: \theta_j = c\lambda \mathbb 1_{j=p}) = \pi_j$.
Let $\mathbb P_{prior} = \mathbb E_{\theta\sim \Pi}\left[\mathcal{N}(\theta,\Sigma)\right]$ denote the corresponding mixture of normal distributions $\mathcal{N}(\theta,\Sigma)$ where $\theta \sim \Pi$.

First, we have $\|\theta\|_t = c\lambda$ and $\|\theta\|_0 = 1 \leq s$ a.s., so that $\theta \in \Theta(c\lambda, s, t)$ a.s. 
We can now compute the $\chi^2$ divergence between our prior and the null distribution. 
To do so, note that conditionally on $p=j$, we have
\begin{align*}
    \forall x = (x_1,\dots, x_d) \in \R^d: \quad \mathbb P_{prior}(x\,|\,p=j) = \frac{e^{-(x_j-c\lambda)^2/2\sigma_j^2}}{\sqrt{2\pi \sigma_j^2}}\prod_{k \neq j} \frac{e^{-x_k^2/2\sigma_k^2}}{\sqrt{2\pi \sigma_k^2}}.
\end{align*}

Therefore,
\begin{align*}
    &1+\chi^2\left(\mathbb P_{prior} \,||\, \mathbb P_0\right) 
    = \int_{\R^d} \left(\frac{\mathbb P_{prior}}{\mathbb P_0}(x)\right)^2 \mathbb P_0(x) dx = \int_{\R^d} 
    \left[
        \sum_{j=1}^d \pi_j \exp\left\{-\frac{(x_j-c\lambda)^2 - x_j^2}{2\sigma_j^2}\right\}
    \right]^2 
    \mathbb P_0(x)dx\\
    & = \sum_{i\neq j} \pi_i \pi_j \int_{\R^2} \!\exp\left\{-\frac{(x_i-c\lambda)^2 - x_i^2}{2\sigma_i^2}-\frac{(x_j-c\lambda)^2 - x_j^2}{2\sigma_j^2}\right\} \left(4\pi^2 \sigma_i^2\sigma_j^2\right)^{-1/2}\!\! \exp\left\{-\frac{x_i^2}{2\sigma_i^2} - \frac{x_j^2}{2\sigma_j^2}\right\} dx_idx_j\\
    &\quad + \sum_{j=1}^d \pi_j^2 \int_\R \exp\left\{-\frac{(x_i-c\lambda)^2 - x_i^2}{\sigma_i^2}\right\} \left(2\pi\sigma_j^2\right)^{-1/2}\exp\left\{- \frac{x_j^2}{2\sigma_j^2}\right\} dx_j\\
    & = \sum_{i\neq j} \pi_i \pi_j + \sum_{j=1}^d \pi_j^2 \int_\R \left(2\pi\sigma_j^2\right)^{-1/2} e^{c^2\lambda^2/\sigma_j^2} \exp\left\{-\frac{(x_j-2c\lambda)^2}{2\sigma_i^2}\right\} dx_j\\
    & = \sum_{i\neq j} \pi_i \pi_j + \sum_{j=1}^d \pi_j^2 \exp\left\{c^2\lambda^2/\sigma_j^2\right\} = 1 + \sum_{j=1}^d \pi_j^2 \left(\exp\left\{c^2\lambda^2/\sigma_j^2\right\} -1\right) \quad \text{ recalling that } \sum_{j=1}^d \pi_j = 1\\
    & = 1 + \sum_{j\leq j_*} \pi_j^2 \left(e^{c^2}-1\right) + 4\sum_{j>j_*} \frac{\sigma_j^{2t} e^{-\lambda^2/\sigma_j^2}}{\nu^{2t}} \exp\left\{-\lambda^2/\sigma_j^2\right\}\left(\exp\left\{c^2\lambda^2/\sigma_j^2\right\} -1\right)\\
    & \leq 1+ 2c^2 + \sup_{\alpha>0}\left[e^{-(1-c^2)\alpha} - e^{-\alpha}\right],
\end{align*}
for $c$ small enough. 
Now it remains to prove that the family of functions $$\left\{f_c: \alpha \in \R_+\longmapsto e^{-(1-c^2)\alpha} - e^{-\alpha} ~\big|~ c>0\right\}$$ uniformly converges to $0$ when $c \downarrow 0$. 
\vspace{1mm}

Let $\delta>0$ and let $A>0$ be such that $e^{-A^2/2} \leq \delta$. 
We note that the family of functions $f_c$ is continuous over the compact set $[0,A]$, converges pointwise to $0$ when $c\downarrow 0$, which is a continuous function, and that $(f_c)_{c>0}$ decreases when $c\downarrow0$. Therefore, by Dini's theorem, $(f_c)_c$ uniformly converges to $0$ as $c\downarrow 0$ over $[0,A]$. 
Now, let $c>0$ such that $\forall \alpha \in [0,A]: |f_c(\alpha)|\leq \delta$, then by definition of $A$, we also have $\forall \alpha \geq A: |f_c(\alpha)|\leq \delta$. 
Combining the two guarantees proves the desired uniform convergence over $\R_+$.
\vspace{1mm}

For $c$ small enough, we therefore have $\operatorname{TV}(\mathbb P_{prior}, \mathbb P_0) \leq \sqrt{\chi^2(\mathbb P_{prior}\,||\, \mathbb P_0)} \leq c'$ (see e.g.~\cite{tsybakov2008introduction}) where $c'$ can be arbitrarily small. 
Therefore,
\begin{align}
    R^*(\epsilon,s,t,\Sigma) &= \inf_{\psi} R(\psi,\epsilon,s,t,\Sigma) = \inf_{\psi} \left\{\mathbb P_{0} \left(\psi = 1\right) + \sup\left\{\mathbb P_{\theta} \left(\psi = 0\right)~ \Big| ~ \theta \in \Theta(\epsilon,s,t)\right\}\right\}\nonumber\\
    & \geq \inf_{\psi} \left\{\mathbb P_{0} \left(\psi = 1\right) + \mathbb E_{\theta\sim \Pi}\left[ \mathbb P_{\theta} \left(\psi = 0\right)\right]\right\} = \inf_{\psi} \left\{1 - \mathbb P_{0} \left(\psi = 0\right) + \mathbb P_{prior}\left[ \psi = 0\right]\right\}\nonumber\\
    & = 1 - \sup_{A \in \mathcal{B}(\R^d)} \left\{\mathbb P_{prior}(A) - \mathbb P_0(A)\right\} = 1 - \operatorname{TV}(\mathbb P_0,\mathbb P_{prior})\geq 1-c'> \eta,\label{eq:conclusion_LB}
\end{align}
for $c'$ smaller than a quantity depending only on $\eta$. 
At the last line, we denoted by $\mathcal{B}(\R^d)$ the family of Borelian subsets of $\R^d$.
Since for any point $\theta$ of the prior,  we have $\|\theta\|_t = c\lambda$, then by definition of $\epsilon^*(s,t,\Sigma)$, we have proved that $\epsilon^*(s,t,\Sigma)\geq c\lambda$. 
\vspace{1mm}

We now prove that $\epsilon^*(s,t,\Sigma)\geq c\nu$. 
Note that if $\lambda + \nu \asymp \lambda$, then the result is clear. 
Otherwise, we will show that the case $\nu \gg \lambda$ can only arise if $\nu \leq C \sigma_1$, and we will show that $\epsilon^*(s,t,\Sigma)\geq c\sigma_1$.
\vspace{1mm}

If $t \in [1,2]$, then recalling the notation of Lemma~\ref{lem:relations_between_contributions}, we have $\lambda + \nu \asymp \lambda + \nu_{fdense}$ since $s\leq C$. 
If $i_* = 0$, then $\lambda + \nu \asymp \lambda$, and we have proved $\epsilon^*(s,t,\Sigma)\geq c \lambda \asymp \lambda + \nu$. 
Otherwise, we note that
$$ 1\geq \sum_{j\leq i_*} \pi_j = \frac{2i_*}{s}, \quad \text{ so that } i_* \leq s \leq C.$$

Therefore, $\nu_{fdense} = \|\sigma_{\leq i_*}\|_a \leq \sigma_1 i_*^{1/a} \leq C\sigma_1.$ 
Moreover, if $t\geq 2$, then since $s\leq C$, we have
\begin{align*}
    &\nu^{2t} = \sum_{j=1}^d \sigma_j^{2t} e^{-\lambda^2/\sigma_j^2} =   \sum_{j=1}^d \frac{s\pi_j}{2}\sigma_j^t \nu^t \leq C
    \sigma_1^{t}\nu^t,
    \text{ so that } \nu \leq C\sigma_1.
\end{align*}

We now set a different prior over the parameter space. Namely, we define a random vector $\theta'$ such that $\forall j \in \{2,\dots,d\}: \theta'_j=0$ and $\theta'_1 = b \cdot c \sigma_1$ for a sufficiently small constant $c>0$ depending only on $\eta$ and where $b \sim \operatorname{Ber}(1/2)$. 
We let $\mathbb P'_{prior}$ denote the probability distribution corresponding to this mixture. 
Denoting by $\operatorname{KL}$ the KL-divergence between probability distributions (see for instance~\cite{tsybakov2008introduction}), we can compute $\operatorname{KL}(\mathbb P'_{prior} \,||\, \mathbb P_0)$ as follows. 
We set $\theta^{(0)} = 0$ and $\theta^{(1)} = \left(c\sigma_1,0\dots, 0\right)$.
Note that $\mathbb P'_{prior} = \frac{1}{2} \mathbb P_{\theta_0} + \frac{1}{2} \mathbb P_{\theta_1}$ and by convexity of the KL-divergence:
\begin{align*}
    \operatorname{KL}(\mathbb P'_{prior} \,||\, \mathbb P_0) \leq \frac{1}{2} \operatorname{KL}(\mathbb P_{\theta^{(1)}} \,||\, \mathbb P_{\theta^{(0)}}) \leq \operatorname{KL}\big(\mathcal{N}(0,\sigma_1^2) \,||\, \mathcal{N}(c \sigma_1,\sigma_1^2)\big) \leq c^2,
\end{align*}
so that $\operatorname{TV}(\mathbb P'_{prior} , \mathbb P_0) \leq \sqrt{\operatorname{KL}(\mathbb P'_{prior} \,||\, \mathbb P_0)} \leq c$. 
By the same argument as in~\eqref{eq:conclusion_LB}, we can conclude that $\epsilon^*(s,t,\Sigma) \geq c\sigma_1 \geq c\nu$ so that $\epsilon^*(s,t,\Sigma) \geq c (\lambda + \nu)$. 
The proof is complete.

\end{proof}

\subsection{Generalities for upper bounds}
\begin{lemma}\label{lem:analysis_Tsparse}(Analysis of $T_{sparse}$)
Let $\lambda$, $\nu$, $j_*$, $\uptau$, $\rho$ be defined as in ~\eqref{eq:beta_t>2}, \eqref{eq:implicit_separation_ellt_>2}, \eqref{eq:explicit_separation_ellt_>2}, \eqref{eq:def_nu_>2}, \eqref{eq:def_rho_>2} for $t \geq 2$ and as in~\eqref{eq:def_lambda_ellt}, \eqref{eq:def_nu(lambda)_ellt}, 
\eqref{eq:def_jstar_ellt}, \eqref{eq:def_tau_ellt}, \eqref{eq:def_rho_ellt} for $t \in [1,2]$. 
Let $T_{sparse}$ be defined as in~\eqref{eq:def_test_stat_>2} for $t \geq 2$ and \eqref{eq:def_Tsparse_ellt} for $t \in [1,2]$.
Then, there exists a constant $C_0$ depending only on $t$, such that when $\|\theta\|_0 \leq s$, we have

\begin{center}
    \begin{tabular}{|c|c|c|}
    \hline
    & Under $H_0$ & When $\left\|\theta_{>j_*}\right\|_t^t \geq 4\Cbar \rho$ \rule[-2ex]{0pt}{6ex} \\
    \hline
    $\mathbb E^2 T_{sparse}$ & $= 0$ & $\geq \Cbar\rho /\,2^{t+2}$ \rule[-2ex]{0pt}{6ex}\\
    \hline
    $\mathbb V T_{sparse}$ & $ \leq C_0 \rho$ & $\leq c\, \mathbb E_\theta^2 \left[T_{sparse}\right]. $\rule[-2ex]{0pt}{6ex}\\
    \hline
\end{tabular}
\end{center}
In the last cell, $c$ is a constant depending only on $\eta$ and $t$, that can be made arbitrarily small provided that $\Cbar = \Cbar(\eta,t)$ is large enough, and $\Cbar$ can be chosen independently of $C_0$. 
\end{lemma}

\begin{proof}

\begin{enumerate}
    \item Under $H_0$, $T_{sparse}$ is centered by definition. As for the variance, by independence, we get:
    \begin{align*}
        \mathbb V \big[T_{sparse}\big] & = \sum_{j>j_*} \mathbb E \left[\left(|X_j|^t - \alpha_j\right)^2 \mathbb 1\left\{|X_j|^t > \uptau\right\}\right] \\
        &= \sum_{j >j_*} \mathbb E \left[|X_j|^{2t} \, \mathbb 1\left\{|X_j|^t > \uptau\right\}\right] - \alpha_j^2 \,\mathbb P\left[|X_j|^t >\uptau\right]\\
        & \leq \sum_{j>j_*} \mathbb E \left[|X_j|^{2t} \, \mathbb 1\left\{|X_j|^t > \uptau\right\}\right] \leq C \rho^2 \quad \text{ by Lemma~\ref{lem:fourth_moment_ellt}}.
    \end{align*}
    \item Let $\|\theta_{>j_*}\|_t^t \geq 4\Cbar\rho$ and let 
    \begin{align*}
        I = \Big\{j >j_* ~\big|~ |\theta_j|^t \geq \Cbar \uptau \Big\}.
    \end{align*}
    We write $\theta_I = (\theta_j)_{j\in I}$ and $\theta_{>j_*} = (\theta_j)_{j>j_*}$.
    Then we have $\big\|\theta_I\big\|_t^t \geq \frac{1}{2}\big\|\theta_{j>j_*}\big\|_t^t \geq \Cbar\rho$. Indeed, 
    \begin{align*}
        \sum_{\substack{j>j_*\\ j\notin I}} |\theta_j|^t \leq s \Cbar\uptau \leq 2 \Cbar \rho \leq \frac{1}{2} \big\|\theta_{j>j_*}\big\|_t^t, \quad \text{ so that } \quad \big\|\theta_I\big\|_t^t = \big\|\theta_{>j_*}\big\|_t^t - \sum_{\substack{j>j_*\\ j\notin I}} \big|\theta_j|^t \geq \frac{1}{2}\|\theta_{j>j_*}\big\|_t^t.
    \end{align*}
    Throughout the proof, we shall use the fact that for any $j\in I$ and any $\alpha \geq 1$, we have
    \begin{align}
        \E |\xi_j|^\alpha \leq C(\alpha) \sigma_j^\alpha \leq C(\alpha) \lambda^\alpha \leq C(\alpha) \uptau^{\alpha/t} \leq \frac{C(\alpha)}{\Cbar^{\alpha/t}} |\theta_j|^\alpha \quad \text{ by definition of } I.\label{eq:property_I}
    \end{align}
    For the first inequality, see for example~\cite{winkelbauer2012moments}. Now, we can now bound from below $\mathbb E T$ as follows:
    \begin{itemize}
        \item Fix any $j \in I$. 
        Noting that $|\cdot|^t$ is convex, we have for any $u \in \R: |u|^t + |1-u|^t \geq 2^{1-t}$, so that for any $a,b \in \R$, we get $|a+b|^t \geq 2^{1-t}|a|^t-|b|^t$ (set $u = -b/a$ when $a\neq 0$). 
        We therefore get
    \begin{align}
        \mathbb E\left[(|X_j|^t - \alpha_j) \mathbb 1\left\{|X_j|^t > \uptau\right\}\right] 
        &\geq \mathbb E\left[\left(2^{1-t}|\theta_j|^t - |\xi_j|^t - \alpha_j\right) \mathbb 1\left\{|X_j|^t > \uptau\right\}\right]\nonumber\\
        &\geq 
            \mathbb E\left[2^{1-t}|\theta_j|^t \mathbb 1\left\{|X_j|^t > \uptau\right\}\right] - \mathbb E |\xi_j|^t - \alpha_j \nonumber\\
        & \geq 
            2^{-t} |\theta_j|^t - C|\sigma_j|^t - \alpha_j ~~ \text{ by Lemma~\ref{lem:prob_Xj2>2lambda_ellt}, for $\Cbar$ large enough}\nonumber\\
        & \geq 
            2^{-t}|\theta_j|^t - (C+\CL{\ref{lem:alpha_j_leq_2lambda_ellt}}) \uptau, \label{eq:expec_I_ellt}
    \end{align}
    where $\CL{\ref{lem:alpha_j_leq_2lambda_ellt}}$ is the constant from Lemma~\ref{lem:alpha_j_leq_2lambda_ellt}. 
    \item Fix any $j>j_*$ such that $j \in S \setminus I$. We have
    \begin{align}
        \mathbb E \left[\left(|X_j|^t - \alpha_j\right) \mathbb 1\left\{|X_j|^t > \uptau\right\}\right] \geq -\alpha_j \geq - \CL{\ref{lem:alpha_j_leq_2lambda_ellt}} \uptau. \label{eq:expect_I-S_ellt}
    \end{align}
    \item Fix any $j>j_*$ such that $j \notin S$. Then by definition of $\alpha_j$,
    \begin{align}
        \mathbb E \left[\left(|X_j|^t - \alpha_j\right) \mathbb 1\left\{|X_j|^t > \uptau\right\}\right] =0. \label{eq:expect_not_S_ellt}
    \end{align}
    \end{itemize}
\end{enumerate}

Combining equations~\eqref{eq:expec_I_ellt}, \eqref{eq:expect_I-S_ellt} and \eqref{eq:expect_not_S_ellt}, we can conclude that
\begin{align*}
    \mathbb E T_{sparse} &\geq 2^{-t}\|\theta_I\|_t^t - (1+2\CL{\ref{lem:alpha_j_leq_2lambda_ellt}}) s \uptau \geq 2^{-t-1}\|\theta_I\|_t^t ~~\text{ choosing $\Cbar$ large enough, by definition of $I$},
\end{align*}
so that $\mathbb E T_{sparse} \geq 2^{-t-2}\big\|\theta_{j>j_*}\big\|_t^t$, which proves the claim. \\

We now move to the variance term. 
Again, there are three cases.
\begin{enumerate}
    \item If $j \in S \setminus I$, then, using $\E |X_j|^{2t} \leq C|\theta_j|^{2t} + C \sigma_j^{2t}$ and Lemma~\ref{lem:alpha_j_leq_2lambda_ellt}, we get
    \begin{align}
        \mathbb V \left[\left(|X_j|^t - \alpha_j\right) \mathbb 1\left\{|X_j|^t > \uptau\right\}\right] &\leq \mathbb E \left[\left(|X_j|^t - \alpha_j\right)^2 \mathbb 1\left\{|X_j|^t > \uptau\right\}\right] \leq 2 \mathbb E |X_j|^{2t} + 2 \alpha_j^2\nonumber\\
    &   \leq C \theta_j^{2t} + (C + C\CL{\ref{lem:alpha_j_leq_2lambda_ellt}}^2) \uptau^2 \leq C\left(1+\Cbar+\CL{\ref{lem:alpha_j_leq_2lambda_ellt}}^2\right) \uptau^2.\label{eq:variance_S-I_ellt}
    \end{align}
    \item If $j \notin S$: we are back to the analysis of the variance under $H_0$, which allows us to directly conclude that
    \begin{align}
        \sum_{j\notin S}\mathbb V \left[\left(|X_j|^t - \alpha_j\right) \mathbb 1\left\{|X_j|^t > \uptau\right\}\right] &\leq  \sum_{j\notin S}\mathbb E \left[X_j^{2t} \, \mathbb 1\left\{|X_j|^t > \uptau\right\}\right] \leq C \rho^2 \quad \text{ by Lemma~\ref{lem:fourth_moment_ellt}}.\label{eq:variance_-S_ellt}
    \end{align}
    \item If $j \in I$, we prove that there exists a small constant $c>0$ such that $\mathbb V \left[\left(|X_j|^t - \alpha_j\right) \zeta_j\right] \leq c \big|\theta_j\big|^{2t}$.
    We define the random variable $\zeta_j = \mathbb 1\left\{|X_j|^t \geq \uptau\right\}$ and $q_j = \mathbb E \zeta_j$.
Note that $\mathbb E\left[|X_j|^t\right] \leq 2^{t-1} \mathbb E\left[|\theta_j|^t + |\xi_j|^t\right] \leq C \left(|\theta_j|^t + \sigma_j^t\right)$. 
Therefore, we have
\begin{align}
    \hspace{-10mm}&\mathbb V \left[\left(|X_j|^t - \alpha_j\right) \zeta_j\right]  =  \mathbb E \left[\left(|X_j|^t - \alpha_j\right)^2 \mathbb \zeta_j\right] - \mathbb E^2 \left[\left(|X_j|^t - \alpha_j\right) \mathbb \zeta_j\right]\nonumber\\
    & = \left\{\mathbb E\left[|X_j|^{2t} \zeta_j\right] - 2 \alpha_j \mathbb E\left[|X_j|^t \zeta_j\right] + \alpha_j^2 q_j\right\} 
    - 
    \left\{\mathbb E^2\left[|X_j|^t \zeta_j\right] - 2 \alpha_j q_j \mathbb E\left[|X_j|^t \zeta_j\right] + \alpha_j^2 q_j^2\right\}\nonumber\\
    & \leq \left\{\mathbb E\left[|X_j|^{2t} \zeta_j\right] - \mathbb E^2\left[|X_j|^t \zeta_j\right] \right\} + \alpha_j^2 q_j + 2 \alpha_j q_j \mathbb E\left[|X_j|^t \zeta_j\right] \nonumber \\
    & \leq \left\{\mathbb E\left[|X_j|^{2t} \zeta_j\right] - \mathbb E^2\left[|X_j|^t \zeta_j\right] \right\} + \CL{\ref{lem:alpha_j_leq_2lambda_ellt}}^2 \uptau^2 + 2 \CL{\ref{lem:alpha_j_leq_2lambda_ellt}} \uptau \; \mathbb E\left[|X_j|^t\right] 
    \nonumber\\
    & \leq \left\{\mathbb E\left[|X_j|^{2t} \zeta_j\right] - \mathbb E^2\left[|X_j|^t \zeta_j\right] \right\} + C \uptau^2 + C \uptau |\theta_j|^t
    \nonumber\\
    & \leq 
    \left\{\mathbb E\left[|X_j|^{2t} \zeta_j\right] - \mathbb E^2\left[|X_j|^t \zeta_j\right] \right\} 
    +
    \frac{C}{\Cbar} |\theta_j|^{2t}.\label{eq:var_Tsparse_ellt}
\end{align}

To control the term $\left\{\mathbb E\left[|X_j|^{2t} \zeta_j\right] - \mathbb E^2\left[|X_j|^t \zeta_j\right] \right\}$, there are two cases.

\begin{itemize}
    \item If $t \leq 2$, then we have
    \begin{align}
        \mathbb E\left[|X_j|^{2t} \zeta_j\right]& \leq \mathbb E\left[X_j^4 \right]^{t/2} = \left[\theta_j^4 + 6 \theta_j^2 \sigma_j^2 + 3\sigma_j^4 \right]^{t/2} \nonumber\\
        & \leq 
        \left[\theta_j^4 + 6 \theta_j^2 \cdot \frac{\theta_j^2}{\Cbar^{2/t}} + 3\frac{\theta_j^4}{\Cbar^{4/t}} \right]^{t/2}
        \quad 
        \text{ since over $I$ we have $\sigma_j \leq \lambda \leq \uptau^{1/t} \leq \frac{|\theta_j|}{\Cbar^{1/t}}$}\nonumber\\
        &
        \leq 
        \theta_j^{2t} \left( 1 +  \frac{9}{\Cbar^{2/t}} \right)^{t/2}.
        \label{eq:moment_2t_ellt}
    \end{align}
    Now, we define a large constant $C^*$, and in what follows, we first take $C^*$ sufficiently large before taking $\Cbar \in [C^*,+\infty)$ sufficiently large for fixed $C^*$.
    Defining $z_j = \mathbb 1 \left\{|\xi_j|\leq C^* \sigma_j\right\}$, we note that 
    \begin{align}
        \forall j \in I: \quad \E \left[|\xi_j| z_j \zeta_j\right] \leq C^*\sigma_j \E \left[ z_j \zeta_j\right] \leq \frac{C^*}{\Cbar} |\theta_j| \E\left[z_j \zeta_j\right] \leq |\theta_j| \E\left[z_j \zeta_j\right] \text{ if } \Cbar \geq C^*.\label{eq:C*}
    \end{align}
    Therefore, we get
    \begin{align}
        \mathbb E\left[|X_j|^t \zeta_j\right] 
        &\geq 
        \mathbb E^t\left[|X_j| \zeta_j\right] 
        \geq 
        \mathbb E^t\left[\Big||\theta_j| - |\xi_j| \Big| \, \cdot \zeta_j\right] 
        \geq 
        \left|\mathbb E\Big[\big(|\theta_j| - |\xi_j| \big)z_j \zeta_j \Big] \right|^t
        \nonumber \\
        & = 
            \mathbb E^t\Big[\big(|\theta_j| - |\xi_j| \big)z_j \zeta_j \Big] \quad \text{ by equation \eqref{eq:C*}} \nonumber\\
        & \geq
            \Big\{|\theta_j| \,\mathbb E (z_j \zeta_j) - C^* \sigma_j \mathbb E (z_j \zeta_j)\Big\}^t \\
        &  \geq 
            |\theta_j|^t \left(1-\frac{C^*}{\Cbar}\right)^t \mathbb P\left(|\xi_j|\leq C^* \sigma_j \text{ and } |X_j|^t \geq \uptau\right)^t \nonumber\\
        & \geq 
            |\theta_j|^t \left(1-\frac{C^*}{\Cbar}\right)^t \mathbb P\left( |X_j|^t \geq \uptau ~ \Big| ~ |\xi_j|\leq C^* \sigma_j \right)^t \mathbb P\left(|\xi_j|\leq C^* \sigma_j\right)^t
        \nonumber \\
        & =
            |\theta_j|^t \left(1-\frac{C^*}{\Cbar}\right)^t \mathbb P\left(|\xi_j|\leq C^* \sigma_j\right)^t \quad \text{ if $\Cbar$ is large enough for fixed $C^*$.}\label{eq:moment_t_ellt}
    \end{align}
    Note that the constant $\left(1-\frac{C^*}{\Cbar}\right)^t \mathbb P\left(|\xi_j|\leq C^* \sigma_j\right)^t$ can be made arbitrarily close to $1$ provided that we choose successively $C^*$ and $\Cbar$ large enough. 
    Now, combining equations~\eqref{eq:var_Tsparse_ellt}, \eqref{eq:moment_2t_ellt} and \eqref{eq:moment_t_ellt}, we get
    \begin{align}
        \mathbb V \left[\left(|X_j|^t - \alpha_j\right) \zeta_j\right] 
        &\leq 
        \theta_j^{2t} \left( 1 +  \frac{9}{\Cbar^{2/t}} \right)^{t/2}
        - |\theta_j|^{2t} \left(1-\frac{C^*}{\Cbar}\right)^{2t} \mathbb P\left(|\xi_j|\leq C^* \sigma_j\right)^{2t} 
        + \frac{C}{\Cbar} |\theta_j|^{2t}
        \nonumber
        \\
        & \leq c|\theta_j|^{2t},\label{eq:variance_I_ellt_<2}
    \end{align}
    where the constant $c$ can be made arbitrarily small by successively choosing $C^*$ and $\Cbar$ large enough. 
    
    \item If $t \geq 2$, we use the following Taylor expansion: For any $\alpha \geq 2$, for any $\theta_j, \xi_j\in \R$, there exists $\xi_j' \in (0,\xi_j)$ such that
    \begin{align}
        \big|\theta_j + \xi_j\big|^\alpha &= |\theta_j|^\alpha + z_j^{(\alpha)} 
        \nonumber\\
        \text{where } z_j^{(\alpha)} :&=
        \alpha \big|\theta_j^{\alpha-1} \xi_j\big| \operatorname{sign}(\theta_j) + \frac{\alpha (\alpha-1)}{2} \xi_j^2 \big|\theta_j + \xi'_j\big|^{\alpha-2}.\label{eq:taylor}
    \end{align}
    For $j\in I$, we can bound $\mathbb E \big|z_j^{(\alpha)}\big|$ from above as follows 
    \begin{align}
        \mathbb E \big|z_j^{(\alpha)}\big| \leq \alpha C \Cbar^{-1/t} |\theta_j|^\alpha + \frac{\alpha (\alpha-1)}{2} C \left\{ \frac{|\theta_j|^{\alpha}}{\Cbar^{2/t}} + \frac{|\theta_j|^\alpha}{\Cbar^{\alpha/t}} \right\} \leq C \frac{|\theta_j|^\alpha}{\Cbar^{1/t}},\label{eq:UB_zj_alpha}
    \end{align}
    for some constant $C$ depending on $\alpha$.
    We now have
    \begin{align*}
        &\mathbb E\left[|X_j|^{2t} \zeta_j\right] \leq \mathbb E\left[\big|\theta_j\big|^{2t} + \big|z_j^{(2t)}\big| \right] \leq |\theta_j|^{2t} \left(1 +  \frac{C}{\Cbar^{1/t}} \right),
        \nonumber \\
        &
        \mathbb E\left[|X_j|^t \zeta_j\right] \geq \mathbb E\left[\left(\big|\theta_j\big|^{t} - \big|z_j^{(t)}\big|\right) \zeta_j \right] \geq |\theta_j|^{t} q_j - \mathbb E \left[\big|z_j^{(t)}\big|\right] \geq |\theta_j|^{t} \left(q_j -  \frac{C}{\Cbar^{1/t}} \right) \quad \text{ by \eqref{eq:UB_zj_alpha}}.
    \end{align*}
    Combining the above two relations, we obtain
    \begin{align}
        \mathbb E\left[|X_j|^{2t} \zeta_j\right] - \mathbb E^2\left[|X_j|^t \zeta_j\right] \leq |\theta_j|^{2t} \left(1 +  \frac{C}{\Cbar^{1/t}} \right) - |\theta_j|^{2t} \left(q_j -  \frac{C}{\Cbar^{1/t}} \right)^2 \leq c \big|\theta_j\big|^{2t},\label{eq:variance_I_ellt>2}
    \end{align}
    where the constant $c$ can be made arbitrarily small by choosing $\Cbar$ large enough. 
\end{itemize}

    Now, combining equations~\eqref{eq:variance_S-I_ellt}, \eqref{eq:variance_-S_ellt} and \eqref{eq:variance_I_ellt_<2} for $t\leq 2$ or \eqref{eq:variance_I_ellt>2} for $t \geq 2$, we get
    \begin{align*}
        \mathbb V T &\leq \sum_{\substack{j>j_* \\ j \in S\setminus I} } C \uptau^2 + C\rho^2 + \sum_{j \in I} c \theta_j^{2t} \leq C \,s \uptau^2 + C\rho^2 + c \|\theta_{I}\|_{2t}^{2t}\\
        & \leq C \rho^2 + c \|\theta_{>j_*}\|_t^{2t} \quad\quad \text{ since $\|\cdot\|_{2t} \leq \|\cdot \|_t$}\\
        & \leq \frac{C}{4\Cbar}\left\|\theta_{>j_*}\right\|_t^t + c \|\theta_{>j_*}\|_t^{2t} \quad \text{ by assumption}\\
        & \leq 2c \|\theta_{>j_*}\|_t^{2t} , \quad\quad\quad \text{ provided that $\Cbar$ is large enough}\\
        & \leq 8c \, \mathbb E^2\! \left(T_{sparse}\right).
    \end{align*}
\end{enumerate}
\end{proof}

\begin{lemma}\label{lem:gaussian_tail_moments_ellt}
Let $Y \sim \mathcal{N}(0,1)$ and let $ \alpha \geq 1 $ and $x>0$ such that $x^2 \geq 2(\alpha-1) $. 
Then we have $$\mathbb E \big[ |Y|^\alpha \mathbb 1_{|Y|\geq x}\big] \leq 4 x^{\alpha -1} e^{-x^2/2}.$$
\end{lemma}

\begin{proof}[Proof of Proposition~\ref{lem:gaussian_tail_moments_ellt}]
By integration by parts, we have 
\begin{align*}
    \int_x^\infty y^\alpha e^{-y^2/2} dy &= x^{\alpha-1} e^{-x^2/2} + (\alpha\!-\!1)\! \int_x^{\infty} \! y^{\alpha-2} e^{-y^2/2} dy\\
    &\leq x^{\alpha-1} e^{-x^2/2} + \frac{1}{2} \int_x^\infty y^\alpha e^{-y^2/2} dy ~~ \text{using that $y^2 \geq \alpha-1$ over $[x,+\infty)$},\\
    \text{so that } \quad 
        \mathbb E \big[ |Y|^\alpha \mathbb 1_{|Y|\geq x}\big] &= 2 \int_x^\infty y^\alpha e^{-y^2/2} dy \leq 4x^{\alpha-1} e^{-x^2/2}.
\end{align*}
\end{proof}

\begin{lemma}\label{lem:fourth_moment_ellt}
Let $j_*$, $\uptau$ and $\rho$ be respectively defined as in \eqref{eq:explicit_separation_ellt_>2}, \eqref{eq:def_nu_>2} and \eqref{eq:def_rho_>2} for $t \geq 2$ and as in \eqref{eq:def_jstar_ellt}, \eqref{eq:def_tau_ellt} and \eqref{eq:def_rho_ellt} for $t \in [1,2]$. 
Assume that for any $j >j_*$, we have $X_j \sim \mathcal{N}(0,\sigma_j^2)$. Then there exists a constant $C>0$ depending only on $t$, such that
$$ \sum_{j>j_*} \mathbb E\left[|X_j|^{2t} \mathbb 1\left\{|X_j|^t > \uptau\right\}\right] \leq C \rho^2.$$
\end{lemma}

\begin{proof}[Proof of Lemma~\ref{lem:fourth_moment_ellt}]
We apply Lemma~\ref{lem:gaussian_tail_moments_ellt} with $\alpha = 2t$ and $x = C_t^{1/t}\lambda / \sigma_j \geq C_t^{1/t}$, which satisfies $x^2 \geq 2(\alpha-1)$. 
\begin{align*}
    & ~ 
        \sum_{j>j_*} \mathbb E
        \left[|X_j|^{2t} \mathbb 1\left\{|X_j|^t > \uptau\right\}\right] 
        = \sum_{j>j_*} \sigma_j^{2t} \, \mathbb E \left[|Y|^{2t} \, \mathbb 1 \left\{|Y|^t >\uptau/\sigma_j^t\right\}\right] ~~ \text{ where } Y \sim \mathcal{N}(0,1)\\
    \leq & 
        ~\sum_{j>j_*} \sigma_j^{2t} \cdot 4 v_j^{2t-1}e^{-v_j^2/2} = 4 \sum_{j>j_*} \uptau^{2-1/t} \sigma_j e^{-v_j^2/2} \quad \text{ by Lemma~\ref{lem:gaussian_tail_moments_ellt}, where $v_j = \frac{\uptau^{1/t}}{\sigma_j}$}\\
    \leq & C \rho^2 \quad \text{ by Lemma~\ref{lem:relation_tau_rho}}.
\end{align*}
\end{proof}

\begin{lemma}\label{lem:relation_tau_rho}
    Let $j_*$, $\uptau$ and $\rho$ be respectively defined as in \eqref{eq:explicit_separation_ellt_>2}, \eqref{eq:def_nu_>2} and \eqref{eq:def_rho_>2} for $t \geq 2$ and as in \eqref{eq:def_jstar_ellt}, \eqref{eq:def_tau_ellt} and \eqref{eq:def_rho_ellt} for $t \in [1,2]$. 
    Then there exists a constant $C>0$ depending only on $t$, such that
    $$\sum_{j>j_*} \uptau^{2-1/t} \sigma_j e^{-v_j^2/2} \leq C\rho^2 \quad\quad \text{ where  $v_j = \uptau^{1/t}/\sigma_j$} .$$
\end{lemma}

\begin{proof}
Recalling that $C_t = (4t)^{t} \geq 2^t$, we have
\begin{align}
    v_j^2 = \frac{\left(C_t \lambda^t + \nu^t/s\right)^{2/t}}{\sigma_j^2} \geq \max\left(C_t^{2/t} \frac{\lambda^2}{\sigma_j^2}, \frac{\nu^2}{s^{2/t} \sigma_j^2}\right) \geq \frac{1}{2} \left[C_t^{2/t} \frac{\lambda^2}{\sigma_j^2} + \frac{\nu^2}{s^{2/t} \sigma_j^2}\right] \geq \frac{2\lambda^2}{\sigma_j^2} + \frac{\nu^2}{2s^{2/t} \sigma_j^2}.\label{eq:LB_vj}
\end{align}
    We have
    \begin{align*}
        &\sum_{j>j_*} \uptau^{2-1/t} \sigma_j e^{-v_j^2/2} \leq \sum_{j>j_*} \uptau^{2-1/t} \sigma_j \exp\left(-\frac{\lambda^2}{\sigma_j^2} + \frac{\nu^2}{4s^{2/t} \sigma_j^2}\right) \\
    \leq & 
        \uptau^{2-1/t} \sum_{j>j_*} \sigma_j \exp\left(-\frac{\lambda^2}{\sigma_j^2}\right) \left(\frac{\sigma_j s^{1/t}}{\nu}\right)^{t-1} \quad \text{ using that } e^{-x^2/4} \leq C x^{1-t} \text{ for some } C = C(t)\\
    \leq & 
        C \uptau^{2-1/t} \frac{s^{1-1/t}}{\nu^{t-1}} \sum_{j>j_*}  \sigma_j^t \exp\left(-\frac{\lambda^2}{\sigma_j^2}\right) \leq C \uptau^{2-1/t} \frac{s^{1-1/t}}{\nu^{t-1}} s\nu^t \quad \text{ by equations~\eqref{eq:beta_t>2}, \eqref{eq:implicit_separation_ellt_>2} and \eqref{eq:def_f_ellt},~\eqref{eq:def_lambda_ellt}}\\
    = & 
        C (\uptau s)^2 \frac{\nu}{(\uptau s)^{1/t}} \leq C \uptau^2 s^2 \leq C\rho^2.
    \end{align*}
\end{proof}

\begin{lemma}\label{lem:prob_Xj2>2lambda_ellt}
Assume that, for some $j \in \{1,\dots,d\}$ and that for some large enough constant $C\geq 1$, and some arbitrary real number $\tau>0$ satisfying $\sigma_j^t \leq \uptau$, we have $|\theta_j|^t \geq C^t\uptau$ and let $X_j \sim \mathcal{N}(\theta_j, \sigma_j^2)$. Then 
$$ \mathbb P \left(|X_j|^t \leq \uptau \right) \leq e^{-(C-1)^2/2}.$$
\end{lemma}

\begin{proof}[Proof of Lemma~\ref{lem:prob_Xj2>2lambda_ellt}]
Assume without loss of generality that $\theta_j \geq 0$. We have
\begin{align*}
     \mathbb P \left(|X_j|^t \leq \uptau \right) &= \mathbb P \left(\bigg|\mathcal{N}\bigg(\frac{\theta_j}{\sigma_j}, 1\bigg) \bigg|^t \leq \frac{\uptau}{\sigma_j^t}\right) \leq \mathbb P \left(\mathcal{N}\bigg(\frac{\theta_j}{\sigma_j}, 1\bigg) \leq \frac{\uptau^{1/t}}{\sigma_j}\right)\\
     & = \mathbb P \left(\mathcal{N}\left(0, 1\right) \leq \frac{\uptau^{1/t} - \theta_j}{\sigma_j}\right) \leq \exp\left({-\frac{(\uptau^{1/t} - \theta_j)^2}{2\sigma_j^2}}\right) \leq e^{-(C-1)^2/2},
\end{align*}
using the relations $\theta_j \geq C^t \uptau \geq C^t \sigma_j$.
\end{proof}

\vspace{2mm}

\begin{lemma}\label{lem:alpha_j_leq_2lambda_ellt}
Let $0<\sigma_j^t \leq \uptau$ be any two positive numbers and let $X_j \sim \mathcal{N}(0,\sigma_j^2)$. Then for some absolute constant $\CL{\ref{lem:alpha_j_leq_2lambda_ellt}}>0$
$$ \alpha_j := \mathbb E \left[|X_j|^t \,|\, |X_j|^t > \uptau\right] \leq \CL{\ref{lem:alpha_j_leq_2lambda_ellt}} \uptau.$$
\end{lemma}

\begin{proof}[Proof of Lemma~\ref{lem:alpha_j_leq_2lambda_ellt}]
Letting $Z_j \sim \mathcal{N}(0,1)$, we have by Lemma~\ref{lem:gaussian_tail_moments_ellt}, and Lemma 4 from~\cite{collier2017minimax}:
\begin{align*}
    \alpha_j &= \mathbb E \Big[|X_j|^t \,\Big|\, |X_j|^t > \uptau\Big] = \frac{ \mathbb E \left[|X_j|^t \mathbb 1\left\{ |X_j|^t > \uptau\right\}\right]}{\mathbb P \left( |X_j|^t > \uptau\right)} = \sigma_j^t \frac{ \mathbb E \left[|Z_j|^t \mathbb 1\left\{ |Z_j|^t > \uptau/\sigma_j^t \right\}\right]}{\mathbb P \left( |Z_j|^t > \uptau/\sigma_j^t\right)}\\
    & \leq C\sigma_j^t \left(\frac{\uptau^{1/t}}{\sigma_j}\right)^{t-1} \frac{ \exp\left(-\uptau^{2/t}/2\sigma_j^2\right)}{\frac{\sigma_j}{\uptau^{1/t}} \exp\left(-\uptau^{2/t}/2\sigma_j^2\right)} =: \CL{\ref{lem:alpha_j_leq_2lambda_ellt}} \uptau.
\end{align*}
\end{proof}

\section{Proof of Theorem \ref{th:rate_ellt_>2}}\label{sec:proof_theorem_ellt>2}
We need the following lemma before we begin the proof of the theorem.
\begin{lemma}\label{lem:monotonicity_t>2}
The following function $\phi: \R \longrightarrow \mathbb R$ is continuous and strictly decreasing over $\mathbb R$:
$$ \phi(\beta) = \dfrac{\sum_{j=1}^d \sigma_j^2 \exp\left(-\beta/\sigma_j^2\right)}{\sqrt{\sum_{j=1}^d \sigma_j^4 \exp\left(-\beta/\sigma_j^2\right)}}.$$
\end{lemma}

\begin{proof}[Proof of Lemma~\ref{lem:monotonicity_t>2}]
The function $\phi$ is clearly differentiable hence continuous, and
\begin{align*}
    \phi'(\beta) &= \frac{\left(\sum_{j=1}^d \sigma_j^2 \exp\left(-\beta/\sigma_j^2\right) \right)^2 - 2 \left(\sum_{j=1}^d \exp\left(-\beta/\sigma_j^2\right) \right)\left(\sum_{j=1}^d \sigma_j^4 \exp\left(-\beta/\sigma_j^2\right) \right)}{2\left(\sum_{j=1}^d \sigma_j^4 \exp\left(-\beta/\sigma_j^2\right)\right)^{3/2}}
    \\
    &
    = \frac{1}{2}\left(\sum_{j=1}^d \sigma_j^4 \exp\left(-\beta/\sigma_j^2\right)\right)^{-3/2}
    \left(\sum_{j=1}^d \exp\left(-\beta/\sigma_j^2\right) \right)^{-2} \\
    &\times\left\{
        \left(
            \frac{\sum_{j=1}^d \sigma_j^2 \exp\left(-\beta/\sigma_j^2\right)}{\sum_{j=1}^d \exp\left(-\beta/\sigma_j^2\right)}
        \right)^2
        - 
        2 \frac{\sum_{j=1}^d \sigma_j^4 \exp\left(-\beta/\sigma_j^2\right)}{\sum_{j=1}^d \exp\left(-\beta/\sigma_j^2\right)}
    \right\}
    \\
    &
    = \frac{1}{2} \left(\sum_{j=1}^d \sigma_j^4 \exp\left(-\beta/\sigma_j^2\right)\right)^{-3/2}
    \left(\sum_{j=1}^d \exp\left(-\beta/\sigma_j^2\right) \right)^{-2} \left[\mathbb E^2_{Y\sim \mu} Y^2 - 2 \mathbb E_{Y \sim \mu} Y^4\right] <0.
\end{align*}

At the last line, we have defined the probability measure $\mu = \sum_{j=1}^d w_j\delta_{\sigma_j}$ where 
$$w_j = \frac{\exp\left(- \beta/\sigma_j^2\right)}{\sum_{j=1}^d \exp\left(- \beta/\sigma_j^2\right)} $$ 
and where $\delta_x$ denotes the Dirac measure at point $x\in \R$. 
We also used Jensen's inequality combined with the fact that, when $Y\sim \mu$, we have $\mathbb E Y^4 > 0$ as by assumption, $\sigma_j>0,\forall j\in[d]$.
\end{proof}

\begin{lemma}\label{lem:relations_between_contributions_ell>2}
Let $t\geq 2$ and let $\lambda, \nu$ and $j_*$ be defined as in~\eqref{eq:beta_t>2}, \eqref{eq:implicit_separation_ellt_>2} and~\eqref{eq:explicit_separation_ellt_>2}. We let
\begin{align*}
    \left\{\begin{array}{ll}
         s_{sparse} = \dfrac{1}{\nu^t}\sum_{j>j_*} \sigma_j^t \exp\left(-\lambda^2/\sigma_j^2\right),\\[10pt]
         \nu_{sparse}^t = \sqrt{\sum_{j>j_*} \sigma_j^{2t} \exp\left(-\lambda^2/\sigma_j^2\right)},\\[10pt]
         \nu_{dense}^t = \sqrt{\sum_{j\leq j_*} \sigma_j^{2t}}.
    \end{array}\right.
\end{align*}
Up to absolute constants, we have $\nu \asymp \nu_{dense} + \nu_{sparse}$ and $\nu_{sparse}^{2t} \leq \nu^t \lambda^t s_{sparse}$. 

\noindent In particular, it holds that $\nu^t + \lambda^t s \asymp \nu_{dense}^t + \lambda^t s$.
\end{lemma}

\begin{remark}\label{rem:isotropic}
In the isotropic case $\sigma_j = \sigma, ~ \forall j \in [d]$ and for $t = 2$, we recall that $\lambda_2 $ plays the role of~$\lambda^2$.
Consider the case where $s \leq 2 \sqrt{d/e}$.
Then, we have $j_* = d$ so that $\nu_{dense} = 0$. 
Therefore, it holds that $\nu^t = \nu_{sparse}^t \leq \lambda_2 s_{sparse} \leq \lambda_2 s$, hence the term $\nu_2/s$ never dominates over $\lambda_2 s$.
\end{remark}

\begin{proof}[Proof of Lemma~\ref{lem:relations_between_contributions_ell>2}]
By definition of $j_*$, the following relations are true up to absolute constants:
\begin{align*}
    \nu = \sqrt{\sum_{j=1}^d \sigma_j^{2t} \exp\left(-\lambda^2/\sigma_j^2\right)} 
    &\asymp \sqrt{\sum_{j\leq j_*} \sigma_j^{2t} \exp\left(-\lambda^2/\sigma_j^2\right)} + \sqrt{\sum_{j>j_*} \sigma_j^{2t} \exp\left(-\lambda^2/\sigma_j^2\right)}\\ 
    &\asymp \sqrt{\sum_{j>j_*} \sigma_j^{2t} } + \sqrt{\sum_{j>j_*} \sigma_j^{2t} \exp\left(-\lambda^2/\sigma_j^2\right)}\\
    & = \nu_{dense}^t + \nu_{sparse}^t.
\end{align*}
Moreover, we have
\begin{align}
    \nu_{sparse}^{2t} = \sum_{j>j_*} \sigma_j^{2t} \exp\left(-\lambda^2/\sigma_j^2\right) \leq 
    \lambda^t \sum_{j>j_*} \sigma_j^t \exp\left(-\lambda^2/\sigma_j^2\right) = \nu^t \lambda^t s_{sparse}.\label{eq:relation_sparse_ell2}
\end{align}
We now prove that $\nu^t + \lambda^t s \asymp \nu_{dense}^t + \lambda^t s$. If $\nu_{sparse} \leq \nu_{dense}$, this relation is clear. Otherwise, $\nu^t \leq 2 \nu_{sparse}^t$ so that by~\eqref{eq:relation_sparse_ell2}, we have $\nu_{sparse}^t\leq 2 \lambda^t s_{sparse}$. Therefore, $\nu^t\leq 2 \nu_{sparse}^t \leq 4 \lambda^t s_{sparse} \leq 4 \lambda^t s$, so that $\nu^t + \lambda^t s \asymp \lambda^t s \asymp \nu_{dense}^t + \lambda^t s$. 
\end{proof}

\subsection{Proof of Lower Bounds for $t \geq 2$}\label{subsec:proof_theorem_ellt>2_lower}

\begin{proof}[Proof of Theorem \ref{th:rate_ellt_>2}.\ref{th:rate_ellt_>2_lower}]
First, if $s\leq C$ for some constant $C$ depending only on $\eta$, then by Lemma~\ref{lem:generalities_LB}, we have $\epsilon^*(s,t,\Sigma)^t \geq c(\lambda+\nu)^t \asymp \lambda^t s + \nu^t$. 
From now on, we will assume that $s$ is larger than a sufficiently large constant $c(\eta)$ depending only on $\eta$.
\vspace{1mm}

We denote by $\Pi$ the prior distribution over $\theta$, defined such that $\forall b \in \{0,1\}^d$, $\forall \omega \in \{\pm 1\}^d$: $$\mathbb P_{\Pi} \left(\theta = (b_j \omega_j \gamma_j)_j \right) = \frac{1}{2^d}\Pi_{j=1}^d \pi_j^{b_j} (1-\pi_j)^{1-b_j}. $$ 

Let $\mathbb P_{prior} = \mathbb E_{\theta\sim \Pi}\left[\mathcal{N}(\theta,\Sigma)\right]$ denote the corresponding mixture of normal distributions $\mathcal{N}(\theta,\Sigma)$ where $\theta \sim \Pi$.
Note that if $\beta \geq 0$, then $\sum_{j=1}^d \pi_j = s/2$, 
otherwise, if $\beta<0$, then by monotonicity (see Lemma~\ref{lem:monotonicity_t>2}), we have $\sum_{j=1}^d \pi_j \leq s/2$. 
Our prior has a random sparsity, equal to $\sum_{j=1}^d b_j$. 
Therefore, with high probability, its sparsity is at most $s$.
To justify this, note that
\begin{align*}
    \mathbb E\Bigg[\sum_{j=1}^d b_j\Bigg] = \sum \pi_j \leq s/2,
        \quad
        \text{ and }
        \quad
    \mathbb V \Bigg[\sum_{j=1}^d b_j\Bigg] = \sum_{j=1}^d \pi_j(1-\pi_j) \leq s/2.
\end{align*}

Provided that $s \geq 20/\eta$, we have by Chebyshev's inequality, 
\begin{align}
    \mathbb P\left(\sum_{j=1}^d b_j > s\right) \leq \frac{\mathbb V \left[\sum_{j=1}^d b_j\right]}{(s/2)^2} \leq \frac{2}{s} \leq \frac{\eta}{10}.\label{eq:control_prior_sparsity}
\end{align}

We use Lemma 23 in~\cite{liu2021minimax} to compute $\chi^2\left(\mathbb P_{prior} ~ || ~ \mathbb P_0\right)$. Let $\theta, \theta'$ be two independent random variables with distribution $\Pi$. Then 
\begin{align}
    &1+\chi^2\left(\mathbb P_{prior} ~ || ~ \mathbb P_0\right) = \mathbb E_{\theta, \theta'}\left[ \exp\left(\theta^\top \Sigma^{-1} \theta'\right)\right] = \prod_{j=1}^d \mathbb E_{\theta_j, \theta'_j}\left[ \exp{\left(\theta_j \theta'_j \sigma_j^{-2}\right)}\right]\nonumber\\
    & = \prod_{j=1}^d \mathbb E \left[\exp\left(b_j b'_j \omega_j \omega'_j \gamma_j^2 \sigma_j^{-2}\right)\right]\nonumber\\
    & = \prod_{j=1}^d \left[(1-\pi_j^2) + \pi_j^2 \left(\frac{1}{2}\exp\left(-\frac{\gamma_j^2}{\sigma_j^2}\right) + \frac{1}{2}\exp\left(\frac{\gamma_j^2}{\sigma_j^2}\right) \right)\right]\nonumber\\
    & = \prod_{j=1}^d \left[1 + 2 \pi_j^2 \sinh^2\left(\frac{\gamma_j^2}{2\sigma_j^2}\right) \right] ~\leq~ \exp\left[ \sum_{j=1}^d \pi_j^2 \cdot 2\sinh^2\left(\frac{\gamma_j^2}{2\sigma_j^2}\right)\right]\label{eq:chi2}\\
    & = \exp\left[ \sum_{j=1}^d \pi_j^2 \left[\sqrt{1+c^2\exp\left(2\lambda^2/\sigma_j^2\right)} - 1\right]\right] ~~ \text{ since } 2 \sinh^2 \left(\frac{\arg \sinh v}{2} \right) = \sqrt{1+v^2} -1 \nonumber\\
    & \leq \exp\left\{c \sum_{j=1}^d \pi_j^2  \exp\left(\frac{\lambda^2}{\sigma_j^2} \right)\right\} ~~ \text{ using } \sqrt{1+u}-1 \leq  \sqrt{u}.\label{eq:indist_cond}
\end{align}
There are two cases.
\begin{enumerate}
    \item First case: $\beta \leq 0$ i.e. $\lambda = 0$, then the relation~\eqref{eq:indist_cond} simplifies as:
    \begin{align*}
        1+\chi^2\left(\mathbb P_{prior} ~ || ~ \mathbb P_0\right) \leq \exp\left[c\sum_{j=1}^d \pi_j^2 \right] = \exp (c),
    \end{align*}
    ensuring indistinguishability if $c$ is small enough. 
    Moreover, using that $\arg \sinh u \geq \log(1+u)$, we conclude that the prior's $t$-th power of its $L^t$ norm concentrates on
    \begin{align*}
        \sum_{j=1}^d \pi_j \gamma_j^t &\geq \frac{\sum_{j=1}^d \sigma_j^{2t}  }{\sqrt{\sum_{j=1}^d\sigma_j^{2t}}}\log\left(1+c\right) = \log(1+c) \sqrt{\sum_{j=1}^d\sigma_j^{2t}},
    \end{align*}
    which exactly corresponds to the dense case. 
    Note that, since $\lambda = 0$, this quantity also coincides with $\log (1+c) \sqrt{\sum_{j=1}^d\sigma_j^{2t} e^{-\lambda^2/\sigma_j^2}} + \frac{1}{2} \lambda^t s$.
    \item Second case: $\lambda> 0$. Then the relation~\eqref{eq:indist_cond} simplifies as:
    \begin{align*}
        1+\chi^2\left(\mathbb P_{prior} ~ || ~ \mathbb P_0\right) \leq \exp\left[c \sum_{j=1}^d \pi_j^2  \exp\left(\lambda^2/\sigma_j^2\right)\right]
        &= \exp\left[c\cdot \frac{\sum_{j=1}^d \sigma_j^{2t}\exp\left(-\lambda^2/\sigma_j^2\right)}{\sum_{j=1}^d \sigma_j^{2t} \exp\left(-\lambda^2/\sigma_j^2\right)} \right] = \exp c.
    \end{align*}

Moreover, we have that $\arg \sinh u \geq \log(1+u)$. Note also that, for $x \geq 1$ and $c<1$, by concavity of $x\mapsto x^c$, the function $x^c$ is always below its tangent in $x=1$ so that $x^c \leq 1+c(x-1) \leq 1+cx$. Therefore, we have the relation $\log(1+cx) \geq c\log x$. Moreover, for $x\geq 1$ we also have $\log(1+cx) \geq \log(1+c)$, so that $\log(1+cx) \geq c\log x \lor \log(1+c)$. We now apply this for $x = \exp(\lambda^2/\sigma_j^2)$, which yields that the prior's $t$-th power of its $L^t$ norm concentrates on
    \begin{align*}
        \sum_{j=1}^d \pi_j \gamma_j^t &\geq \frac{\sum_{j=1}^d \sigma_j^{2t} \exp\left(-\lambda^2/\sigma_j^2\right) \log^{t/2}\left(1+c\exp\left(\lambda^2/\sigma_j^2\right)\right)}{\sqrt{\sum_{j=1}^d\sigma_j^{2t}\exp\left(-\lambda^2/\sigma_j^2\right)}} \\
        &\geq \log^{t/2} (1+c)  \sqrt{\sum_{j=1}^d\sigma_j^{2t}\exp\left(-\lambda^2/\sigma_j^2\right)} 
        \lor 
        c^t \lambda^t \frac{\sum_{j=1}^d \sigma_j^t \exp\left(-\lambda^2/\sigma_j^2\right)}{\sqrt{\sum_{j=1}^d\sigma_j^{2t}\exp\left(-\lambda^2/\sigma_j^2\right)}} \\
        &= \log^{t/2} (1+c) \sqrt{\sum_{j=1}^d\sigma_j^{2t} \exp\left(-\lambda^2/\sigma_j^2\right)} \lor c^t \frac{ \lambda^t s}{2}.
    \end{align*}
\end{enumerate}

\end{proof}

\subsection{Upper bounds for $t\geq 2$}\label{subsec:proof_theorem_ellt>2_upper}
\begin{proof}[Proof of Theorem \ref{th:rate_ellt_>2}.\ref{th:rate_ellt_>2_upper}]

We recall the definitions of the tests $T_{dense}$ and $T_{sparse}$ from equations~\eqref{eq:def_test_stat_>2}. 

We write $\theta_{\leq j_*} = (\theta_1,\dots, \theta_{j_*})$, and $\theta_{>j_*} = (\theta_{j_*+1},\dots, \theta_{d})$, where we recall that $j_* = \max\big\{j \in [d] ~\big|~ \sigma_j\geq \lambda\big\}$.
We show that there exist large enough constants $C_0, C_1$ and $\Cbar$ such that, when $\|\theta\|_0 \leq s$, we have:

\vspace{2mm}

\begin{tabular}[m]{|c|c|c||c|c|c|}
    \cline{1-6}
    & Under $H_0$ 
    & If $\left\|\theta_{\leq j_*}\right\|_t^t \geq \Cbar \rho$
    & 
    & Under $H_0$ 
    & If $\left\|\theta_{>j_*}\right\|_t^t \geq \Cbar \rho$ \rule[-2ex]{0pt}{6ex} \\
    \cline{1-6}
    $\mathbb E^2 T_{fdense}$ 
    & $= 0$ 
    & $\geq C_1\|\theta_{\leq j_*}\|_t^2$
    & $\mathbb E^2 T_{sparse}$ 
    & $= 0$ 
    & $\geq \Cbar\rho/2^{t+2} $ \rule[-2ex]{0pt}{6ex}\\
    \cline{1-6}
    $\mathbb V T_{fdense}$ 
    & $ \leq C_0\hspace{-1mm} \sum\limits_{j\leq j_*}\hspace{-1mm} \sigma_j^{2t}$ 
    & $\leq \Cbar^{-2/t} \mathbb E_\theta^2\! \left[T_{fdense}\right]$
    & $\mathbb V T_{sparse}$ 
    & $ \leq C_0 \rho$
    & $\leq c\, \mathbb E_\theta^2 \!\left[T_{sparse}\right] $\rule[-2ex]{0pt}{6ex}\\
    \cline{1-6}
\end{tabular}
\vspace{4mm}

\noindent where in the last cell, $c$ is a constant depending only on $\eta$ and $t$, that can be made arbitrarily small provided that $\Cbar$ is large enough, and $\Cbar$ can be chosen independently of $C_0$. 
Proposition~\ref{prop:UB_gaussian_ellt} is proved by combining the above relations with Chebyshev's inequality.

\vspace{3mm}

\textit{Analysis of $T_{dense}$.} Under $H_0$, $T_{dense}$ is centered by definition. 
Its variance under $H_0$ can be bounded from above as follows
\begin{align*}
    \mathbb V T_{dense} = \sum_{j\leq j_*} \mathbb V |X_j|^t \leq \sum_{j\leq j_*} \mathbb E |X_j|^{2t} \leq C \sum_{j\leq j_*} \sigma_j^{2t},
\end{align*}
for some constant $C$ depending on $t$ (see for example~\cite{winkelbauer2012moments}). 

\vspace{3mm}

Now, if $\left\|\theta_{\leq j_*}\right\|_t^t \geq \Cbar \rho$ and $\|\theta\|_0 \leq s$, then, writing $X_j = \theta_j + \xi_j $, we will prove that $\sum_{j\leq j_*}\mathbb E_\theta |X_j|^t - \mathbb E|\xi_j|^t \geq \frac{1}{4} \|\theta\|_t^t$.
We have $\mathbb E_\theta |X_j|^t \geq |\theta_j|^t \mathbb P\left(\theta_j \xi_j \geq 0\right) \geq \frac{1}{2} |\theta_j|^t.$
Therefore, if $|\theta_j|^t \geq 4 \, \mathbb E|\xi_j|^t$, then 
\begin{align*}
    \mathbb E_\theta |X_j|^t - \mathbb E|\xi_j|^t \geq \frac{1}{4} |\theta_j|^t.
\end{align*}

Otherwise, we can assume that $|\theta_j|^t < 4 \mathbb E \mathbb |\xi_j|^t$. 
We use a Taylor expansion that is analogous to~\eqref{eq:taylor}, except that we swap the roles of $\xi_j$ and~$\theta_j$. For fixed $\xi_j$, we define the function $\phi(\theta_j) = |\xi_j + \theta_j|^\alpha$, which is twice continuously differentiable. 
For any $\theta_j \in \R$, there exists $\theta'_j \in [0,\theta_j]$ such that $\phi(\theta_j) = \phi(0) + \theta_j \phi'(0) + \frac{1}{2}\theta_j^2 \phi''(\theta_j')$, or equivalently
\begin{align*}
    \big|\theta_j + \xi_j\big|^\alpha 
    &= 
    |\xi_j|^\alpha 
    + \alpha \big|\theta_j^{\alpha-1} \xi_j\big| \operatorname{sign}(\theta_j \xi_j) 
    + \frac{\alpha (\alpha-1)}{2} \theta_j^2 \big|\xi_j + \theta'_j\big|^{\alpha-2},
\end{align*}
for any $\alpha \geq 2$. Taking the expectation, for $\alpha = t$, gives
\begin{align*}
    \mathbb E \big|\theta_j + \xi_j\big|^t 
    &= 
    \mathbb E \left[|\xi_j|^t \right] + 0 +
    \frac{t (t-1)}{2} \theta_j^2  \mathbb E \left[ \big|\xi_j + \theta'_j\big|^{\alpha-2}\right] 
    \\
    & 
    \geq \mathbb E |\xi_j|^t 
    + 
    \frac{t (t-1)}{2} \theta_j^2 \, \mathbb E \left[|\xi_j|^{t-2} \mathbb 1\left\{\theta_j \xi_j >0\right\}\right] 
    \\
    &
    = \mathbb E |\xi_j|^t  
    + 
    C \theta_j^2 \sigma_j^{t-2}
    \\
    & 
    \geq \mathbb E |\xi_j|^t  
    + 
    C \theta_j^2 {C}'^{(2-t)/t} |\theta_j|^{t-2} \quad \text{ recalling that $|\theta_j|^t < 4 \mathbb E |\xi_j|^t =: C' \sigma_j^t$ }
    \\
    & 
    = \mathbb E |\xi_j|^t  
    + 
    C'' |\theta_j|^{t}.
\end{align*}

In both cases, we have found a constant $C''$ depending only on $t$ such that $\mathbb E \big|\theta_j + \xi_j\big|^t - E |\xi_j|^t \geq C'' \mathbb E |\theta_j|^{t}$, which yields
\begin{align}
    \mathbb E_\theta T_{dense} \geq C'' \|\theta_{\leq j_*}\|_t^t.
\end{align}

We now turn to the variance term. 
We have the following classical inequalities (see~\cite{ingster2003nonparametric}):
\begin{align*}
    & (|x|+|y|)^\alpha \leq 2^{\alpha-1} \left(|x|^\alpha + y^{\alpha}\right),\\
    & \left||x+y|^\alpha - |y|^\alpha\right|  \leq \alpha 2^{\alpha - 1} |x| \left(|x|^{\alpha-1} + |y|^{\alpha-1}\right),
\end{align*}
true for any $x,y\in \R$ and $\alpha \geq 1$. Therefore,
\begin{align*}
    \mathbb V \big|\theta_j + \xi_j\big|^t &= \min_{x \in \R} \mathbb E \left[\left(\big|\theta_j + \xi_j\big|^t - x\right)^2\right]\leq \mathbb E \left[\left(\big|\theta_j + \xi_j\big|^t - |\theta_j|^t\right)^2\right]\\
    & \leq 
    t^2 2^{2t-2} \mathbb E\left[|\xi_j|^2 \left(|\xi_j|^{t-1} + |\theta_j|^{t-1}\right)^2\right]
    \\
    & \leq t^2 2^{2t-1} \mathbb E\left[|\xi_j|^2 \left(|\xi_j|^{2t-2} + |\theta_j|^{2t-2}\right)\right]
    \leq 
    C \left(\sigma_j^{2t} + \sigma_j^2 |\theta_j|^{2t-2}\right).
\end{align*}

Now, setting $u = t$ and $v = \frac{t}{t-1}$, we have $\frac{1}{u} + \frac{1}{v} = 1$, so that by Hölder's inequality
\begin{align*}
    \sum_{j\leq j_*} \sigma_j^2 |\theta_j|^{2t-2} &\leq \left(\sum_{j\leq j_*} \sigma_j^{2t} \right)^{1/u} \left(\sum_{j\leq j_*} |\theta_j|^{2t}\right)^{1/v} = \left\|\sigma_{\leq j_*}\right\|_{2t}^2 \left\|\theta_{\leq j_*}\right\|_{2t}^{2t-2}\\
    & \leq e^{1/t} \Cbar^{-2/t} \left\|\theta_{\leq j_*}\right\|_t^{2t} \leq c \, \E^2 \big[T_{dense}\big],
\end{align*}
where in the last line we used the fact that, since $\forall j \leq j_*: \sigma_j \geq \lambda$, we have $\sqrt e\big\|\sigma_{\leq j_*}\big\|_{2t}^t \leq \nu^t \leq \rho \leq \frac{1}{\Cbar} \big\|\theta_{\leq j_*}\big\|_{t}^t$. 

\textit{Analysis of $T_{sparse}$:} See Lemma~\ref{lem:analysis_Tsparse}.

\end{proof}
\section{Proof of Theorem \ref{th:rate_ellt}}\label{sec:proof_theorem_ellt}
We first present a couple of useful lemmas before presenting the detailed proofs of the lower and upper bounds corresponding to Theorem \ref{th:rate_ellt}.

\begin{lemma}\label{lem:nu_continuous_ellt}
The functions $\overline \nu(x)$ and $f(x)$ are continuous with respect to $x$. Moreover, $\lim\limits_{x\to 0^+} f(x) = d$ and $\lim\limits_{x\to +\infty} f(x) = 0$.
\end{lemma}

\begin{proof}[Proof of Lemma~\ref{lem:nu_continuous_ellt}]
In this proof, we will use the notation from Section~\ref{sec:results_ellt_leq2}. 
Fix $x_0 \geq 0$ and recall~\eqref{eq:def_jstar_function} and \eqref{eq:def_jstar_ellt}. 
Note that $j_*(x)$ is always left-continuous: $j_*(x) = \lim_{y \to x^-} j_*(y)$. Therefore, from~\eqref{eq:def_nu_ellt}, $\overline \nu(x)$ is a left-continuous function. 
Now, we show that $x \mapsto j_*(x)$ and $x \mapsto \overline \nu(x)$ are right-continuous functions. 
Fix $x_0 \in \R$; we always have $\sigma_{j_*} \geq x_0$ by definition.
If $\sigma_{j_*}> x_0$, then $x \mapsto j_*(x)$ is clearly continuous on a neighborhood of $x_0$ and so is $x \mapsto \overline \nu(x)$. 
Otherwise, we have $\sigma_{j_*} = x_0$ and we show that $x \mapsto j_*(x)$ and $x \mapsto \overline \nu(x)$ are still right-continuous in $x_0$. 
Define
\begin{align}
    J(x_0) = \left\{j ~\big|~ \sigma_j = x_0 = \sigma_{j_*}\right\}.
\end{align}

By~\eqref{eq:def_nu_ellt}, we have that $\overline \nu^t(x_0) \geq C \dfrac{\sigma_{j_*}^{a}}{\overline \nu(x_0)^{a-t}} \land \dfrac{\sigma_{j_*}^4}{x_0^{4-2t}\, \overline \nu^t(x_0)} = C \dfrac{x_0^{a}}{\overline \nu(x_0)^{a-t}} \land \dfrac{x_0^{2t}}{\overline \nu^t(x_0)}$, so that $\overline \nu(x_0) \geq x_0$ and consequently, $\forall j \in J(x_0): \dfrac{\sigma_{j}^{a}}{\overline \nu(x_0)^{a-t}} \land \dfrac{\sigma_{j}^4}{x_0^{4-2t} \,\overline \nu^t(x_0)} =  \dfrac{x_0^{2t}}{\overline \nu^t(x_0)}$. 
We can now write
\begin{align}
    \overline \nu^t(x_0) &= \sum_{\substack{j\leq j_*\\j \notin J(x_0)}} \frac{\sigma_{j}^{a}}{\overline \nu(x_0)^{a-t}} \land \frac{\sigma_{j}^4}{x_0^{4-2t} \,\overline \nu^t(x_0)} + \big|J(x_0)\big| \frac{x_0^{2t}}{\overline \nu^t(x_0)} + \sum_{j>j_*} \frac{\sigma_j^{2t} }{\nu^t}\exp\!\bigg(\!\!-\frac{x_0^2}{\sigma_j^2}+1\bigg)\nonumber\\
    & = \sum_{\substack{j\leq j_*\\j \notin J(x_0)}} \frac{\sigma_{j}^{a}}{\overline \nu^{a-t}(x_0)} \land \frac{\sigma_{j}^4}{x_0^{4-2t} \,\overline \nu^t(x_0)} + \sum_{\substack{j>j_*\\
    \text{ or } j \in J(x_0)}} \frac{\sigma_j^{2t} }{\overline \nu^t(x_0)}\exp\!\bigg(\!\!-\frac{x_0^2}{\sigma_j^2}+1\bigg).\label{eq:nu(x0)_ellt}
\end{align}

Noting that for a sufficiently small $\delta>0$, we have for any $x \in (x_0, x_0+\delta)$
\begin{align*}
    \overline \nu^t(x) = \sum_{\substack{j\leq j_*\\j \notin J(x_0)}} \frac{\sigma_{j}^{a}}{\overline \nu^{a-t}(x)} \land \frac{\sigma_{j}^4}{x^{4-2t} \overline \nu^t(x)} + \sum_{\substack{j>j_*\\
    \text{ or } j \in J(x_0)}} \frac{\sigma_j^{2t} }{\overline \nu^t(x)}\exp\!\bigg(\!\!-\frac{x^2}{\sigma_j^2}+1\bigg).
\end{align*}

Now, the two sets of summation indices are \textit{fixed} when $x \in (x_0, x_0+\delta)$, so that the right-hand side is clearly continuous with respect to $x$ over $(x_0,x_0+\delta)$. 

Therefore, writing $\overline \nu(x_0^+) = \lim\limits_{x \to x_0^+} \overline \nu(x)$, we get
\begin{align*}
    \overline \nu^t(x_0^+) = \sum_{\substack{j\leq j_*\\j \notin J(x_0)}} \frac{\sigma_{j}^{a}}{\overline \nu^{a-t}(x_0^+)} \land \frac{\sigma_{j}^4}{x_0^{4-2t} \,\overline \nu^t(x_0^+)} + \sum_{\substack{j>j_*\\
    \text{ or } j \in J(x_0)}} \frac{\sigma_j^{2t} }{\overline \nu^t(x_0^+)}\exp\!\bigg(\!\!-\frac{x_0^2}{\sigma_j^2}+1\bigg).
\end{align*}

Comparing with~\eqref{eq:nu(x0)_ellt}, we note that $\overline \nu(x_0^+)$ and $\overline \nu(x_0)$ solve the same equation, hence $\overline \nu(x_0^+) = \overline \nu(x_0)$ by uniqueness of the solution of~\eqref{eq:def_nu_ellt}. 
This proves that $\nu$ is right-continuous in $x_0$, concluding the proof of the continuity of $\nu$. 
The continuity of $f$ can be proved by exactly following the same steps.\\

When $x \to 0^+$, we have $j_*(x) \to d$ and $\overline \nu(x) \to \Cnu^{-1/a}\|\sigma\|_{a}$. 
Therefore, $f(x) \to \sum_{j=1}^d 1 = d$. \\

When $x \to \infty$, we have $j_*(x) = 0$ and $\overline \nu(x) = \left(\sum\limits_{j=1}^d e\sigma_j^{2t}\exp\!\left(\!\!-\frac{x^2}{\sigma_j^2}\right)\right)^{1/2t}$ for $x$ large enough. 
Therefore, still for $x$ large enough, we have $f(x) = \dfrac{\sum_{j=1}^d \sigma_j^t \exp\!\Big(\!\!-\frac{x^2}{\sigma_j^2}\Big)}{\sqrt{\sum_{j=1}^d \sigma_j^{2t} \exp\!\Big(\!\!-\frac{x^2}{\sigma_j^2}\Big)}} \underset{x \to \infty}{\longrightarrow} 0$.
\end{proof}

\vspace{2mm}

\vspace{2mm}

\begin{lemma}\label{lem:relations_between_contributions}
Recall the notation of Section~\ref{sec:results_ellt_leq2}.
Writing $\sigma_{\leq i_*} = (\sigma_1,\dots, \sigma_{i_*})$ and $\sigma_{int} = (\sigma_{i_*\!+1},\dots, \sigma_{j_*})$, we let 
\begin{align*}
    \left\{\begin{array}{lll}
        s_{fdense} = i_*,
        & s_{inter} = \sum\limits_{j = i_*\!+1}^{j_*} \frac{\sigma_j^4}{\lambda^{4-t} \nu^t},
        & s_{sparse} = \sum_{j>j_*} \frac{e\sigma_j^t}{\nu^t} \exp\left(-\lambda^2/\sigma_j^2\right),\\
        \nu_{fdense}^t = \|\sigma_{\leq i_*}\|_a^t,
        & \nu_{inter}^t =  \lambda^{t-2} \|\sigma_{int}\|_4^2,
        & \nu_{sparse}^t = \sqrt{\sum_{j>j_*} e\sigma_j^{2t} \exp\left(-\lambda^2/\sigma_j^2\right)}.
    \end{array}\right.
\end{align*}
Note that by equations~\eqref{eq:def_f_ellt} and~\eqref{eq:def_lambda_ellt}, we have $\frac{s}{2} = s_{fdense} + s_{inter} + s_{sparse}$, and by Lemma~\ref{lem:explicit_expr_nu_ellt}, we have $\nu^t \asymp \nu^t_{fdense} + \nu^t_{inter} + \nu^t_{sparse}$. 
Then the following relations hold:
\begin{enumerate}
    \item $\lambda^t \left(s_{fdense}+s_{inter}\right) \leq 2 \nu^t$. In particular, we have $\nu^t + \lambda^t s_{sparse} \asymp \nu^t + \lambda^t s$.
    \item\label{item:rate_ellt} $\nu^{2t}_{inter} = \nu^t \lambda^t s_{inter} $ and $\nu^{2t}_{sparse} \leq \nu^t \lambda^t s_{sparse} $. In particular: $\nu^t + \lambda^t s \asymp \nu_{fdense}^t + \lambda^t s$.
\end{enumerate}
\end{lemma}

\begin{proof}[Proof of Lemma~\ref{lem:relations_between_contributions}]
\begin{enumerate}
    \item We have by Lemma~\ref{lem:explicit_expr_nu_ellt}
\begin{align*}
    \lambda^t s_{fdense} &= \lambda^t i_* = \frac{i_* \lambda^t \nu^{a-t}}{\nu^{a-t}} \leq \frac{1}{\nu^{a-t}} \sum_{j\leq i_*} \sigma_j^{a} \leq \nu^t
    \quad \text{ by definition of } i_* \text{ from~\eqref{eq:def_istar_ellt}.}\\
    \lambda^t s_{inter} &= \sum\limits_{j = i_*\!+1}^{j_*} \frac{\sigma_j^4}{\lambda^{4-2t} \nu^t} \leq \nu^t.
\end{align*}

    \item We have by definition of $s_{int}$ and $\nu_{int}^t$:
    \begin{align*}
        \lambda^t s_{inter} = \sum\limits_{j = i_*\!+1}^{j_*} \frac{\sigma_j^4}{\lambda^{4-2t} \nu^t} = \frac{\nu_{inter}^{2t}}{\nu^t}.
    \end{align*}
    Moreover, we have  by definition of $s_{sparse}$ and $\nu_{sparse}^t$:
    \begin{align}
        \lambda^t s_{sparse} 
        = \sum_{j>j_*} \frac{e\lambda^t\sigma_j^t}{\nu^t} \exp\left(-\lambda^2/\sigma_j^2\right) 
        \geq \sum_{j>j_*} \frac{e\sigma_j^{2t}}{\nu^t} \exp\left(-\lambda^2/\sigma_j^2\right) =  \frac{ \nu_{sparse}^{2t}}{\nu^t}.\label{eq:relation_sparse}
    \end{align}
    Finally, we prove that $\nu^t + \lambda^t s \asymp \nu_{fdense}^t + \lambda^t s$. If $\max(\nu_{fdense}, \nu_{int}, \nu_{sparse}) = \nu_{fdense}$, then the result is clear. 
    Otherwise, assume first that $\nu_{fdense}\leq \nu_{int}\leq \nu_{sparse}$. 
    Then we have $\nu^t \leq 3\nu_{sparse}^t $, which, by equation~\eqref{eq:relation_sparse}, yields $\nu_{sparse}^t \leq 3 \lambda^t s_{sparse}$. In particular, $\nu^t \leq 9 \lambda^t s$, so that $\lambda^t s + \nu_{fdense}^t \asymp \lambda^t s + \nu^t \asymp \lambda^t s$.
    Proceeding similarly if $\nu_{fdense}\leq \nu_{sparse}\leq \nu_{int}$ concludes the proof.
\end{enumerate}
\end{proof}

\subsection{Lower bounds}\label{subsec:proof_theorem_ellt_lower}

\begin{proof}[Proof of Theorem \ref{th:rate_ellt}.\ref{th:rate_ellt_<2_lower}]
First, if $s\leq C$ for some constant $C$ depending only on $\eta$, then by Lemma~\ref{lem:generalities_LB}, we have $\epsilon^*(s,t,\Sigma)^t \geq c(\lambda+\nu)^t \asymp \lambda^t s + \nu^t$. 
From now on, we will assume that $s$ is larger than a sufficiently large constant $c(\eta)$ depending only on $\eta$.
\vspace{1mm}

We recall the definition of the prior from Subsection~\ref{subsec:LB_ellt}. We can bound from above the $\chi^2$ divergence between this prior and $\mathbb P_0$ as in~\eqref{eq:chi2}:
\begin{align*}
    1+\chi^2\left(\mathbb P_{prior} ~ || ~ \mathbb P_0\right) \leq \exp\left[ \sum_{j=1}^d \pi_j^2 \cdot 2\sinh^2\left(\frac{\gamma_j^2}{2\sigma_j^2}\right)\right].
\end{align*}

Recall the notation from Section~\ref{sec:results_ellt_leq2}. 
Lemma~\ref{lem:explicit_expr_nu_ellt} ensures that $\nu \geq \sigma_j$ for any $j \leq i_*$, so that $\gamma_j^t = c\sigma_j^a/\nu^{a-t} \leq c \sigma_j^t$. 
Moreover, by definition of $j_*$, we also have $\gamma_j = c \lambda \leq \sigma_j$ for any $j \in \{i_*\!+1,\dots, j_*\}$. 
Therefore, on the dense part $\{1,\dots, j_*\}$, we can use the relation $\sinh(x) \leq 2x$ which holds for any $x \leq 1$. We get:
\begin{align}
    2\sum_{j \leq j_*} \pi_j^2 \sinh^2\left(\frac{\gamma_j^2}{2\sigma_j^2}\right) \leq \sum_{j\leq j_*} \pi_j^2 \cdot \frac{\gamma_j^4}{\sigma_j^4} = c^4\sum_{j \leq i_*} \frac{\sigma_j^{a}}{\nu^{a}} +  \frac{c^4}{\lambda^{4-2t}\nu^{2t}} \sum_{j = i_*\!+1}^{j_*} \sigma_j^4 \leq 2 c^4 ~~ \text{ by Lemma~\ref{lem:explicit_expr_nu_ellt}.}\label{eq:control_chi2_dense_ellt}
\end{align}
Now, for any $j>j_*$, we use $\sinh(x) \leq e^x/2$ to get
\begin{align*}
    2\sum_{j > j_*} \pi_j^2 \sinh^2\left(\frac{\gamma_j^2}{2\sigma_j^2}\right) \leq 2\sum_{j>j_*} \frac{\sigma_j^{2t}}{\nu^{2t}} \exp\left(-2 \lambda^2/\sigma_j^2\right) \cdot \frac{\exp}{4}\left(\lambda^2/\sigma_j^2\right) = \sum_{j>j_*} \frac{\sigma_j^{2t}}{\nu^{2t}}  \exp\left(\lambda^2/\sigma_j^2\right).
\end{align*}

By Lemma~\ref{lem:explicit_expr_nu_ellt}, the latter quantity can be made arbitrarily small provided that $\Cnu$ is sufficiently small.
Combining this fact with~\eqref{eq:control_chi2_dense_ellt}, we conclude that $\chi^2\left(\mathbb P_{prior} ~ || ~ \mathbb P_0\right)$ can be made arbitrarily small by choosing $c$ and $\Cnu$ small enough, ensuring the indistinguishability condition.\\

By definition of $\lambda$ from~\eqref{eq:def_lambda_ellt}, we have $f(\lambda) = s/2$ so that $\sum_{j=1}^d \pi_j = s/2$. 
Therefore, with high probability, the prior's sparsity is at most $s$, provided that $s$ is greater than a sufficiently large constant depending only on $\eta$.
Now, letting $s_{sparse} = \sum_{j>j_*} \pi_j$, the prior's $L^t$ norm raised to the power $t$ concentrates on 
\begin{align*}
    \sum_{j=1}^d \gamma_j^t \pi_j &= \sum_{j\leq i_*} c \frac{\sigma_j^{a}}{\nu^{a-t}} + c\sum_{j=i_*}^{j_*} \frac{\sigma_j^4}{\lambda^{4-2t} \nu^t} + c\lambda^t s_{sparse}\\
    & \geq \sum_{j\leq i_*} c \frac{\sigma_j^{a}}{\nu^{a-t}} + c\sum_{j=i_*}^{j_*} \frac{\sigma_j^4}{\lambda^{4-2t} \nu^t} + \frac{c}{\nu^t} \sum_{j>j_*}\frac{\lambda^t \sigma_j^t + \sigma_j^{2t}}{2} \exp\left(-\frac{\lambda^2}{\sigma_j^2}\right) ~~ \text{ since } \lambda \geq \sigma_j \text{ for } j>j_*\\
    & \geq c \left( \nu^t + \frac{\lambda^t}{2} s_{sparse}\right), \text{ by Lemma~\ref{lem:explicit_expr_nu_ellt}.}
\end{align*}
Now, by Lemma~\ref{lem:relations_between_contributions}, we have $\nu^t + \lambda^t s_{sparse} \asymp \nu^t + \lambda^t s$. 
This concludes the proof.

\end{proof}

\subsection{Upper bounds for $t\in [1,2]$}\label{subsec:proof_theorem_ellt_upper}
\begin{proof}[Proof of Theorem \ref{th:rate_ellt}.\ref{th:rate_ellt_<2_upper}]
Theorem~\ref{th:rate_ellt}.\ref{th:rate_ellt_<2_upper} is proved by combining Lemmas~\ref{lem:exp_var_T_ellt}, \ref{lem:analysis_Tsparse} and Chebyshev's inequality.
\end{proof}

\begin{lemma}\label{lem:exp_var_T_ellt}
We write $\theta_{\leq i_*} = (\theta_1,\dots, \theta_{i_*})$, and $\theta_{int} = (\theta_{i_*\!+1},\dots, \theta_{j_*})$. 
There exist two large enough constants $C_0$ and $\Cbar$ such that, when $\|\theta\|_0 \leq s$ the following relations hold:

\vspace{2mm}

\begin{tabular}[m]{| c | c| c ||c|c|c|}
    \cline{1-6}
    & Under $H_0$ 
    & When $\left\|\theta_{\leq i_*}\right\|_t^t \geq \Cbar \rho$
    & 
    & Under $H_0$ 
    & When $\left\|\theta_{int}\right\|_t^t \geq \Cbar \rho$ \rule[-2ex]{0pt}{6ex}\\
    \cline{1-6}
    $\mathbb E^2 T_{fdense}$ 
    & $= 0$ 
    & $\geq \Cbar^{4/t} \displaystyle\sum_{j\leq i_*} \sigma_j^{a}$
    & $\mathbb E^2 T_{inter}$ 
    & $= 0$ 
    & $\geq 4 \Cbar^2\displaystyle\sum\limits_{j=i_*\!+1}^{j_*} \sigma_j^4$ \rule[-2ex]{0pt}{6ex}\\
    \cline{1-6}
    $\mathbb V T_{fdense}$ 
    & $ = 2 \sum\limits_{j\leq i_*} \sigma_j^{a}$ 
    & $\leq c \, \mathbb E_\theta^2 \left[T_{fdense}\right]$ 
    & $\mathbb V T_{inter}$ 
    & $ = 2\sum\limits_{j=i_*\!+1}^{j_*} \sigma_j^4$ 
    & $\leq c \, \mathbb E_\theta^2 \big[T_{inter}\big]$ \rule[-2ex]{0pt}{6ex}\\
    \cline{1-6}
\end{tabular}

In the above table, the constant $c$ can be made arbitrarily small provided that $\Cbar$ is large enough.
\end{lemma}

\begin{proof}[Proof of Lemma~\ref{lem:exp_var_T_ellt}]

\phantom{ }

\begin{enumerate}
    \item \textit{Analysis of $T_{fdense}$:} Under $H_0$, the relations $\mathbb E T_{fdense} = 0$ and $\mathbb V T_{fdense} = 2 \sum\limits_{j\leq i_*} \sigma_j^{a}$ are clear. 
    Under the alternative, assume that $\|\theta_{\leq i_*}\|_t^t \geq C\rho \geq C\left[\sum_{j \leq i_*} \sigma_j^{a}\right]^{t/a}$ by Lemma~\ref{lem:explicit_expr_nu_ellt}. 
    By the Hölder inequality, we have
    \begin{align*}
        \mathbb E T_{fdense} = \sum_{j\leq i_*} \frac{\theta_j^2}{\sigma_j^{2b}} \geq \frac{\left\|\theta_{\leq i_*}\right\|_t^2}{\left(\sum\limits_{j\leq i_*} \sigma_j^{a}\right)^{(2-t)/t}} \geq \Cbar^{2/t} \left(\sum\limits_{j\leq i_*} \sigma_j^{a}\right)^{1/2}.
    \end{align*}
    As for the variance, we have
    \begin{align*}
        \mathbb V T_{fdense} &= 4\sum_{j\leq i_*} \frac{\theta_j^2 \sigma_j^2}{\sigma_j^{4b}} + 2 \sum_{j \leq i_*} \sigma_j^{a} = 4\sum_{j \leq i_*} \frac{\theta_j^2}{\sigma_j^{2b}} \sigma_j^{a/2} + 2 \sum_{j \leq i_*} \sigma_j^{a} \\
        & \leq 4 \sqrt{\sum_{j \leq i_*} \frac{\theta_j^4}{\sigma_j^{4b}} \sum_{j\leq i_*} \sigma_j^{a}} + 2 \sum_{j \leq i_*} \sigma_j^{a} \quad\quad \text{ by the Cauchy-Schwarz inequality}\\
        & \leq 4\sum_{j \leq i_*} \frac{\theta_j^2}{\sigma_j^{2b}} \cdot \frac{1}{\Cbar^{2/t}} \mathbb E \left[T_{fdense}\right] + \frac{2}{\Cbar^4} \mathbb E^2 \left[T_{fdense}\right]\\
        & \leq c \, \mathbb E^2 \left[T_{fdense}\right].
    \end{align*}
    Note that we did not need to use the fact that $\|\theta\|_0 \leq s$ in the analysis of $T_{fdense}$.
    \item \textit{Analysis of $T_{inter}$:} Under $H_0$, the relations $\mathbb E T_{inter} = 0$ and $\mathbb V T_{inter} = 2 \displaystyle\sum_{j > i_*}^{j_*} \sigma_j^4$ are clear. 
    Now, assume that $\|\theta\|_0 \leq s$ and $\left\|\theta_{int}\right\|_t^t \geq \Cbar \rho \geq \Cbar \lambda^t s+ \dfrac{\Cbar}{\lambda^{2-t}} \displaystyle\left[\sum\limits_{j>i_*}^{j_*} \sigma_j^4\right]^{1/2}$. Note that
    \begin{align}
        \frac{1}{\Cbar^t}\left\|\theta_{int}\right\|_t^t \geq \lambda^t s+\dfrac{1}{\lambda^{2-t}} \displaystyle\left[\sum\limits_{j>i_*}^{j_*} \sigma_j^4\right]^{1/2} \hspace{-4mm} \geq \inf_{\lambda' >0} (\lambda')^t s+ \dfrac{1}{(\lambda')^{2-t}} \displaystyle\left[\sum\limits_{j>i_*}^{j_*} \sigma_j^4\right]^{1/2} \hspace{-3mm} = 2 s^{1-t/2} \left[\sum\limits_{j>i_*}^{j_*} \sigma_j^4\right]^{t/4}.\label{eq:relation_interm_ellt}
    \end{align}
    Now, by Hölder's inequality, we can bound from below the expectation term as follows
    \begin{align}
        \mathbb E T_{inter} = \sum_{j>i_*}^{j_*} \theta_j^2 = \left\|\theta_{int}\right\|_2^2 \geq \frac{\|\theta_{int}\|_t^2}{s^{1-2/t}} \geq 4 \Cbar^{2/t} \left[\sum\limits_{j>i_*}^{j_*} \sigma_j^4\right]^{1/2}.\label{eq:exp_T_inter_ellt}
    \end{align}
    As for the variance, we have $\mathbb V \left[T_{inter}\right] = 2 \displaystyle\sum_{j>i_*}^{j_*} \sigma_j^4 + 4 \displaystyle\sum_{j>i_*}^{j_*}\sigma_j^2 \theta_j^2$, and by the Cauchy-Schwarz inequality:
    \begin{align*}
        \sum_{j>i_*}^{j_*}\sigma_j^2\theta_j^2 \leq \sqrt{\sum_{j>i_*}^{j_*} \sigma_j^4 \sum_{j> i_*}^{j_*} \theta_j^4} \leq \frac{1}{4\Cbar^{2/t}} \mathbb E \big[T_{inter}\big] \|\theta_{int}\|_2^2 \leq \frac{1}{4\Cbar^{2/t}} \mathbb E^2 \big[T_{inter}\big] \quad \text{ by~\eqref{eq:exp_T_inter_ellt}.}
    \end{align*}
    Therefore, sill by~\eqref{eq:exp_T_inter_ellt}:
    \begin{align*}
        \mathbb V \left[T_{inter}\right] \leq \frac{1}{\Cbar^{4/t}}\mathbb E^2\big[T_{inter}\big] + 4\cdot \frac{1}{4\Cbar^{2/t}} \mathbb E^2\big[T_{inter}\big] \leq c \, \mathbb E^2\big[T_{inter}\big].
    \end{align*}    
\end{enumerate}
\end{proof}

\subsection{Proof of Lemma \ref{lem:explicit_expr_nu_ellt}}

\begin{proof}[Proof of Lemma~\ref{lem:explicit_expr_nu_ellt}]
By definition of $i_*$ from~\eqref{eq:def_istar_ellt} and $\nu$ from~\eqref{eq:def_nu_ellt}, we have
\begin{align}
    \Cnu &= \frac{A}{\nu^a}+ \frac{B}{\nu^{2t}} \label{eq:nu_A_B}\\
    \text{where } A &= \sum_{j\leq i_*} \sigma_j^a \quad \text{ and } \quad B = \sum_{j>i_*} \frac{\sigma_j^4}{\lambda^{4-2t}} \land \sigma_j^{2t} \exp\left(-\frac{\lambda^2}{\sigma_j^2} +1\right).\nonumber
\end{align}

Therefore, we have $\nu \geq \left(\Cnu^{-1} A\right)^{1/a} \lor \left(\Cnu^{-1} B\right)^{1/2t}$, hence:
\begin{align*}
     \nu^t \geq \frac{\Cnu^{-t/a}}{2} A^{t/a} + \frac{\Cnu^{-1/2}}{2} B^{1/2}.
\end{align*}

Setting $C_1 = (\Cnu^{-t/a}/2) \, \land \, (\Cnu^{-1/2}/4)$ yields the lower bound part of the claim. 
For the upper bound part, note that equation~\eqref{eq:nu_A_B} yields
\begin{align*}
    \frac{\Cnu}{2} \leq \frac{A}{\nu^a} \lor \frac{B}{\nu^{2t}},
\end{align*}
therefore, we have $\nu \leq (2 A/\Cnu)^{1/a}$ or $\nu \leq (2 B/ \Cnu)^{1/2t}$ so that 
$$\nu^t \leq (2 A/\Cnu)^{t/a} \lor (2 B/\Cnu)^{1/2} \leq (2 A/\Cnu)^{t/a} + (2 B/\Cnu)^{1/2}.$$

Taking $C_2 = (2\Cnu)^{-t/a} \lor (2\Cnu)^{-1/2} $ concludes the proof.
\end{proof}
\section{Proof of Theorem~\ref{th:rate_Linfty}}\label{Appendix:Proof_Linfty}

\begin{proof}[Proof of Lemma~\ref{lem:equiv_Linfty_Lt}]
In this proof, we let $\epsilon^*_t(s) = \epsilon^*(1,t,\Sigma)$ for any $s\in [d]$ and $t \in [1,\infty]$, and we also define $\epsilon^*_t = \epsilon^*_t(1)$.

\begin{enumerate}

    \item We first prove that     $\epsilon^*_\infty(s) = \epsilon_\infty^*$. 
    The inequality $\epsilon^*_\infty(s) \geq \epsilon^*_\infty$ is clear. 
    To prove the converse bound, we define $\Pi$ as a prior over the parameter space $$\Theta(\epsilon^*_\infty(s),s,\infty) = \left\{\theta \in \R^d ~\big|~ \|\theta\|_\infty \geq \epsilon^*_\infty(s) \text{ and }
    \|\theta\|_0 \leq s\right\}$$ and denote by $\mathbb P_{\Pi} = \mathbb E_{\theta\sim \Pi}\left[\mathcal{N}(\theta, \Sigma)\right]$ the corresponding mixture induced by $\Pi$. 
    Assume moreover that $\operatorname{TV}(\mathbb P_{\Pi}, \mathbb P_0) \leq 1-\eta$. 
    For any $\theta \in \R^d$, we let $i(\theta) = \min \left\{\arg\max_{j \in [d]} |\theta_j| \right\}$ denote the smallest index $i$ such that $|\theta_i| = \|\theta\|_\infty$. 
    We also let $\phi(\theta) = \left(\theta_j \mathbb 1_{j=i(\theta)}\right)_{j\in[d]}$ be the vector obtained by zeroing out all coordinates of $\theta$ except for the first extremal one. 
    Note that $\phi(\theta)$ is always $1$-sparse and that $\|\phi(\theta)\|_\infty = \|\theta\|_\infty$ by construction.
    We consider the new prior $\phi \Pi$ by 
    $$P_{\phi\Pi} = \mathbb E_{\theta\sim \Pi} \left[\mathcal{N}(\phi(\theta), \Sigma)\right].$$
    Therefore, $\phi\Pi$ is a prior over $\Theta(\epsilon^*(s), 1, \infty)$ and clearly, $\operatorname{TV}(\mathbb P_{\phi\Pi}, \mathbb P_0) \leq \operatorname{TV}(\mathbb P_{\Pi}, \mathbb P_0) \leq 1-\eta$, which proves by~\eqref{eq:conclusion_LB} that $\epsilon^*_\infty(s) \leq \epsilon^*_\infty$.
    
    \item Now, we prove that $ \epsilon^*_\infty = \epsilon^*_t$. We let $\psi_\infty$ be the test defined in~\eqref{eq:test_Linfty} and $\psi_t$ be the test defined in~\eqref{eq:def_tests_ellt} for $t \in [1,2]$, or in~\eqref{eq:def_tests_>2} for $t \in[2,\infty)$. 
    Let $\theta \in \R^d$ such that $\|\theta\|_0 = 1$ and $\|\theta\|_\infty > \epsilon^*_t$. 
    Then we also have $\|\theta\|_t > \epsilon^*_t$, hence 
    \begin{align*}
    &\mathbb P_0\left(\psi_t = 1\right) + \mathbb P_{\theta}\left(\psi_t = 0\right) \leq \eta\\
    \implies  R(\psi_t, \epsilon^*_t, 1,\infty,\Sigma) & = \mathbb P_{0} \left(\psi_t = 1\right) + \sup \left\{\mathbb P_{\theta} \left(\psi_t = 0\right)~ \Big| ~ \theta \in \Theta(\epsilon^*_t,1,\infty)\right\}\leq \eta\\
    \implies R^*(\epsilon^*_t, 1,\infty, \Sigma) \leq & \eta \implies \epsilon^*_t \geq \epsilon^*_\infty.
    \end{align*}
    The proof of the converse bound $\epsilon^*_t \leq \epsilon^*_\infty$ is analogous.
\end{enumerate}
\end{proof}

\subsection{Upper bound for $t = \infty$}\label{subapp:UB_Linfty}
We recall the expression of our test $\psi^* = \mathbb 1\Big\{\exists j \in [d]: |X_j| > C(\lambda+\nu)\Big\}$, where $\lambda$ and $\nu$ are respectively defined as in~\eqref{eq:beta_t>2} and~\eqref{eq:implicit_separation_ellt_>2} by taking $t'=2$ and $s' = 1$ in these equations. 
We first show that $\forall j \in [d]: \sigma_j \leq e(\lambda+\nu)$. 
If $\sigma_1 \leq \lambda$, then the result is clear. 
Otherwise, we have $\nu = \left(\sum_{j=1}^d \sigma_j^4 e^{-\lambda^2/\sigma_j^2}\right)^{1/4} \geq \sigma_1 e^{-\lambda^2/\sigma_j^2} \geq \sigma_1/e\geq \sigma_j/e$. 
We now set $\rho = \lambda + \nu$ for readability.

\begin{itemize}
    \item Under $H_0$, taking $C \geq \frac{8}{\sqrt{2\pi}e}$, we have by Lemma 4 from~\cite{collier2017minimax}
\vspace{-2mm}
\begin{align*}
    \mathbb P_0 \left(\psi^* = 0\right) &= \mathbb P_{0} \left(\forall j \in [d]: |X_j|\leq C\rho\right) = \prod_{j=1}^d \left(1 - \mathbb P\left(|X_j|> C\rho\right)\right)\\
    & \geq 
        \prod_{j=1}^d \left(1 - \frac{4}{\sqrt{2\pi} C\rho/\sigma_j} \exp\left(-C^2\rho^2/\sigma_j^2\right)\right) \\
    & \geq 
        \exp\left\{-\sum_{j=1}^d \frac{40}{\sqrt{2\pi} C\rho/\sigma_j} \exp\left(-C^2\rho^2/\sigma_j^2\right)\right\} \quad \text{ using that } 1-x \geq e^{-10x} \text{ for } x \in \left[0,\frac{1}{2}\right]\\
    & \geq 
        \exp\left\{-c\right\} \geq 1- 2c, \quad\quad \text{ by Lemma~\ref{lem:relation_tau_rho}},
\end{align*}
for any small constant $c>0$, provided that $C$ large enough. 
Note that in Lemma~\ref{lem:relation_tau_rho}, since $s=1$, we have $\uptau = \rho$, where the notation $\rho$ represents $\rho^2$ in the present proof.

\item Under $H_1$: Assume that $\|\theta\|_\infty \geq C' (\lambda + \nu)$ for some large enough $C'$ depending on $\eta$. 
Let $j = \arg\max_{i=1}^d |\theta_i|$ and without loss of generality, assume that $\theta_j>0$. Then, writing $X_j = \theta_j + \xi_j$, we get
\begin{align*}
    \mathbb P\left(\psi^* = 1\right) &\geq \mathbb P\left(|X_j|> C\rho\right) \geq \mathbb P\left(X_j> C\rho\right)  \geq \mathbb P\left(\xi_j> C\rho - \theta_j\right)\\
    &\geq \mathbb P\left(\xi_j> (C - C')\rho\right) \geq \mathbb P\left(\mathcal{N}(0,1)> \frac{1}{e}\left(C-C'\right)\right) \geq 1-c
\end{align*}
for any small constant $c>0$, provided that, for fixed $C$, the constant $C'$ is large enough.
At the last line, we used the fact that $\sigma_j \leq e\rho$ for any $j\in[d]$.

\end{itemize}

\subsection{Proof of Theorem~\ref{th:rate_Linfty_explicit}}\label{app:proof_th_infty_explicit}

\textit{Lower bound}.
Let $\mathbb{P}_{0}$ denote the distribution $N(0, \Sigma)$. Set $\rho=\max _{i} \sigma_{i} \sqrt{\log (1+i)}$ and define the prior $\pi$ on $\mathbb{R}^{d}$ as follows. First, define
$$
i^{*}= \max \big\{\arg \max _{i} \sigma_{i} \sqrt{\log (1+i)}\big\}
$$

Draw a random index $I \sim \operatorname{Uniform}\left(\left\{1, \ldots, i^{*}\right\}\right)$ and set $\theta=c \rho e_{I}$ where $\left\{e_{1}, \ldots, e_{d}\right\}$ denotes the standard basis of $\mathbb{R}^{d}$. Indeed, $\theta \sim \pi$ implies $\|\theta\|_{0}=1$ and $\|\theta\|_{\infty}=c \rho$ and so $\pi$ is properly supported. Let $\mathbb{P}_{\pi}=\int \mathbb{P}_{\theta} \pi(d \theta)$ denote the mixture induced by $\pi$, where $\mathbb{P}_{\theta}$ is the distribution $N(\theta, \Sigma)$. 
Then, by Lemma 23 from~\cite{liu2021minimax}, we obtain
$$
1+\chi^{2}\left(\mathbb{P}_{\pi}, \mathbb{P}_{0}\right)=\mathbb{E}_{\theta, \theta^{\prime}}\left(\exp \left(\sum_{i=1}^{d} \frac{\theta_{i} \theta_{i}^{\prime}}{\sigma_{i}^{2}}\right)\right)=\mathbb E_{\theta, \theta^{\prime}}\left(\exp \left(c^{2} \rho^{2} \sigma_{I}^{-2} \mathbb{1}\left\{I=I^{\prime}\right\}\right)\right).
$$

Since $I \in\left\{1, \ldots, i^{*}\right\}$, we have $\sigma_{I} \geq \sigma_{i^{*}}$. Noting $\rho=\sigma_{i^{*}} \sqrt{\log \left(1+i^{*}\right)}$, we have
$$
\begin{aligned}
\mathbb{E}_{\theta, \theta^{\prime}}\left(\exp \left(c^{2} \rho^{2} \sigma_{I}^{-2} \mathbb{1}_{\left\{I=I^{\prime}\right\}}\right)\right) & \leq \mathbb{E}_{\theta, \theta^{\prime}}\left(\exp \left(c^{2} \log \left(1+i^{*}\right) \mathbb{1}\left\{I=I^{\prime}\right\}\right)\right) \\
& =\left(1-\frac{1}{i^{*}}+\frac{1}{i^{*}} e^{c^{2} \log \left(1+i^{*}\right)}\right)  \leq1+2c^{2}, 
\end{aligned}
$$
for some sufficiently small $c$.
Hence $\chi^{2}\left(\mathbb{P}_{\pi}, \mathbb{P}_{0}\right) \leq 2c^{2}$ can be made arbitrarily small.

\textit{Upper bound}. Note we can write $\psi=\max \left\{\psi_{+}, \psi_{-}\right\}$where
$$
\begin{aligned}
& \psi_{+}=\mathbb{1}\left\{\max _{i} X_{i}>C \rho\right\}, \\
& \psi_{-}=\mathbb{1}\left\{\max _{i}-X_{i}>C \rho\right\} .
\end{aligned}
$$

Therefore, an application of union bound yields $\mathbb{P}_{\theta}(\psi=1) \leq \mathbb{P}_{\theta}\left(\psi_{+}=1\right)+\mathbb{P}_{\theta}\left(\psi_{-}=1\right)$. Under the null hypothesis $\theta=0$, Lemma 2.3 of~\cite{van2017spectral} gives $\mathbb{E}_{0}\left(\max _{i} X_{i}\right)=\mathbb{E}_{0}\left(\max _{i}-X_{i}\right) \leq K \rho$ where $K>0$ is a universal constant. Markov's inequality yields
$$
\mathbb{P}_{0}\left(\psi_{+}=1\right)+\mathbb{P}_{0}\left(\psi_{-}=1\right) \leq \frac{2 K \rho}{C \rho}=\frac{2 K}{C}.
$$

The choice $C=(4 K/\eta)$ implies $\mathbb{P}_{0}(\psi=1) \leq \eta/2$.
Fix $\theta \in \mathbb{R}^{d}$ with $\|\theta\|_{\infty} \geq C^{\prime} \rho$. Note $\max _{i} \theta_{i} \geq C^{\prime} \rho$ or $\max _{i}-\theta_{i} \geq C^{\prime} \rho$. Suppose $\max _{i} \theta_{i} \geq C^{\prime} \rho$. Let $j^{*}=\arg \max _{j} \theta_{j}$. Writing $X_{i}=\theta_{i}+\xi_{i}$, consider
$$
\begin{aligned}
\mathbb{P}_{\theta}(\psi=0) & \leq \mathbb{P}_{\theta}\left(\psi_{+}=0\right)  =\mathbb{P}_{\theta}\left(\max _{i} X_{i} \leq C \rho\right)  \leq \mathbb{P}_{\theta}\left(X_{j^{*}} \leq C \rho\right) \\
& =\mathbb{P}_{\theta}\left(\theta_{j^{*}}+\xi_{j^{*}} \leq C \rho\right) \leq \mathbb{P}_{\theta}\left(\xi_{j^{*}} \leq\left(C-C^{\prime}\right) \rho\right) \\
& \leq \frac{\eta}{2}
\end{aligned}
$$
by choosing $C^{\prime}$ sufficiently large. The case $\max _{i}-\theta_{i} \geq C^{\prime} \rho$ is handled by essentially the same argument (examining $\psi_{-}$instead of $\left.\psi_{+}\right)$.

\section{Proof of examples}\label{app:proof_examples}

\subsection{Isotropic case}\label{subsec:proof_isotropic}

In this Subsection, assume that $\sigma_1 = \dots = \sigma_d = \sigma$. 
Assume that $t \geq 2$. We have by equation~\eqref{eq:beta_t>2}:
\begin{align*}
    s/2 = \frac{d \sigma^{t} \exp\left(-\beta/\sigma^2\right)}{\sqrt{d \sigma^{2t} \exp\left(-\beta/\sigma^2\right)}} = \sqrt{d}\exp\left(-\beta/2\sigma^2\right), \quad \text{ so that } \beta = 2 \sigma^2 \log\left(\frac{2\sqrt{d}}{s}\right).
\end{align*}

If $s \geq 2\sqrt{d}$, then $\lambda = 0$ so that $\nu^t = \sqrt{d} \sigma^t$ and $\epsilon^*(s,t,\sigma^2 I_d) \asymp \sigma d^{1/2t}$. 
Otherwise, by the definition of~$\nu$ in~\eqref{eq:implicit_separation_ellt_>2}:
$$ \nu^t = \sqrt{d \sigma^{2t} \exp\left(-\beta/\sigma^2\right) } = \sqrt{d} \sigma^t \frac{s}{2\sqrt{d}} = \sigma^t s/2,$$
so that $\epsilon^*(s,t,\sigma^2 I_d)^t \asymp \nu^t + \lambda^t s \asymp \sigma^t s \log^{t/2}\left(\sqrt{d}/s\right)$.
\vspace{2mm}

Now, assume that $t \leq 2$. If $s=d$, then from~\eqref{eq:def_lambda_ellt} we have $\lambda = 0$ and $\nu = \|\sigma\|_a$. 
Otherwise, we have $i_*=0$. If $s \geq \sqrt{d}$, we have $\lambda = \left(\sum_{j=1}^d \sigma^4\right)^{1/4}/\sqrt{s} = \sigma d^{1/4}/\sqrt{s}$ and $j_*=d$. 
Therefore, $\epsilon^*(s,t,\sigma^2 I_d)^t \asymp \nu^t + \lambda^t s \asymp \sigma^t d^{t/4} s^{1-t/2}$.
If $s < \sqrt{d}$, then $j_*=0$ and the analysis follows the same lines as in the case $t \geq 2:$ we have
\begin{align*}
    &s/2 = 
        \sum_{j=1}^d \frac{\sigma^t}{\nu^t}\exp \left(\frac{-\lambda^2}{\sigma^2}\right) 
        = 
        d\frac{\sigma^t}{\nu^t}\exp \left(\frac{-\lambda^2}{\sigma^2}\right) 
        \quad \text{ and } \quad 
    \nu^t= 
        \sqrt{\sum_{j=1}^d \sigma^{2t} \exp\left(\frac{-\lambda^2}{\sigma^2}\right)} 
        = 
        \sqrt{d} \sigma^{t} \exp\left(-\frac{\lambda^2}{2\sigma^2}\right)\\
        &\text{ i.e. } \quad s/2 = \sqrt{d}\exp\left(-\frac{\lambda^2}{2\sigma^2}\right)
        \quad
        \text{ hence }
        \quad
        \lambda^2 = 2 \sigma^2 \log\left(\frac{2\sqrt{d}}{s}\right).
\end{align*}

Therefore, we have $\epsilon^*(s,t,\sigma^2 I_d)^t \asymp \sigma^t s \log^{t/2}\left(2\sqrt{d}/s\right)$.

\subsection{Polynomially increasing variances}
Assume that $\alpha \geq 1$ and $t \geq 2$. 
We have 
\begin{align*}
    s/2 = \frac{\sum_{j=1}^d j^{\alpha t} \exp\left(-\lambda^2/\sigma_j^2\right)}{\sqrt{\sum_{j=1}^d j^{2\alpha t} \exp\left(-\lambda^2/\sigma_j^2\right)}} \geq \frac{\sum_{j\geq d/2} \left(\frac{d}{2}\right)^{\alpha t} \exp\left(-\lambda^2/d^{2\alpha}\right)}{\sqrt{ 2 \sum_{j \geq d/2} d^{2\alpha t} \exp\left(-\lambda^2/(d/2)^{2\alpha}\right)}} \asymp_\alpha \sqrt{d} \exp\left(-c \lambda^2/d^{2\alpha}\right).
\end{align*}
Similarly:
\begin{align*}
    s/2 = \frac{\sum_{j=1}^d j^{\alpha t} \exp\left(-\lambda^2/\sigma_j^2\right)}{\sqrt{\sum_{j=1}^d j^{2\alpha t} \exp\left(-\lambda^2/\sigma_j^2\right)}} \leq \frac{2\sum_{j\geq d/2} d^{\alpha t} \exp\left(-\lambda^2/(d/2)^{2\alpha}\right)}{\sqrt{\sum_{j \geq d/2} (d/2)^{2\alpha t} \exp\left(-\lambda^2/d^{2\alpha}\right)}} \asymp_\alpha \sqrt{d} \exp\left(-c' \lambda^2/d^{2\alpha}\right).
\end{align*}
Therefore, we have $\lambda \asymp_\alpha d^\alpha \sqrt{\log(d/s^2)}$ if $s\leq C \sqrt{d}$ and $\lambda = 0$ otherwise. 
Therefore, if $s \geq C \sqrt{d}$, then we have $\nu^{2t} = \sum_{j=1}^d j^{2\alpha} \asymp d^{2\alpha t + 1}$, so that $\epsilon^*(s,t,\Sigma) \asymp \nu \asymp d^\alpha d^{1/2t} = \sigma_{\max} d^{1/2t}$, where $\sigma_{\max} = \sigma_d = \max_j \sigma_j$. Otherwise,
\begin{align*}
    \nu^{2t} = \sum_{j=1}^d j^{2\alpha} \exp\left(-\lambda^2/j^{2\alpha}\right) \leq 2\sum_{j \geq d/2} d^{2\alpha t} \exp\left(-\lambda^2/d^{2\alpha}\right) \asymp_\alpha d^{2\alpha t +1 } \left(\frac{s^2}{d}\right) \leq d^{2\alpha t} s^2,
\end{align*}
therefore, $\epsilon^*(s,t,\Sigma) \asymp \lambda s^{1/t} \asymp d^{\alpha} s^{1/t} \sqrt{\log(d/s^2)}$.

\subsection{Exponentially decreasing variances}
Assume that $t \geq 2$. 
If $\alpha^{td} \geq \alpha^t/4$, then we are back to the isotropic case and $\epsilon^*(s,t,\Sigma) \asymp \epsilon^*(s,t,\alpha I_d)$. 
Otherwise, we have $\alpha^{d} < \alpha/4$. 
Let $j_0 = \min \left\{j ~\big|~ \alpha^j \leq \alpha/2\right\}$ and $j_1 = \min \left\{j ~\big|~ \alpha^j \leq \alpha/4\right\}$. 
Then,
\begin{align*}
    \sum_{j>j_1} \sigma_j^t \exp\left(-\beta/\sigma_j^2\right) \leq \alpha^{t (j_1+1)} \frac{1-\alpha^{t(d-j_1)}}{1-\alpha^t} \exp\left(-16 \beta/\alpha^2\right) 
    \leq  \frac{\alpha^t/4}{1-\alpha^t} \exp\left(-16 \beta/\alpha^2\right).
\end{align*}

Moreover, 
\begin{align*}
    \sum_{j < j_0} \sigma_j^t \exp\left(-\beta/\sigma_j^2\right) 
    \geq \sum_{j < j_0} \sigma_j^t \exp\left(-4\beta/\alpha^2\right) 
    =  \frac{\alpha^t - \alpha^{tj_0}}{1-\alpha^t} \exp\left(-4 \beta/\alpha^2\right)
    \geq  \frac{\alpha^t/2}{1-\alpha^t} \exp\left(-4 \beta/\alpha^2\right).
\end{align*}

Therefore, we always have $\sum\limits_{j=1}^d \sigma_j^t \exp\left(-\beta/\sigma_j^2\right) \asymp_t \sum\limits_{j\leq j_1} \sigma_j^t \exp\left(-\beta/\sigma_j^2\right)$. 
Proceeding similarly, we can also get $\sum\limits_{j=1}^d \sigma_j^{2t} \exp\left(-\beta/\sigma_j^2\right) \asymp_t \sum\limits_{j\leq j_1} \sigma_j^{2t} \exp\left(-\beta/\sigma_j^2\right)$. 
Now, for $j \leq j_1$, we have $\sigma_j^t \in [\alpha^t/4^t,\alpha^t]$, so that:
\begin{align*}
    s/2 \asymp \frac{\sum\limits_{j\leq j_1} \sigma_j^t \exp\left(-\beta/\sigma_j^2\right)}{\sqrt{\sum\limits_{j \leq j_1} \sigma_j^{2t} \exp\left(-\beta/\sigma_j^2\right)}} \asymp \sqrt{j_1} \exp\left(-\beta/C'\alpha^2\right), \quad \text{ hence } \beta = C'\alpha^2 \log\left(C\sqrt{j_1}/s\right),
\end{align*}
for some constants $C,C'$ depending only on $t$. 
Moreover, $\nu^{2t} \asymp \sum\limits_{j \leq j_1} \sigma_j^{2t} \exp\left(-\beta/\sigma_j^2\right)$. 
We exactly recover the analysis of the isotropic case from Subsection~\ref{subsec:proof_isotropic}.
In other words, it holds that $\epsilon^*(s,t,\Sigma) \asymp \epsilon^*(s,t,\alpha^2 I_{j_1})$, where by definition of $j_1$, we have $j_1 \asymp \log^{-1}(1/\alpha)$.

\subsection{Variances decreasing at least exponentially}
Assume that $t \geq 2$. 
If $\sigma_d \geq 1/4$, then we are back to the isotropic case and $\epsilon^*(s,t,\Sigma) \asymp \epsilon^*(s,t, I_d)$. 
Otherwise, we have $\sigma_d < 1/4$. 
Let $j_0 = \min \left\{j ~\big|~ \sigma_j < 1/2\right\}$. 
Note that we have $1 \leq j_0$ since $\sigma_0 = 1$.
We will use the following relations, which are due to the convexity of $\phi$, and to the fact that $\phi$ is increasing and satisfies $\phi(0)=0$:
\begin{align}
    &\phi(j_0 +\ell)  \geq \phi(j_0) + \ell \phi'(j_0), \quad \forall \ell \geq 0,\label{eq:phi_cvx_1}\\
    &\phi(j) \leq \frac{\phi(j_0)}{j_0} \!\cdot \! j, \quad \forall j \leq j_0,\label{eq:phi_cvx_2}\\
    &\phi'(j_0) \geq \frac{\phi(j_0)}{j_0}.\label{eq:phi_cvx_3}
\end{align}

The goal will be to prove that $\sum_{j=1}^d \sigma_j^t \exp\left(-\beta/\sigma_j^2\right) \asymp_t \sum_{j < j_0} \sigma_j^t \exp\left(-\beta/\sigma_j^2\right)$, 
in other words, that we can only restrict our analysis to $\{0,\dots,j_0\!-\!1\}$, which is the largest set where the variances can be considered ``almost constant'': $\forall j \in \{0,\dots,j_0\!-\!1\}: \sigma_j \in [1/4,1]$. 
To do so, our strategy is to justify the following relation, which clearly implies the latter one: 
\begin{align}
    \sum_{j\geq j_0} \sigma_j^t \exp\left(-\beta/\sigma_j^2\right) \lesssim \sum_{j \leq j_0} \sigma_j^t \exp\left(-\beta/\sigma_j^2\right). \label{eq:strategy}
\end{align}

To control the left hand side of~\eqref{eq:strategy}, we can write:
\begin{align*}
    \sum_{j \geq j_0} \sigma_j^t \exp\left(-\beta/\sigma_j^2\right) &= \sum_{j \geq j_0} \exp\left(-t\phi(j) -\beta/\sigma_j^2\right)\\
    &\leq \exp\left(-t\phi(j_0) - \beta/\sigma_{j_0}^2 \right)\sum_{\ell \geq0} \exp(-t \ell \phi'(j_0) ) \\
    & \leq \frac{\exp\left(-t\phi(j_0) - \beta/\sigma_{j_0}^2\right)}{1 - \exp(-t \phi'(j_0))}\\
    & < \frac{\exp\left(- \beta/\sigma_{j_0}^2\right)}{1 - \exp(-t \phi'(j_0))} \cdot 2^{-t}, \text{ by definition of } j_0. 
\end{align*}

To control the right-hand side of~\eqref{eq:strategy}, we can write
\begin{align*}
    \sum_{j \leq j_0} \sigma_j^t \exp\left(-\beta/\sigma_j^2\right) 
    &\geq \exp\left(-\beta/\sigma_{j_0}^2\right) \sum_{j \leq j_0} \exp\left(- t\frac{\phi(j_0)}{j_0} \!\cdot \! j\right) \\
    & = \exp\left(-\beta/\sigma_{j_0}^2\right) \frac{1 - \exp\left(- t\frac{\phi(j_0)}{j_0} \!\cdot \! j_0\right)}{1-\exp\left(- t\frac{\phi(j_0)}{j_0}\right)} \\
    &\geq \exp\left(-\beta/\sigma_{j_0}^2\right) \frac{1-2^{-t}}{1-\exp\left(- t\frac{\phi(j_0)}{j_0}\right)} \text{ by definition of } j_0\\
    & \geq  \exp\left(-\beta/\sigma_{j_0}^2\right) \frac{1-2^{-t}}{1-\exp\left(- t \phi'(j_0)\right)} \text{ by \eqref{eq:phi_cvx_3}}. 
\end{align*}

Therefore, the relation~\eqref{eq:strategy} is proven. 
Proceeding similarly, we can also get $\sum\limits_{j=1}^d \sigma_j^{2t} \exp\left(-\beta/\sigma_j^2\right) \asymp_t \sum\limits_{j\leq j_0} \sigma_j^{2t} \exp\left(-\beta/\sigma_j^2\right)$. 
Now, it follows that
\begin{align*}
    s/2 \asymp \frac{\sum\limits_{j\leq j_0} \sigma_j^t \exp\left(-\beta/\sigma_j^2\right)}{\sqrt{\sum\limits_{j \leq j_0} \sigma_j^{2t} \exp\left(-\beta/\sigma_j^2\right)}} \asymp \sqrt{j_0} \exp\left(-C'\beta \right), 
    \quad \text{ hence } \beta = C'\log\left(C\sqrt{j_0}/s\right),
\end{align*}
for some constants $C,C'$ depending only on $t$. 
Moreover, $\nu^{2t} \asymp \sum_{j \leq j_0} \sigma_j^{2t} \exp\left(-\beta/\sigma_j^2\right)$. 
We exactly recover the analysis of the isotropic case from Subsection~\ref{subsec:proof_isotropic}.
In other words, it holds that $\epsilon^*(s,t,\Sigma) \asymp \epsilon^*(s,t, I_{j_0})$, where by definition of $j_0$.
\section{Heuristic derivation of $T_{fdense}$ and $T_{inter}$}\label{appendix:derivation_weights}

In this Section, we explain how the optimization problem~\eqref{eq:optim_pb} can be leveraged to guess the weights in the chi-square type test statistics $T_{fdense}$ and $T_{inter}$, from~\eqref{eq:def_Tfdense_ellt} and~\eqref{eq:def_Tinter_ellt}, and in the test statistic  $T_{dense}$ when $t \geq 2$ from~\eqref{eq:def_test_stat_>2}. 
    In the dense case when $t\leq 2$, we look for test statistics $T_{fdense}$ and $T_{inter}$ that are of chi-squared type, meaning that they can be written in the following form
    $$T_{fdense} = \sum_{j\leq i_*}\omega_j (X_j^2 - \sigma_j^2) \quad \quad \text{ and } \quad\quad T_{inter} = \sum_{j = i_*+1}^{j_*} \omega_j (X_j^2 - \sigma_j^2),$$
    and the goal is to derive the weights $\omega_j$. 
    The intuition is that these test statistics are well-suited for playing against the optimal prior, whose parameters solve the optimization program (13):
    $$
    \max_{\gamma, \pi} ~~ \sum\limits_{j=1}^d \gamma_j^t \pi_j ~~ \text{ s.t. } ~~ \begin{cases}\sum_{j=1}^d \pi_j = s/2\\
    \pi_j \in [0,1] ~~ \forall j \in [d]\\
    \sum_{j=1}^d \pi_j^2 \sinh^2\left(\frac{\gamma_j^2}{2\sigma_j^2}\right) \leq c',\end{cases}$$
    
    By definition of the cut-off $j_*$, we have $\gamma_j \lesssim \sigma_j$ for any $j\leq j_*$. 
    Therefore, on the dense part $j \leq j_*$, the indistinguishability condition can be linearized as follows
    \begin{align}
        \sum_{j \leq j_*}\pi_j^2 \, \frac{\gamma_j^4}{\sigma_j^4} \leq c. \label{linearization}
    \end{align}
    
    Now, assume that we are given an observation from the prior $X\sim N(\theta, \Sigma)$ where $\theta_j = \pm b_j \gamma_j$ and $b_j \sim \operatorname{Ber}(\pi_j)$, and let us show that if $c$ is too large, the above test statistics reject the null hypothesis. 
    We can compute the expectation of our test statistic $T_{fdense}$ as follows (and $T_{inter}$ can be analyzed analogously):
    $$\mathbb E\left[T_{fdense}\right] = \sum_{j\leq i_*} \omega_j \left(\big(\pm b_j \gamma_j + \sigma_j \xi_j \big)^2 - \sigma_j^2\right) = \sum_{j\leq i_*} \omega_j \cdot \pi_j \gamma_j^2.$$
    
    By identification with~\eqref{linearization}, this suggests taking
    \begin{align}
        \omega_j \cdot \pi_j \gamma_j^2 = \pi_j^2 \,\frac{\gamma_j^4}{\sigma_j^4} \quad \quad \text{ i.e. } \quad\quad \omega_j = \pi_j \frac{\gamma_j^2}{\sigma_j^4}. \label{derivation_omega}
    \end{align}
    
    The expression of the optimal weights $\omega_j$'s is then given as follows.
    \begin{itemize}
        \item If $j \leq i_*$, then since the solution to the optimization problem ensures that $\pi_j = 1$ and $\gamma_j = c\frac{\sigma_j^{a/t}}{\nu^{a/t-1}}$, we obtain $\omega_j \propto \sigma^{-2b}$, where we recall that $a = 4t/(4-t)$ and $b= (4-2t)/(4-t)$.
        \item If $i_*<j\leq j_*$, then we also know that $\pi_j = \sigma_j^4/(\nu^t \lambda^{4-t})$ and $\gamma_j = c \lambda$ so that $\omega_j =c' \lambda^{2-t}$ is independent of $j$. In other words, reweighting is not helpful. 
    \end{itemize}
    
    To understand why taking $\omega_j$ as in~\eqref{derivation_omega} is a good idea, note that with this choice of $\omega_j$, the expectation of $T$ under the prior is $\mathbb E[T] \approx c$ (if the indistinguishability condition is saturated), whereas the variance under the prior can be computed as follows.
    \begin{align*}
        \mathbb V [T_{fdense}] &= \sum_{j \leq i_*} \omega_j^2 \left\{\mathbb E\left[(\pm b_j \gamma_j + \sigma_j \xi_j)^4\right] - \mathbb E^2\left[(\pm b_j \gamma_j + \sigma_j \xi_j)^2\right]\right\}\\
        & = \sum_{j\leq i_*} \omega_j^2 \left\{\pi_j \gamma_j^4 + 6 \pi_j \gamma_j^2 \sigma_j^2 + \sigma_j^4 - \left(\pi_j \gamma_j^2 + \sigma_j^2\right)^2\right\}\\
        & = \sum_{j\leq i_*} \omega_j^2 \left\{\pi_j(1-\pi_j) \gamma_j^4 + 4 \pi_j \gamma_j^2 \sigma_j^2 \right\}\\
        & \asymp \sum_{j\leq i_*} \omega_j^2 \pi_j \gamma_j^2 \sigma_j^2 \quad \text{ since $\gamma_j \lesssim \sigma_j$, by definition of $j_*$}\\
        & = \sum_{j\leq i_*} \pi_j^3 \frac{\gamma_j^6}{\sigma_j^6} = \bigg\|\Big(\pi_j \frac{\gamma_j^2}{\sigma_j^2}\Big)_{j\leq i_*}\bigg\|_3^3 \leq \bigg\|\Big(\pi_j \frac{\gamma_j^2}{\sigma_j^2}\Big)_{j\leq i_*}\bigg\|_2^3 \leq c^{3/2}.
    \end{align*}
    
    Now, if $c\ll 1$, then our lower bounds ensure that no test can distinguish between the prior and $H_0$. 
    Conversely, if $c\gg 1$, then $c^{3/2}\ll c^2 = \mathbb E^2[T_{fdense}]$ so that $T_{fdense}$ can distinguish between the two. 
    The same line of reasoning can also be applied to the test statistic $T_{dense}$ when $t\geq 2$. 
    Indeed, if we guess that $T_{dense}$ should be of the form
    $$T_{dense} = \sum_{j\leq j_*} \omega_j \left(|X_j|^t - \mathbb E_{H_0}|X_j|^t\right),$$
    then we can compute its expectation under the prior as follows. 
    Assume that we observe data point from the optimal prior $X \sim N(\theta, \Sigma)$ where $\theta_j = \pm b_j \gamma_j$ and $b_j \sim \operatorname{Ber}(\pi_j)$. 
    For $t \geq 2$, we always have $\gamma_j \asymp \sigma_j$ when $j \leq j_*$.
    Therefore, we get
    \begin{align*}
        \mathbb E T_{dense} &= \sum_{j\leq j_*} \pi_j \omega_j \mathbb E\left[|\gamma_j + \sigma_j \xi_j|^t - |\sigma_j \xi_j|^t \right] \\
        &= \pi_j \omega_j \mathbb E\left[|c \sigma_j + \sigma_j \xi_j|^t - |\sigma_j\xi_j|^t \right]\\
        & \asymp \pi_j \omega_j \sigma_j^t.
    \end{align*}
    
    At the last line, we used $\mathbb E\left[|N(c,1)|^t - |N(0,1)|^t\right] \asymp 1$ if $c\asymp 1$. 
    By identification with~\eqref{linearization}, this suggests taking
    $$\pi_j \omega_j \sigma_j^t \asymp \pi_j^2 \frac{\gamma_j^4}{\sigma_j^4} \Longleftrightarrow \omega_j \asymp \pi_j \sigma_j^{-t} \text{ since } \gamma_j \asymp \sigma_j.$$
    
    Recalling that $\pi_j = \sigma_j^t/\nu^t$ when $j \leq j_*$, we conclude that $\omega_j$ should be independent of $j$, in other words, that reweighting is suboptimal.

\bibliographystyle{alpha}
\newcommand{\etalchar}[1]{$^{#1}$}

\end{document}